\documentclass{article}
\title{Category Theory Using String Diagrams}
\author{Dan Marsden}

\usepackage[theorems]{macros}
\usepackage{diag}
\usepackage{names}
\usepackage{bussproofs}
\usepackage{natbib}
\usepackage{eqproof}

\newcommand{\leftop}{\ensuremath{\operatorname{\lhd}}}
\newcommand{\rightop}{\ensuremath{\operatorname{\rhd}}}

\newcommand{\mvright}[1]{\ensuremath{#1^{\rhd}}}
\newcommand{\mvrightright}[1]{\ensuremath{#1^{\rhd\rhd}}}
\newcommand{\mvleft}[1]{\ensuremath{#1^{\lhd}}}
\newcommand{\mvleftleft}[1]{\ensuremath{#1^{\lhd\lhd}}}

\newcommand{\adjsql}[1]{\ensuremath{#1_l}}
\newcommand{\adjsqr}[1]{\ensuremath{#1_r}}

\newcommand{\repcontrato}[4]{
\path (#1.center) ++(0,-0.5) coordinate (#1-r)
     +(-1,1) coordinate (#1-rl)
     +(1,1) coordinate (#1-rr);
\coordinate (#1-a) at (#1.south);
\path
 let \p1 = (#1.north) in
 let \p2 = (#1-rl) in
 let \p3 = (#1-rr) in
 coordinate (#1-b) at (\x2, \y1)
 coordinate (#1-c) at (\x3, \y1);
\fill[catset] (#1.south) -- (#1-r) to[out=0, in=270] (#1-rr) -- (#1-c) -- (#1.north east) -- (#1.south east) -- cycle;
\fill[catcop] (#1-b) -- (#1-rl) to[out=270,in=180] (#1-r) to[out=0,in=270] (#1-rr) -- (#1-c) -- cycle;
\fill[catterm] (#1.south) -- (#1-r) to[out=180, in=270] (#1-rl) -- (#1-b) -- (#1.north west) -- (#1.south west) -- cycle;
\draw (#1-r) to[out=180,in=270] (#1-rl) to node[swap]{#2} (#1-b);
\draw (#1-r) to[out=0, in=270] (#1-rr) to node[swap]{#3} (#1-c);
\draw (#1-a) to node[swap]{$*$} (#1-r);
\strnat{#1-r};
\strlabu{#1-r}{#4}
}
\newcommand{\repcontratoex}[5]{
\coordinate[label=below:#2] (#1-b) at (#1.south);
\coordinate[label=above:#5] (#1-t) at (#1.north);
 \draw (#1-b) -- ++(0,1.5)
       (#1-t) -- ++(0,-1.5);
 \path (#1.south west) -- ++(0.5,1) coordinate (#1-bl)
       (#1.north east) -- ++(-0.5,-1) coordinate (#1-tr);
 \node[rectangle, fit=(#1-bl)(#1-tr)] (#1-subdiag) {};
 \repcontrato{#1-subdiag}{#2}{#3}{#4};
 \draw[very thick] (#1-subdiag.south west) rectangle (#1-subdiag.north east);
\begin{pgfonlayer}{background}
 \fill[catterm] (#1.south west) rectangle (#1.north);
 \fill[catc] (#1.south east) rectangle (#1.north);
\end{pgfonlayer}
}

\begin{document}

\maketitle

\begin{abstract}
In \citep{Fokkinga1992, Fokkinga1992b} 
and \citep{FokkingaMeertens1994} a calculational approach to category theory
is developed. The scheme has many merits, but sacrifices useful type information
in the move to an equational style of reasoning. By contrast, traditional proofs
by diagram pasting retain the vital type information, but poorly express the reasoning
and development of categorical proofs. In order to combine the strengths of these
two perspectives, we propose the use of string diagrams, common folklore in the category
theory community, allowing us to retain the type information whilst pursuing a calculational
form of proof. These graphical representations provide a topological perspective on
categorical proofs, and silently handle functoriality and naturality conditions that
require awkward bookkeeping in more traditional notation.

Our approach is to proceed primarily by example, systematically applying graphical techniques to
many aspects of category theory. We develop string diagrammatic formulations
of many common notions, including adjunctions, monads, Kan extensions,
limits and colimits. We describe representable functors
graphically, and exploit these as a uniform source of graphical calculation rules
for many category theoretic concepts.
These graphical tools are then used to explicitly prove many standard
results in our proposed diagrammatic style.  
\end{abstract}

\section{Introduction}
\label{sec:introduction}
This work develops in some detail many aspects of basic category theory, 
with the aim of demonstrating the combined effectiveness of two key concepts:
\begin{enumerate}
 \item The calculational reasoning approach to mathematics.
 \item The use of string diagrams in category theory.
\end{enumerate}
The calculational style of mathematics is an approach to formal proofs originating in
the computer science community. The scheme is typically characterized by a goal directed systematic manipulation of
equations, similar in flavour to high school arithmetic.
The merits of this style of reasoning are discussed in \citep{vanGasteren1990, DijkstraScholten1990}
and discrete mathematics is formulated in a calculational style in \citep{GriesSchneider1993}.

String diagrams are a graphical formalism for working in category theory,
providing a convenient tool for describing morphisms within bicategories \citep{Benabou1967},
some background on string diagrams can be found in \citep{Street1995}.
Recent work in quantum computation and quantum foundations has successfully applied various diagrammatic
calculi based on string diagrams to model quantum 
systems in various forms of monoidal category, see for example
\citep{AbramskyCoecke2008, CoeckeDuncan2011, StayVicary2013, Vicary2012, BaezStay2011}
and the comprehensive survey \citep{Selinger2011} for more details. Further applications
of string diagrams for monoidal categories in more logical settings can be found in \citep{Mellies2006} and \citep{Mellies2012}.
Technical aspects of using string diagrams to reason in monoidal categories are detailed in
\citep{KellyLaplaza1980, JoyalStreet1991, JoyalStreet1995, JoyalStreet1988}.

Category theory was developed in a calculational style in \citep{Fokkinga1992, Fokkinga1992b} 
and \citep{FokkingaMeertens1994}. In these papers, we are presented with a choice of the traditional
commuting diagram style of reasoning about category theory, and a calculational approach based on formal
manipulation of the corresponding equations. We propose that string diagrams provide a third alternative
strategy, naturally supporting calculational reasoning, but continuing to carry the important type information
that is abandoned in the move to symbolic equations. Additionally, string diagram notation 
``does a lot of the work for free'',  an important aspect in choosing notation as advocated in \citep{Backhouse1989}. 
Issues of associativity, functoriality and naturality are handled silently by the notation, allowing
attention to be focused on the essential aspects of the proof.

To illustrate the difference between traditional ``diagram pasting'', the calculational approach
using symbolic equations as described in for example \citep{Fokkinga1992}, 
and the string diagrammatic calculus, we will investigate a simple example.

\begin{example}[Algebra Homomorphisms]
\label{ex:alghom}
Let $\mathcal{C}$ be a category, and $T : \mathcal{C} \rightarrow \mathcal{C}$ be an endofunctor on $\mathcal{C}$.
A \define{$T$-algebra} is a pair consisting of an object $X$ of $\mathcal{C}$ and a $\mathcal{C}$ morphism of type $T(X) \rightarrow X$.
A \define{$T$-algebra homomorphism} of type $(X, a) \rightarrow (Y,a')$ is a $\mathcal{C}$ morphism $h : X \rightarrow Y$ such that
the following diagram commutes:
\begin{equation}
\label{commuting:1}
\begin{gathered}
\begin{tikzpicture}[auto, node distance=2cm]
\node (TX) {$T(X)$}
      node[below of=TX] (X) {$X$}
      node[right of=TX] (TY) {$T(Y)$}
      node[below of=TY] (Y) {$Y$};
\draw [->] (TX) to node[swap]{$a$} (X);
\draw [->] (TY) to node{$a'$} (Y);
\draw [->] (TX) to node{$T(h)$} (TY);
\draw [->] (X) to node[swap]{$h$} (Y);
\end{tikzpicture}
\end{gathered}
\end{equation}
$T$-algebras and their homomorphisms form a category \algcat{T}.
Now we will show, using the three different proof styles discussed above, 
that for a category $\mathcal{D}$, a functor $F : \mathcal{C} \rightarrow \mathcal{D}$, an
endofunctor $T':\mathcal{D} \rightarrow \mathcal{D}$ and natural transformation $\alpha : T' \circ F \Rightarrow F \circ T$,
if $h : (X, a) \rightarrow (Y, a')$ is a $T$-algebra homomorphism then $F h$ is a $T'$-algebra homomorphism of type 
$(X, F(a) \circ \alpha_X) \rightarrow (Y, F(a') \circ \alpha_Y)$. 

Firstly we will prove this in the traditional diagram pasting
manner by noting that by naturality of $\alpha$ the following diagram commutes:
\begin{equation}
\label{commuting:2}
\begin{gathered}
\begin{tikzpicture}[auto, node distance=2cm]
\path node(T'X) {$T'F(X)$}
      node[below of=T'X] (TX) {$FT(X)$}
      node[right of=T'X, node distance=3cm] (T'Y) {$T'F(Y)$}
      node[below of=T'Y] (TY) {$FT(Y)$};
\draw [->] (T'X) to node{$T'F(h)$} (T'Y);
\draw [->] (TX) to node[swap]{$FT(h)$} (TY);
\draw [->] (T'X) to node[swap]{$\alpha_X$} (TX);
\draw [->] (T'Y) to node{$\alpha_Y$} (TY);
\end{tikzpicture}
\end{gathered}
\end{equation}
Also, as functor application preserves commuting diagrams, the following also commutes:
\begin{equation}
\label{commuting:3}
\begin{gathered}
\begin{tikzpicture}[auto, node distance=2cm]
\node (TX) {$FT(X)$}
      node[below of=TX] (X) {$F(X)$}
      node[right of=TX, node distance=3cm] (TY) {$FT(Y)$}
      node[below of=TY] (Y) {$F(Y)$};
\draw [->] (TX) to node[swap]{$F(a)$} (X);
\draw [->] (TY) to node{$F(a')$} (Y);
\draw [->] (TX) to node{$FT(h)$} (TY);
\draw [->] (X) to node[swap]{$F(h)$} (Y);
\end{tikzpicture}
\end{gathered}
\end{equation}
We can then paste diagrams \eqref{commuting:2} and \eqref{commuting:3} together, giving:
\begin{equation}
\label{commuting:4}
\begin{gathered}
\begin{tikzpicture}[auto, node distance=2cm]
\path node (T'X) {$T'F(X)$}
      node[below of=T'X] (TX) {$FT(X)$}
      node[below of=TX] (X) {$F(X)$}
      node[right of=T'X, node distance=3cm] (T'Y) {$T'F(Y)$}
      node[below of=T'Y] (TY) {$FT(Y)$}
      node[below of=TY] (Y) {$F(Y)$};
\draw [->] (TX) to node[swap]{$F(a)$} (X);
\draw [->] (TY) to node{$F(a')$} (Y);
\draw [->] (TX) to node{$FT(h)$} (TY);
\draw [->] (X) to node[swap]{$F(h)$} (Y);
\draw [->] (T'X) to node{$T'F(h)$} (T'Y);
\draw [->] (T'X) to node[swap]{$\alpha_X$} (TX);
\draw [->] (T'Y) to node{$\alpha_Y$} (TY);
\end{tikzpicture}
\end{gathered}
\end{equation}
The claim then follows from the outer rectangle of diagram \eqref{commuting:4}.

A proof of the same claim in a calculational style similar to \citep{Fokkinga1992b}
might proceed as follows:
\begin{eqproof*}
F(h) \circ F(a) \circ \alpha_X
\explain{functoriality}
F(h \circ a) \circ \alpha_X
\explain{$h$ is a $T$-algebra homomorphism}
F(a' \circ T(h)) \circ \alpha_X
\explain{functoriality}
F(a') \circ FT(h) \circ \alpha_X
\explain{naturality}
F(a') \circ \alpha_Y \circ T'F(h)
\end{eqproof*}
and this proves the claim.

Finally, we will sketch how such a proof will look using string diagrams, deliberately mirroring
the previous calculational style of proof, but now encoding the mathematical data graphically.
The precise details of this graphical approach will be developed in later sections, 
but hopefully the example is sufficiently straightforward for the underlying ideas to be apparent:
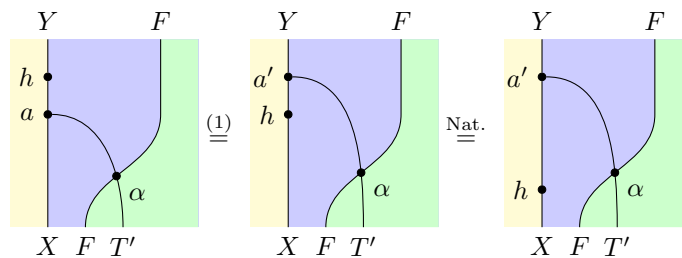
\begin{equation*}
\begin{gathered}
\begin{tikzpicture}[scale=0.5]
\path coordinate[dot, label=left:$h$] (h) +(0,1) coordinate[label=above:$Y$] (tl)
 ++(0,-1) coordinate[dot, label=left:$a$] (a) +(0,-3) coordinate[label=below:$X$] (bl)
 ++(2,-3) coordinate[label=below:$T'$] (br)
 (bl) ++(1,0) coordinate[label=below:$F$] (bm)
 (a) ++(3,0) coordinate (r) ++(0,2) coordinate[label=above:$F$] (tr);
\draw (bl) -- (a) -- (h) -- (tl);
\draw[name path=curv] (bm) to[out=90, in=-90] (r) -- (tr);
\draw[name path=vert] (a) to[out=0, in=90] (br);
\path[name intersections={of=vert and curv}] coordinate[dot, label=-45:$\alpha$] (alpha) at (intersection-1);
\begin{pgfonlayer}{background}
\fill[catterm] ($(tl) + (-1,0)$) rectangle (br);
\fill[catc] (bl) rectangle ($(tr) + (1,0)$);
\fill[catd] (bm) to[out=90, in=-90] (r) -- (tr) -- ++(1,0) -- ++(0,-5) -- cycle;
\end{pgfonlayer}
\end{tikzpicture}
\end{gathered}
\stackrel{\eqref{commuting:1}}{=}
\begin{gathered}
\begin{tikzpicture}[scale=0.5]
\path coordinate[dot, label=left:$a'$] (a') +(0,1) coordinate[label=above:$Y$] (tl)
 +(2,-4) coordinate[label=below:$T'$] (br)
 ++(0,-1) coordinate[dot, label=left:$h$] (h)
 ++(0,-3) coordinate[label=below:$X$] (bl)
 (bl) ++(1,0) coordinate[label=below:$F$] (bm)
 (h) ++(3,0) coordinate (r) ++(0,2) coordinate[label=above:$F$] (tr);
\draw (bl) -- (h) -- (a') -- (tl);
\draw[name path=vert] (br) to[out=90, in=0] (a');
\draw[name path=curv] (bm) to[out=90, in=-90] (r) -- (tr);
\path[name intersections={of=vert and curv}] coordinate[dot, label=-45:$\alpha$] (alpha) at (intersection-1);
\begin{pgfonlayer}{background}
\fill[catterm] ($(tl) + (-1,0)$) rectangle (bl);
\fill[catc] (bl) rectangle ($(tr) + (1,0)$);
\fill[catd] (bm) to[out=90, in=-90] (r) -- (tr) -- ++(1,0) -- ++(0,-5) -- cycle;
\end{pgfonlayer}
\end{tikzpicture}
\end{gathered}
\stackrel{\text{Nat.}}{=}
\begin{gathered}
\begin{tikzpicture}[scale=0.5]
\path coordinate[dot, label=left:$a'$] (a') +(0,1) coordinate[label=above:$Y$] (tl)
 +(2,-4) coordinate[label=below:$T'$] (br)
 ++(0,-3) coordinate[dot, label=left:$h$] (h)
 ++(0,-1) coordinate[label=below:$X$] (bl)
 (bl) ++(1,0) coordinate[label=below:$F$] (bm)
 (h) ++(3,2) coordinate (r) ++(0,2) coordinate[label=above:$F$] (tr);
\draw (bl) -- (h) -- (a') -- (tl);
\draw[name path=vert] (br) to[out=90, in=0] (a');
\draw[name path=curv] (bm) to[out=90, in=-90] (r) -- (tr);
\path[name intersections={of=vert and curv}] coordinate[dot, label=-45:$\alpha$] (alpha) at (intersection-1);
\begin{pgfonlayer}{background}
\fill[catterm] ($(tl) + (-1,0)$) rectangle (bl);
\fill[catc] (bl) rectangle ($(tr) + (1,0)$);
\fill[catd] (bm) to[out=90, in=-90] (r) -- (tr) -- ++(1,0) -- ++(0,-5) -- cycle;
\end{pgfonlayer}
\end{tikzpicture}
\end{gathered}
\end{equation*}
Note the topological nature of the proof, the first step allows us to slide $h$ through $a$ as $h$ is a homomorphism,
the second equality follows as naturality manifests itself
as sliding two morphisms past each other along the ``wires'' of the diagram.
\end{example}

Now we highlight some similarities and differences between the three styles of reasoning:
\begin{itemize}
 \item Using diagram pasting, a proof may often be presented as a single large and potentially complex commuting
  diagram from which the required equalities can be read off. Such a diagram efficiently encodes a large number
  of equalities, and all the associated type information. Unfortunately, this style of presentation does 
  not capture how the proof was
  constructed, it can be unclear in what order the reasoning proceeded, and constructing such a diagram is somewhat
  of an art form.
 \item The calculational approach presents a proof as a simple series of equalities and manipulations of equations.
  The sequence of steps and their motivation can be made very explicit, aiding both the development and later understanding
  of such a proof. The notation is also clearly more compact than the commuting diagrams, but at a significant cost as
  all the type information is lost in the conversion to equational form. This type information can help eliminate errors
  in reasoning, and often proofs in category theory are guided by the need to compose morphisms in a type correct manner.
 \item The string diagrammatic proof combines some of the benefits of both of the previous approaches. Proofs consist of
  equalities and manipulations of equations in a calculational style, but all the type information is retained,
  helping to guide and constrain the manipulations that can be formed. The approach is clearly more verbose than the
  traditional symbolic equations, but this additional verbosity carries useful type information, and tool support could
  reduce the cost of manipulating these more complex objects when developing a proof, as for example is done for a similar
  diagrammatic calculus used in quantum computing by the Quantomatic tool \citep{Quantomatic}. Also, using string diagram notation, 
  some aspects of proofs are ``handled by the notation'', for example the equational calculational proof above
  must explicitly invoke functoriality but in the graphical proof functoriality is implicit in the notation.
\end{itemize}

In this paper, we aim to provide an introductory, and reasonably self contained account 
of the use of string diagrams to reason calculationally about category theory.
We will start with the basic ideas of strings diagrams and then formulate
fundamental concepts of category theory such as adjunctions, monads, Kan extensions and limits and colimits
graphically. To illustrate the approach, these formulations will then be exploited to prove various standard results
in the proposed graphical style.
We will assume some basic familiarity with categorical concepts such as categories, functors and natural transformations,
and the elementary details of adjunctions, limits and colimits, for introductory accounts see for example 
\citep{AbrTze2011} or \citep{Pierce1991}, 
more comprehensive material can be found in the standard reference \citep{MacLane1998}.

As our aim to describe ``ordinary'' category theory, we will work exclusively with the 2-category of all categories, 
rather than working axiomatically with bicategories carrying sufficient structure. Many results are folklore or well
documented in the literature, although our heavy emphasis on providing a graphical formulation of these results in a form
suitable for calculation is certainly non-standard. In section \ref{sec:repr} we provide a new perspective on representable
functors as a source of graphical calculation rules. This perspective is then used to provide graphical rules for
Kan extensions, limits and colimits in later sections.

\section{An Outline of the Graphical Approach}
\label{sec:outline}
In the string diagram formalism, a category $\mathcal{C}$ is represented as a coloured region:
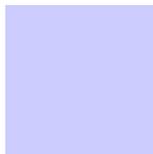
\begin{center}
\begin{tikzpicture}[scale=0.5]
\fill[catc] (0,0) rectangle (4,4);
\end{tikzpicture}
\end{center}
Different colours are used to denote different categories within a diagram. To avoid a great deal
of repetitive detail, in what follows we
will not in general specify which categories correspond to which colours.
This information, if needed, should be apparent from the other visual components within our diagrams.
A functor $F : \mathcal{C} \rightarrow \mathcal{D}$ is represented as an edge, commonly referred to as a wire, 
for example:
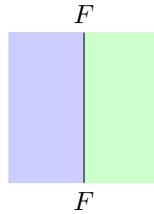
\begin{center}
\begin{tikzpicture}[scale=0.5]
\path coordinate[label=below:$F$] (b) ++(0,4) coordinate[label=above:$F$] (t);
\draw (b) -- (t);
\begin{pgfonlayer}{background}
\fill[catc] ($(t) + (-2,0)$) rectangle (b);
\fill[catd] (t) rectangle ($(b) + (2,0)$);
\end{pgfonlayer}
\end{tikzpicture}
\end{center}
Identity functors are often omitted and simply drawn as a region.
Given a second functor $G : \mathcal{D} \rightarrow \mathcal{E}$, the composite $G \circ F$ is
represented as:
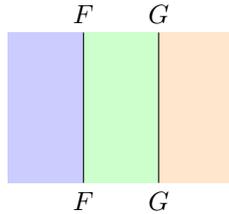
\begin{center}
\begin{tikzpicture}[scale=0.5]
\path coordinate[label=below:$F$] (bl) +(0,4) coordinate[label=above:$F$] (tl)
 ++(2,0) coordinate[label=below:$G$] (br) ++(0,4) coordinate[label=above:$G$] (tr);
\draw (bl) -- (tl)
      (br) -- (tr);
\begin{pgfonlayer}{background}
\fill[catc] ($(tl) + (-2,0)$) rectangle (bl);
\fill[catd] (tl) rectangle (br);
\fill[cate] (tr) rectangle ($(br) + (2,0)$);
\end{pgfonlayer}
\end{tikzpicture}
\end{center}
So \emph{functors compose from left to right}. 
Note how by representing identity functors as regions, composing
on the left or right with an identity functor ``does nothing'' to our diagram as we would expect.
A natural transformation $\alpha : F \Rightarrow F'$ is represented as:
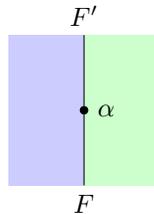
\begin{center}
\begin{tikzpicture}[scale=0.5]
\path coordinate[label=below:$F$] (b) ++(0,2) coordinate[dot, label=right:$\alpha$] (alpha) ++(0,2) coordinate[label=above:$F'$] (t);
\draw (b) -- (alpha) -- (t);
\begin{pgfonlayer}{background}
\fill[catc] ($(t) + (-2,0)$) rectangle (b);
\fill[catd] (t) rectangle ($(b) + (2,0)$);
\end{pgfonlayer}
\end{tikzpicture}
\end{center}
Given a second natural transformation $\alpha' : F' \Rightarrow F''$, the vertical composite $\alpha' \circ \alpha$ is
written as follows:
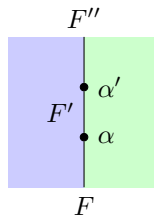
\begin{center}
\begin{tikzpicture}[scale=0.5]
\path coordinate[label=below:$F$] (b) ++(0,4) coordinate[label=above:$F''$] (t);
\coordinate[dot, label=right:$\alpha$] (alpha) at ($(b)!0.333!(t)$);
\coordinate[dot, label=right:$\alpha'$] (alpha') at ($(b)!0.666!(t)$);
\draw (b) -- (alpha) to node[left]{$F'$} (alpha') -- (t);
\begin{pgfonlayer}{background}
\fill[catc] ($(t) + (-2,0)$) rectangle (b);
\fill[catd] (t) rectangle ($(b) + (2,0)$);
\end{pgfonlayer}
\end{tikzpicture}
\end{center}
So \emph{natural transformations compose vertically from bottom to top}. Now if we also have a natural
transformation $\beta : G \Rightarrow G'$ we represent the horizontal composite $\beta * \alpha$ as:
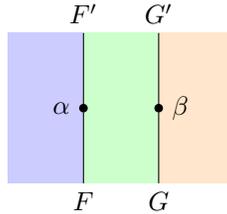
\begin{center}
\begin{tikzpicture}[scale=0.5]
\path coordinate[label=below:$F$] (bl) ++(0,2) coordinate[dot, label=left:$\alpha$] (alpha) ++(0,2) coordinate[label=above:$F'$] (tl)
 (bl) ++(2,0) coordinate[label=below:$G$] (br) ++(0,2) coordinate[dot, label=right:$\beta$] (beta) ++(0,2) coordinate[label=above:$G'$] (tr);
\draw (bl) -- (alpha) -- (tl)
      (br) -- (beta) -- (tr);
\begin{pgfonlayer}{background}
\fill[catc] ($(tl) + (-2,0)$) rectangle (bl);
\fill[catd] (tl) rectangle (br);
\fill[cate] (tr) rectangle ($(br) + (2,0)$);
\end{pgfonlayer}
\end{tikzpicture}
\end{center}
So \emph{natural transformations compose horizontally from left to right}. Identity natural transformations
are omitted, so $1 : F \Rightarrow F$ is drawn as:
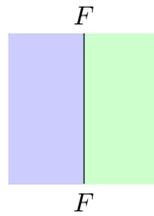
\begin{center}
\begin{tikzpicture}[scale=0.5]
\path coordinate[label=below:$F$] (b) ++(0,4) coordinate[label=above:$F$] (t);
\draw (b) -- (t);
\begin{pgfonlayer}{background}
\fill[catc] ($(t) + (-2,0)$) rectangle (b);
\fill[catd] (t) rectangle ($(b) + (2,0)$);
\end{pgfonlayer}
\end{tikzpicture}
\end{center}
In this way, vertical composition with identity natural transformations ``does nothing'' to our
diagram as we would expect.
Given another natural transformation $\beta' : G' \Rightarrow G''$, the \define{interchange law} holds:
\begin{equation*}
(\beta' \circ \beta) * (\alpha' \circ \alpha) = (\beta' * \alpha') \circ (\beta * \alpha)
\end{equation*}
and so the following composite is well defined, corresponding to a unique natural transformation:
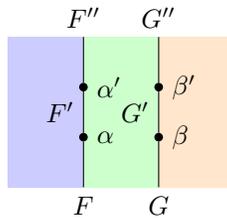
\begin{center}
\begin{tikzpicture}[scale=0.5]
\path coordinate[label=below:$F$] (bl) +(0,4) coordinate[label=above:$F''$] (tl)
 ++(2,0) coordinate[label=below:$G$] (br) ++(0,4) coordinate[label=above:$G''$] (tr);
\coordinate[dot, label=right:$\alpha$] (alpha) at ($(bl)!0.333!(tl)$);
\coordinate[dot, label=right:$\alpha'$] (alpha') at ($(bl)!0.666!(tl)$);
\coordinate[dot, label=right:$\beta$] (beta) at ($(br)!0.333!(tr)$);
\coordinate[dot, label=right:$\beta'$] (beta') at ($(br)!0.666!(tr)$);
\draw (bl) -- (alpha) to node[left]{$F'$} (alpha') -- (tl)
      (br) -- (beta) to node[left]{$G'$} (beta') -- (tr);
\begin{pgfonlayer}{background}
\fill[catc] ($(tl) + (-2,0)$) rectangle (bl);
\fill[catd] (tl) rectangle (br);
\fill[cate] (tr) rectangle ($(br) + (2,0)$);
\end{pgfonlayer}
\end{tikzpicture}
\end{center}
By substituting identities appropriately, we have the following \define{sliding equalities}:
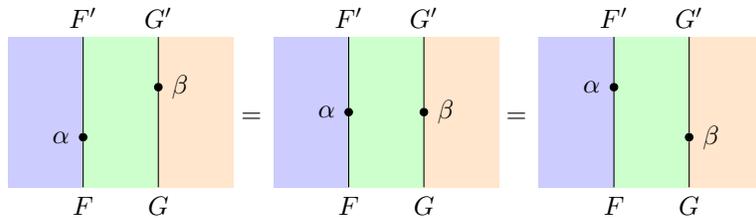
\begin{equation*}
\begin{gathered}
\begin{tikzpicture}[scale=0.5]
\path coordinate[label=below:$F$] (bl) +(0,4) coordinate[label=above:$F'$] (tl)
 ++(2,0) coordinate[label=below:$G$] (br) ++(0,4) coordinate[label=above:$G'$] (tr);
\coordinate[dot, label=left:$\alpha$] (alpha) at ($(bl)!0.333!(tl)$);
\coordinate[dot, label=right:$\beta$] (beta) at ($(br)!0.666!(tr)$);
\draw (bl) -- (alpha) -- (tl)
      (br) -- (beta) -- (tr);
\begin{pgfonlayer}{background}
\fill[catc] ($(tl) + (-2,0)$) rectangle (bl);
\fill[catd] (tl) rectangle (br);
\fill[cate] (tr) rectangle ($(br) + (2,0)$);
\end{pgfonlayer}
\end{tikzpicture}
\end{gathered}
=
\begin{gathered}
\begin{tikzpicture}[scale=0.5]
\path coordinate[label=below:$F$] (bl) +(0,4) coordinate[label=above:$F'$] (tl)
 ++(2,0) coordinate[label=below:$G$] (br) ++(0,4) coordinate[label=above:$G'$] (tr);
\coordinate[dot, label=left:$\alpha$] (alpha) at ($(bl)!0.5!(tl)$);
\coordinate[dot, label=right:$\beta$] (beta) at ($(br)!0.5!(tr)$);
\draw (bl) -- (alpha) -- (tl)
      (br) -- (beta) -- (tr);
\begin{pgfonlayer}{background}
\fill[catc] ($(tl) + (-2,0)$) rectangle (bl);
\fill[catd] (tl) rectangle (br);
\fill[cate] (tr) rectangle ($(br) + (2,0)$);
\end{pgfonlayer}
\end{tikzpicture}
\end{gathered}
=
\begin{gathered}
\begin{tikzpicture}[scale=0.5]
\path coordinate[label=below:$F$] (bl) +(0,4) coordinate[label=above:$F'$] (tl)
 ++(2,0) coordinate[label=below:$G$] (br) ++(0,4) coordinate[label=above:$G'$] (tr);
\coordinate[dot, label=left:$\alpha$] (alpha) at ($(bl)!0.666!(tl)$);
\coordinate[dot, label=right:$\beta$] (beta) at ($(br)!0.333!(tr)$);
\draw (bl) -- (alpha) -- (tl)
      (br) -- (beta) -- (tr);
\begin{pgfonlayer}{background}
\fill[catc] ($(tl) + (-2,0)$) rectangle (bl);
\fill[catd] (tl) rectangle (br);
\fill[cate] (tr) rectangle ($(br) + (2,0)$);
\end{pgfonlayer}
\end{tikzpicture}
\end{gathered}
\end{equation*}
For a similar graphical calculus in 
\citep{Dubuc2013}, this is memorably summarized as:
\begin{quotation}
This allows us to move cells up and down when there are no obstacles, as if they were elevators.
\end{quotation}
Much of what can be proved with the string diagram calculus flows from the fact we can 
slide natural transformations past each other like this, 
as we saw in example \ref{ex:alghom} in the introduction. 

\subsection{Some Elementary Techniques}

\subsubsection*{Objects and Morphisms}
\label{sec:objmor}
Often we will want to reason using categories, functors, natural transformations and ``ordinary'' objects and morphisms
within categories, as we did in example \ref{ex:alghom}. 
To achieve this, we consider the category $1$ with one object and one (identity) morphism.
Functors of type $1 \rightarrow \mathcal{C}$
can be identified with objects of the category $\mathcal{C}$. A morphism
$f : X \rightarrow Y$ is then a natural transformation between two of these functors from the terminal category.
If we identify the functors and the corresponding objects, then we can write $f$ in the obvious way as:
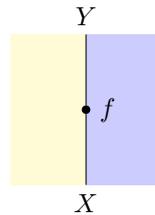
\begin{center}
\twocelldiag{f}{X}{Y}{catterm}{catc}{2}{2}{2}
\end{center}

\subsubsection*{Elements of Sets}
\label{sec:elements}
When we work in the category \refcset, given a set $X$, we may wish to refer to an element
$x \in X$. Extending the idea in section \ref{sec:objmor}, such an element can be identified 
with a morphism $x : * \rightarrow X$, where $*$ is the one element set
(the terminal object in \refcset). We then write this element as:
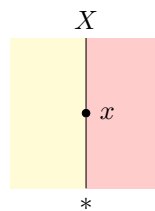
\begin{center}
\twocelldiag{x}{*}{X}{catterm}{catset}{2}{2}{2}
\end{center}

\subsubsection*{Naturality}
Often given a family of morphisms:
\begin{equation*}
(\alpha_X : F(X) \rightarrow G(X))_{X \in \obj{\mathcal{C}}}
\end{equation*}
 we wish to show that this family constitutes a natural transformation. 
This is equivalent to showing that for all $X, Y$ objects in $\mathcal{C}$ and all $f : X \rightarrow Y$
the following equality holds:
\begin{equation*}
\begin{gathered}
\begin{tikzpicture}[scale=0.5]
\path coordinate[label=below:$X$] (bl) ++(0,1) coordinate[dot, label=left:$\alpha_X$] (alphax) ++(0,1) coordinate[dot, label=left:$f$] (f) 
 ++(0,1) coordinate[label=above:$Y$] (tl)
 (bl) +(1,0) coordinate[label=below:$F$] (br)
 (tl) ++(1,0) coordinate[label=above:$G$] (tr) +(0,-1) coordinate (a);
\draw (bl) -- (alphax) -- (f) -- (tl)
 (br) to[out=90, in=0] (alphax) to[out=0, in=-90] (a) -- (tr);
\begin{pgfonlayer}{background}
\fill[catc] ($(tl) + (-2,0)$) rectangle ($(br) + (1,0)$);
\fill[catterm] ($(tl) + (-2,0)$) rectangle (bl);
\fill[catd] (br) to[out=90, in=0] (alphax) to[out=0, in=-90] (a) -- (tr) -- ($(tr) + (1,0)$) -- ($(br) + (1,0)$) -- cycle;
\end{pgfonlayer}
\end{tikzpicture}
\end{gathered}
=
\begin{gathered}
\begin{tikzpicture}[scale=0.5]
\path coordinate[label=below:$X$] (bl) ++(0,1) coordinate[dot, label=left:$f$] (f) ++(0,1) coordinate[dot, label=left:$\alpha_Y$] (alphay) 
 ++(0,1) coordinate[label=above:$Y$] (tl)
 (bl) ++(1,0) coordinate[label=below:$F$] (br) ++(0,1) coordinate (a)
 (tl) ++(1,0) coordinate[label=above:$G$] (tr);
\draw (bl) -- (f) -- (alphay) -- (tl)
 (br) -- (a) to[out=90, in=0] (alphay) to[out=0, in=-90] (tr);
\begin{pgfonlayer}{background}
\fill[catc] ($(tl) + (-2,0)$) rectangle ($(br) + (1,0)$);
\fill[catterm] ($(tl) + (-2,0)$) rectangle (bl);
\fill[catd] (br) -- (a) to[out=90, in=0] (alphay) to[out=0, in=-90] (tr) -- ($(tr) + (1,0)$) -- ($(br) + (1,0)$) -- cycle;
\end{pgfonlayer}
\end{tikzpicture}
\end{gathered}
\end{equation*}
Notice the similarity in the above condition with the sliding equations, effectively naturality says that
the natural transformation and function $f$ ``slide past each other'', and so we can draw them as two parallel
wires to illustrate this.

\section{Adjunctions}
\label{sec:adjunctions}
We will now investigate some of the properties of one of the central concepts of category theory, adjunctions.
As our focus is calculational proofs we will not provide concrete examples of adjunctions in the text, 
for many examples see \citep{MacLane1998, Borceux1994a, ACC2009} or any other good book on category theory.
Instead, we will formulate some of the key concepts graphically, and then illustrate the use of graphical techniques
in the proof of some standard results.

\subsection{Calculational Properties of Adjunctions}
Adjunctions can be described in a myriad of different and equivalent ways. 
In this section we adopt an approach that particularly suits our string diagrammatic presentation.  
In the various examples we will recover many of the other aspects of adjunctions.
A comprehensive discussion of the different formulations of the concept of an adjunction,
using an algebraic calculational style, is given in \citep{FokkingaMeertens1994}.

\begin{definition}[Adjunction]
An \define{adjunction} consists of functors and natural transformations:
\begin{center}
\begin{tabular}{c c c c}
\onecelldiag{F}{catc}{catd}{2}{2}
&
\onecelldiag{G}{catd}{catc}{2}{2}
&
\cupcell{\eta}{F}{G}{catc}{catd}
&
\capcell{\epsilon}{G}{F}{catd}{catc}
\end{tabular}
\end{center}
satisfying the following ``snake equations''
\begin{subequations}
\begin{trivlist}\vspace*{-\baselineskip}
\item
\begin{minipage}{0.495\textwidth}
\begin{equation}
\label{eq:snake-eta-epsilon}
\begin{gathered}
\begin{tikzpicture}[scale=0.5]
\path coordinate[dot, label=above:$\eta$] (eta) ++(1,1) coordinate (a) ++(1,1) coordinate[dot, label=below:$\epsilon$] (epsilon) 
 ++(1,-1) coordinate (b) ++(0,-2) coordinate[label=below:$F$] (br)
 (eta) ++(-1,1) coordinate (c) ++(0,2) coordinate[label=above:$F$] (tl);
\draw (tl) -- (c) to[out=-90, in=180] (eta) to[out=0, in=-90] (a) to[out=90, in=180] (epsilon) to[out=0, in=90] (b) -- (br);
\begin{pgfonlayer}{background}
\fill[catd] (tl) -- (c) to[out=-90, in=180] (eta) to[out=0, in=-90] (a) to[out=90, in=180] (epsilon) to[out=0, in=90] (b) -- (br)
 -- ++(0.5,0) -- ++(0,4) -- cycle;
\fill[catc] (tl) -- (c) to[out=-90, in=180] (eta) to[out=0, in=-90] (a) to[out=90, in=180] (epsilon) to[out=0, in=90] (b) -- (br)
 -- ++(-4.5,0) -- ++(0,4) -- cycle;
\end{pgfonlayer}
\end{tikzpicture}
\end{gathered}
\enskip=\enskip
\begin{gathered}
\begin{tikzpicture}[scale=0.5]
\path coordinate[label=below:$F$] (b) ++(0,4) coordinate[label=above:$F$] (t);
\draw (b) -- (t);
\begin{pgfonlayer}{background}
\fill[catc] ($(t) + (-1,0)$) rectangle (b);
\fill[catd] (t) rectangle ($(b) + (1,0)$);
\end{pgfonlayer}
\end{tikzpicture}
\end{gathered}
\end{equation}
\end{minipage}
\begin{minipage}{0.495\textwidth}
\begin{equation}
\label{eq:snake-epsilon-eta}
\begin{gathered}
\begin{tikzpicture}[scale=0.5]
\path coordinate[dot, label=below:$\epsilon$] (epsilon) ++(1,-1) coordinate (a) ++(1,-1) coordinate[dot, label=above:$\eta$] (eta)
 ++(1,1) coordinate (b) ++(0,2) coordinate[label=above:$G$] (tr)
 (epsilon) ++(-1,-1) coordinate (c) ++(0,-2) coordinate[label=below:$G$] (bl);
\draw (bl) -- (c) to[out=90, in=180] (epsilon) to[out=0, in=90] (a) to[out=-90, in=180] (eta) to[out=0, in=-90] (b) -- (tr);
\begin{pgfonlayer}{background}
\fill[catc] (bl) -- (c) to[out=90, in=180] (epsilon) to[out=0, in=90] (a) to[out=-90, in=180] (eta) to[out=0, in=-90] (b) -- (tr)
 -- ++(0.5,0) -- ++(0,-4) -- cycle;
\fill[catd] (bl) -- (c) to[out=90, in=180] (epsilon) to[out=0, in=90] (a) to[out=-90, in=180] (eta) to[out=0, in=-90] (b) -- (tr)
 -- ++(-4.5,0) -- ++(0,-4) -- cycle;
\end{pgfonlayer}
\end{tikzpicture}
\end{gathered}
\enskip=\enskip
\begin{gathered}
\begin{tikzpicture}[scale=0.5]
\path coordinate[label=below:$G$] (b) ++(0,4) coordinate[label=above:$G$] (t);
\draw (b) -- (t);
\begin{pgfonlayer}{background}
\fill[catd] ($(t) + (-1,0)$) rectangle (b);
\fill[catc] (t) rectangle ($(b) + (1,0)$);
\end{pgfonlayer}
\end{tikzpicture}
\end{gathered}
\end{equation}
\end{minipage}
\end{trivlist}
\end{subequations}
In this case we say that $F$ is a \define{left adjoint} for $G$, or $G$ is a \define{right adjoint} for $F$, written $F \dashv G$.
The natural transformations $\eta$ and $\epsilon$ are referred to as the \define{unit} and \define{counit} of the adjunction.
\end{definition}

\begin{lemma}[Adjunctions Compose]
\label{lem:adjcomp}
Two adjunctions $F : \mathcal{C} \rightarrow \mathcal{D} \dashv G : \mathcal{D} \rightarrow \mathcal{C}$
and $F' : \mathcal{D} \rightarrow \mathcal{E} \dashv G' : \mathcal{E} \rightarrow \mathcal{D}$ compose
to give an adjunction:
\begin{equation*}
F' \circ F \dashv G \circ G'
\end{equation*}
\end{lemma}
\begin{proof}
Take as unit and counit the composites:
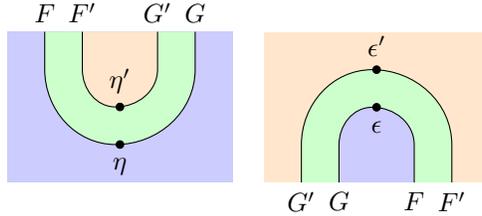
\begin{equation*}
\begin{gathered}
\begin{tikzpicture}[scale=0.5]
\path coordinate[dot, label=above:$\eta'$] 
 (eta') ++(-1,1) coordinate (a) ++(0,1) coordinate[label=above:$F'$] (tl)
 (eta') ++(1,1) coordinate (b) ++(0,1) coordinate[label=above:$G'$] (tr)
 (eta') ++(0,-1) coordinate[dot, label=below:$\eta$] 
 (eta) ++(-2,2) coordinate (a') ++(0,1) coordinate[label=above:$F$] (tll)
 (eta) ++(2,2) coordinate (b') ++(0,1) coordinate[label=above:$G$] (trr);
\draw (tl) -- (a) to[out=-90, in=180] (eta') to[out=0, in=-90] (b) -- (tr)
 (tll) -- (a') to[out=-90, in=180] (eta) to[out=0, in=-90] (b') -- (trr);
\begin{pgfonlayer}{background}
\fill[catc] ($(tll) + (-1,0)$) rectangle ($(trr) + (1,-4)$);
\fill[catd] (tll) -- (a') to[out=-90, in=180] (eta) to[out=0, in=-90] (b') -- (trr) -- cycle;
\fill[cate] (tl) -- (a) to[out=-90, in=180] (eta') to[out=0, in=-90] (b) -- (tr) -- cycle;
\end{pgfonlayer}
\end{tikzpicture}
\end{gathered}
\quad
\begin{gathered}
\begin{tikzpicture}[scale=0.5]
\path coordinate[dot, label=below:$\epsilon$] 
 (epsilon) ++(-1,-1) coordinate (a) ++(0,-1) coordinate[label=below:$G$] (bl)
 (epsilon) ++(1,-1) coordinate (b) ++(0,-1) coordinate[label=below:$F$] (br)
 (epsilon) ++(0,1) coordinate[dot, label=above:$\epsilon'$] 
 (epsilon') ++(-2,-2) coordinate (a') ++(0,-1) coordinate[label=below:$G'$] (bll)
 (epsilon') ++(2,-2) coordinate (b') ++(0,-1) coordinate[label=below:$F'$] (brr);
\draw (bl) -- (a) to[out=90, in=180] (epsilon) to[out=0, in=90] (b) -- (br)
 (bll) -- (a') to[out=90, in=180] (epsilon') to[out=0, in=90] (b') -- (brr);
\begin{pgfonlayer}{background}
\fill[cate] ($(bll) +(-1,0)$) rectangle ($(brr) + (1,4)$);
\fill[catd] (bll) -- (a') to[out=90, in=180] (epsilon') to[out=0, in=90] (b') -- (brr) -- cycle;
\fill[catc] (bl) -- (a) to[out=90, in=180] (epsilon) to[out=0, in=90] (b) -- (br) -- cycle;
\fill[white] ($(bll) +(-1,4)$) rectangle ($(brr) + (1,6)$);
\end{pgfonlayer}
\end{tikzpicture}
\end{gathered}
\end{equation*}
That these satisfy the adjunction axioms is then trivial to show.
\end{proof}
We now consider in some detail how the units and counits of a pair of adjunctions relate to
certain natural transformations. A lot of structure can be derived from the fact that units and
counits let us ``bend wires''.
\begin{definition}[Wire Bending]
\label{def:wirebending}
We consider a situation with adjunctions $F \dashv G$ and $F' \dashv G'$ 
and functors $H, K$, with types as in the following (not necessarily commuting) diagram:
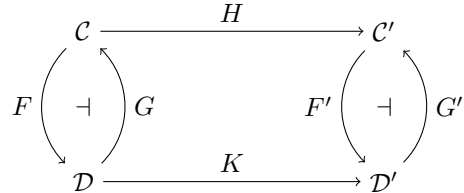
\begin{center}
\begin{tikzpicture}[auto]
\path node (tl) {$\mathcal{C}$}
      -- +(4,0) node (tr) {$\mathcal{C}'$}
      -- +(0,-2) node (bl) {$\mathcal{D}$}
      -- +(4,-2) node (br) {$\mathcal{D}'$};
\draw[->] (tl) to node{$H$} (tr);
\draw[->] (bl) to node{$K$} (br);
\node at ($(tl)!0.5!(bl)$) {$\dashv$};
\node at ($(tr)!0.5!(br)$) {$\dashv$};
\draw [->] (tl) to[out=-135, in=135] node[swap]{$F$} (bl);
\draw [->] (bl) to[out=45, in=-45] node[swap]{$G$} (tl);
\draw [->] (tr) to[out=-135, in=135] node[swap]{$F'$} (br);
\draw [->] (br) to[out=45, in=-45] node[swap]{$G'$} (tr);
\end{tikzpicture}
\end{center}
There are then bijections between the sets of natural transformations of the four types shown below:
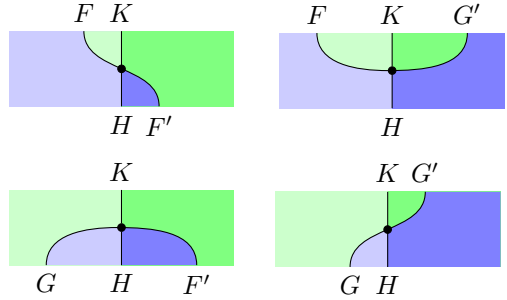
\begin{center}
\begin{tabular}{c c}
\drawnatd{}{$H$}{$K$}{$F$}{$F'$}{2} & \upfork{}{H}{K}{F}{G'}{1}\\
\downfork{}{H}{K}{G}{F'}{1} & \drawnatu{}{$H$}{$K$}{$G$}{$G'$}{2}
\end{tabular}
\end{center}

Graphically, these bijections are induced by ``bending wires'' using the adjunction axioms
to move between the different types.

Starting in the top left hand corner, between each pair of types we define a function \mvright{(-)}, 
referred to as ``move right'', in the clockwise direction,
with an inverse \mvleft{(-)}, referred to as ``move left'', in the counter clockwise direction.
Hopefully this overloading of names will not cause confusion, given the abundance of type information
in the string diagrams we are using. Firstly:
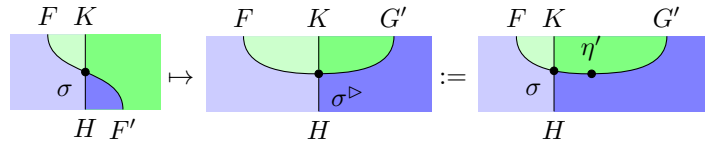
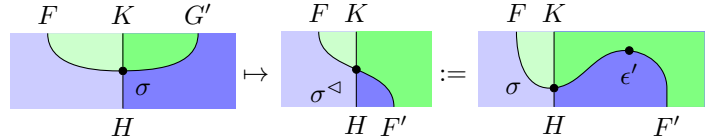
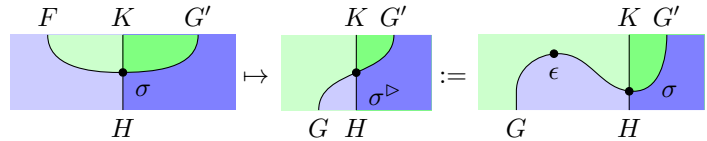
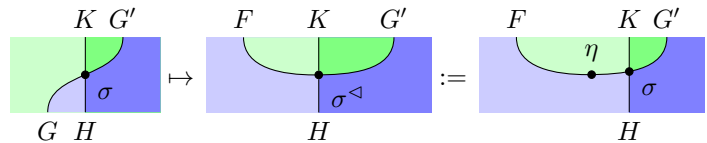
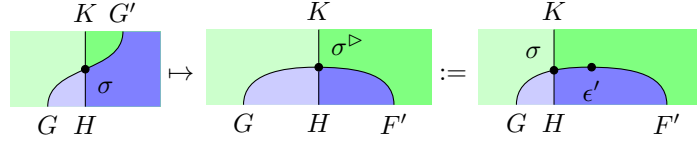
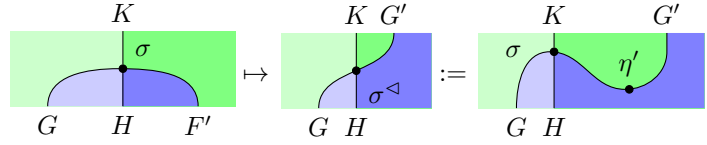
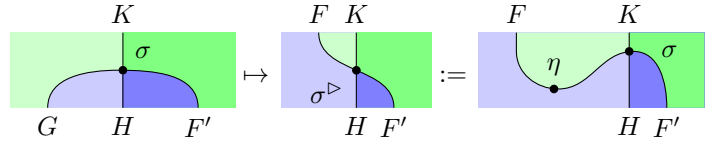
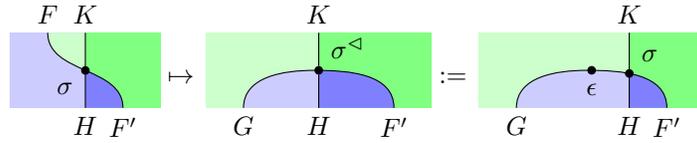
\begin{equation*}
\begin{gathered}
\drawnatd{$\sigma$}{$H$}{$K$}{$F$}{$F'$}{1}
\end{gathered}
\mapsto
\begin{gathered}
\upfork{\mvright{\sigma}}{H}{K}{F}{G'}{1}
\end{gathered}
:=
\begin{gathered}
\slidecex{$\sigma$}{$\eta'$}{F}{G'}{H}{K}
\end{gathered}
\end{equation*}
with inverse:
\begin{equation*}
\begin{gathered}
\upfork{\sigma}{H}{K}{F}{G'}{1}
\end{gathered}
\mapsto
\begin{gathered}
\drawnatd{$\mvleft{\sigma}$}{$H$}{$K$}{$F$}{$F'$}{1}
\end{gathered}
:=
\begin{gathered}
\upforkr{\sigma}{\epsilon'}{H}{K}{F}{F'}
\end{gathered}
\end{equation*}
Secondly:
\begin{equation*}
\begin{gathered}
\upfork{\sigma}{H}{K}{F}{G'}{1}
\end{gathered}
\mapsto
\begin{gathered}
\drawnatu{$\mvright{\sigma}$}{$H$}{$K$}{$G$}{$G'$}{1}
\end{gathered}
:=
\begin{gathered}
\upforkl{\sigma}{\epsilon}{H}{K}{G}{G'}
\end{gathered}
\end{equation*}
with inverse:
\begin{equation*}
\begin{gathered}
\drawnatu{$\sigma$}{$H$}{$K$}{$G$}{$G'$}{1}
\end{gathered}
\mapsto
\begin{gathered}
\upfork{\mvleft{\sigma}}{H}{K}{F}{G'}{1}
\end{gathered}
:=
\begin{gathered}
\slidedex{$\sigma$}{$\eta$}{F}{G'}{H}{K}
\end{gathered}
\end{equation*}
Thirdly:
\begin{equation*}
\begin{gathered}
\drawnatu{$\sigma$}{$H$}{$K$}{$G$}{$G'$}{1}
\end{gathered}
\mapsto
\begin{gathered}
\downfork{\mvright{\sigma}}{H}{K}{G}{F'}{1}
\end{gathered}
:=
\begin{gathered}
\slideaex{$\sigma$}{$\epsilon'$}{G}{F'}{H}{K}
\end{gathered}
\end{equation*}
with inverse:
\begin{equation*}
\begin{gathered}
\downfork{\sigma}{H}{K}{G}{F'}{1}
\end{gathered}
\mapsto
\begin{gathered}
\drawnatu{$\mvleft{\sigma}$}{$H$}{$K$}{$G$}{$G'$}{1}
\end{gathered}
:=
\begin{gathered}
\downforkr{\sigma}{\eta'}{H}{K}{G}{G'}
\end{gathered}
\end{equation*}
Finally:
\begin{equation*}
\begin{gathered}
\downfork{\sigma}{H}{K}{G}{F'}{1}
\end{gathered}
\mapsto
\begin{gathered}
\drawnatd{$\mvright{\sigma}$}{$H$}{$K$}{$F$}{$F'$}{1}
\end{gathered}
:=
\begin{gathered}
\downforkl{\sigma}{\eta}{H}{K}{F}{F'}
\end{gathered}
\end{equation*}
with inverse:
\begin{equation*}
\begin{gathered}
\drawnatd{$\sigma$}{$H$}{$K$}{$F$}{$F'$}{1}
\end{gathered}
\mapsto
\begin{gathered}
\downfork{\mvleft{\sigma}}{H}{K}{G}{F'}{1}
\end{gathered}
:=
\begin{gathered}
\slidebex{$\sigma$}{$\epsilon$}{G}{F'}{H}{K}
\end{gathered}
\end{equation*}
That each of these pairs of maps witness a bijection is easy to check by applying
the adjunction axioms.
\end{definition}
\begin{lemma}
\label{lem:mvnat}
All of the mappings:
\begin{align*}
\rhd &: [\mathcal{C}, \mathcal{D}'](F'H, KF) \rightarrow [\mathcal{C}, \mathcal{C}'](H, G'KF)\\
\rhd &: [\mathcal{C}, \mathcal{C}'](H, G'KF) \rightarrow [\mathcal{D}, \mathcal{C}'](HG, G'K)\\
\rhd &: [\mathcal{D}, \mathcal{C}'](HG, G'K) \rightarrow [\mathcal{D}, \mathcal{D}'](F'HG, K)\\
\rhd &: [\mathcal{D}, \mathcal{D}'](F'HG, K) \rightarrow [\mathcal{C}, \mathcal{D}'](F'H, KF)
\end{align*}
and their inverses are natural in both $H$ and $K$.
\end{lemma}
\begin{proof}
This is easy, if lengthy, to check graphically and left as an exercise.
\end{proof}
\begin{lemma}
Each of the composites $\leftop \leftop \leftop \leftop$ and $\rightop \rightop \rightop \rightop$ as described in definition \ref{def:wirebending}
are equal to the identity.
\end{lemma}
\begin{proof}
Straightforward from expanding definitions and exploiting the adjunction axioms.
\end{proof}
We will now exploit these wire bending ideas to provide a graphical proof of a standard
result about adjunctions.
\begin{lemma}
\label{lem:adjup}
For every morphism $f : X \rightarrow G(Y)$ there exists a unique morphism $\hat{f} : F(X) \rightarrow Y$
such that:
\begin{equation}
\label{eq:adjup}
\begin{gathered}
\bendfmorph{\hat{f}}{0.5}
\end{gathered}
=
\begin{gathered}
\fmorph{f}{G}{X}{Y}{catterm}{catc}{catd}{1.0}{1}
\end{gathered}
\end{equation}
\end{lemma}
\begin{proof}
Both existence and uniqueness follow from the following equivalences:
\begin{eqproof*}
\begin{gathered}
\bendfmorph{g}{0.5}
\end{gathered}
=
\begin{gathered}
\fmorph{f}{G}{X}{Y}{catterm}{catc}{catd}{1.0}{1}
\end{gathered}
\explain[\Leftrightarrow]{wire bending is a bijection}
\begin{gathered}
\begin{tikzpicture}[scale=0.5]
\path coordinate[dot, label=left:$g$] (g) ++(1,-1) coordinate[dot, label=above:$\eta$] (eta) ++(2,1) coordinate[dot, label=below:$\epsilon$] (epsilon)
 ++(1,-1) coordinate (a) ++(0,-0.5) coordinate[label=below:$F$] (br)
 (g) ++(0,-1.5) coordinate[label=below:$X$] (bl)
 (g) ++(0,0.5) coordinate[label=above:$Y$] (t);
\draw (bl) -- (g) -- (t)
 (g) to[out=-90, in=180] (eta) to[out=0, in=180] (epsilon) to[out=0, in=90] (a) -- (br);
\begin{pgfonlayer}{background}
\fill[catterm] ($(bl) + (-1,0)$) rectangle (t);
\fill[catd] (t) rectangle ($(br) + (1,0)$);
\fill[catc] (g) to[out=-90, in=180] (eta) to[out=0, in=180] (epsilon) to[out=0, in=90] (a) -- (br) -- (bl) -- cycle;
\end{pgfonlayer}
\end{tikzpicture}
\end{gathered}
=
\begin{gathered}
\bendgmorph{f}{0.5}
\end{gathered}
\explain[\Leftrightarrow]{adjunction axiom \eqref{eq:snake-eta-epsilon} }
\begin{gathered}
\gmorph{g}{F}{X}{Y}{catterm}{catc}{catd}{1}{1}
\end{gathered}
=
\begin{gathered}
\bendgmorph{f}{0.5}
\end{gathered}
\end{eqproof*}
\end{proof}
\begin{remark}
Although we proved lemma \ref{lem:adjup} directly, we could simply have observed that equation 
\eqref{eq:adjup}
is a special case of the first bijection in definition \ref{def:wirebending} with the 
leftmost adjunction in
the diagram being the trivial adjunction between the identity functors.
\end{remark}
We now see why the suggestive names ``move left'' and ``move right'' were chosen for the mapping in definition
\ref{def:wirebending}, as they allow us to ``slide'' a natural transformation back and forth along the bends
given by the units and counits of adjunctions, as described by the next lemma.
\begin{lemma}[Adjunction Sliding]
\label{lem:sliding}
We have the following equalities allowing us to ``slide'' a natural transformation along a pair
of adjunctions:
\begin{equation*}
\begin{gathered}
\slidecex{$\sigma$}{$\eta'$}{F}{G'}{H}{K}
\end{gathered}
=
\begin{gathered}
\upfork{$\mvright{\sigma}$}{H}{K}{F}{G'}{1}
\end{gathered}
=
\begin{gathered}
\slidedex{$\mvrightright{\sigma}$}{$\eta$}{F}{G'}{H}{K}
\end{gathered}
\end{equation*}
and
\begin{equation*}
\begin{gathered}
\slideaex{$\sigma$}{$\epsilon'$}{G}{F'}{H}{K}
\end{gathered}
=
\begin{gathered}
\downfork{$\mvright{\sigma}$}{H}{K}{G}{F'}{1}
\end{gathered}
=
\begin{gathered}
\slidebex{$\mvrightright{\sigma}$}{$\epsilon$}{G}{F'}{H}{K}
\end{gathered}
\end{equation*}
\end{lemma}
\begin{proof}
The equalities either follow immediately from the definitions, or require a single
application of one of the adjunction axioms.
\end{proof}
\begin{remark}
By combining lemma \ref{lem:sliding} with the fact that the various pairs of functions \mvleft{(-)} and \mvright{(-)}
witness a bijection allow us to derive several similar sliding identities, involving various combinations
of both \mvleft{(-)} and \mvright{(-)}.
\end{remark}

\begin{definition}[Mates Under an Adjunction]
Given an adjunction $F : \mathcal{C} \rightarrow \mathcal{D} \dashv G : \mathcal{D} \rightarrow \mathcal{C}$ 
and endofunctors $H : \mathcal{C} \rightarrow \mathcal{C}, K : \mathcal{D} \rightarrow \mathcal{D}$,
there is a bijection between the natural transformations
of types $HG \Rightarrow GK$ and $FH \Rightarrow KF$. This is easily seen using the operations in 
definition \ref{def:wirebending}, in the special case where both adjunctions are the same, with:
\begin{equation*}
\begin{gathered}
\catstyle{catc'}{\catccol}
\catstyle{catd'}{\catdcol}
\drawnatu{$\sigma$}{$H$}{$K$}{$G$}{$G$}{1}
\end{gathered}
\mapsto
\begin{gathered}
\catstyle{catc'}{\catccol}
\catstyle{catd'}{\catdcol}
\drawnatd{$\mvrightright{\sigma}$}{$H$}{$K$}{$F$}{$F$}{1}
\end{gathered}
\end{equation*}
and
\begin{equation*}
\begin{gathered}
\catstyle{catc'}{\catccol}
\catstyle{catd'}{\catdcol}
\drawnatd{$\sigma$}{$H$}{$K$}{$F$}{$F$}{1}
\end{gathered}
\mapsto
\begin{gathered}
\catstyle{catc'}{\catccol}
\catstyle{catd'}{\catdcol}
\drawnatu{$\mvleftleft{\sigma}$}{$H$}{$K$}{$G$}{$G$}{1}
\end{gathered}
\end{equation*}
Natural transformations related by these bijections are referred to as \define{mates under the adjunction} $F \dashv G$.
\end{definition}

\subsection{Adjoint Squares}
We will now apply our graphical approach to adjunctions in an investigation of some properties of adjoint squares.

\begin{definition}[Adjoint Square]
\label{def:adjsq}
An \define{adjoint square} is a pair of adjunctions and two functors $H, K$ arranged
as in the following (not necessarily commuting) diagram:
\begin{center}
\begin{tikzpicture}[auto]
\path node (tl) {$\mathcal{C}_1$}
      -- +(4,0) node (tr) {$\mathcal{C}_2$}
      -- +(0,-2) node (bl) {$\mathcal{D}_1$}
      -- +(4,-2) node (br) {$\mathcal{D}_2$};
\draw[->] (tl) to node{$H$} (tr);
\draw[->] (bl) to node{$K$} (br);
\node at ($(tl)!0.5!(bl)$) {$\dashv$};
\node at ($(tr)!0.5!(br)$) {$\dashv$};
\draw [->] (tl) to[out=-135, in=135] node[swap]{$F_1$} (bl);
\draw [->] (bl) to[out=45, in=-45] node[swap]{$G_1$} (tl);
\draw [->] (tr) to[out=-135, in=135] node[swap]{$F_2$} (br);
\draw [->] (br) to[out=45, in=-45] node[swap]{$G_2$} (tr);
\end{tikzpicture}
\end{center}
and also natural transformations:
\begin{center}
\begin{tabular}{c c}
\drawnatd{\adjsql{\sigma}}{$H$}{$K$}{$F_1$}{$F_2$}{1}
&
\drawnatu{\adjsqr{\sigma}}{$H$}{$K$}{$G_1$}{$G_2$}{1}
\end{tabular}
\end{center}
satisfying the following sliding equalities:
\begin{subequations}
\begin{equation}
\label{eq:adjsqa}
\begin{gathered}
\slidecex{\adjsql{\sigma}}{$\eta_2$}{F_1}{G_2}{H}{K}
\end{gathered}
=
\begin{gathered}
\slidedex{\adjsqr{\sigma}}{$\eta_1$}{F_1}{G_2}{H}{K}
\end{gathered}
\end{equation}
\begin{equation}
\label{eq:adjsqb}
\begin{gathered}
\slideaex{\adjsqr{\sigma}}{$\epsilon_2$}{G_1}{F_2}{H}{K}
\end{gathered}
=
\begin{gathered}
\slidebex{\adjsql{\sigma}}{$\epsilon_1$}{G_1}{F_2}{H}{K}
\end{gathered}
\end{equation}
\end{subequations}
Such a pair of natural transformations are said to be \define{conjugate}.
See \citep{MacLane1998} exercise IV.7.4.
\end{definition}

\begin{lemma}
\label{lem:adjsqcond}
For definitions as in definition \ref{def:adjsq},
the conditions in equations \eqref{eq:adjsqa} and \eqref{eq:adjsqb} are equivalent.
Furthermore, \adjsql{\sigma} and \adjsqr{\sigma} determine each other uniquely.
\end{lemma}
\begin{proof}
That the conditions are equivalent follows from:
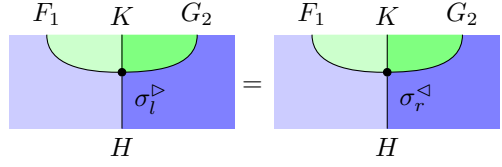
\begin{eqproof*}
\begin{gathered}
\slidecex{\adjsql{\sigma}}{$\eta_2$}{F_1}{G_2}{H}{K}
\end{gathered}
=
\begin{gathered}
\slidedex{\adjsqr{\sigma}}{$\eta_1$}{F_1}{G_2}{H}{K}
\end{gathered}
\explain[\Leftrightarrow]{definitions}
\begin{gathered}
\upfork{\mvright{\adjsql{\sigma}}}{H}{K}{F_1}{G_2}{1.5}
\end{gathered}
=
\begin{gathered}
\upfork{\mvleft{\adjsqr{\sigma}}}{H}{K}{F_1}{G_2}{1.5}
\end{gathered}
\explain[\Leftrightarrow]{applying the $\leftop$ wire bending bijection twice}
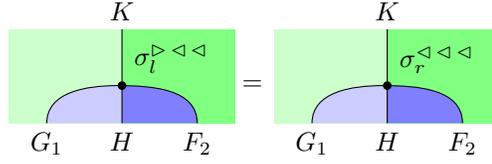
\begin{gathered}
\downfork{\adjsql{\sigma}^{\rightop \leftop \leftop}}{H}{K}{G_1}{F_2}{1.5}
\end{gathered}
=
\begin{gathered}
\downfork{\adjsqr{\sigma}^{\leftop \leftop \leftop}}{H}{K}{G_1}{F_2}{1.5}
\end{gathered}
\explain[\Leftrightarrow]{$\leftop$ and $\rightop$ are inverses, and $\leftop \leftop \leftop \leftop$ is the identity }
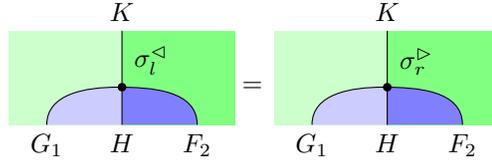
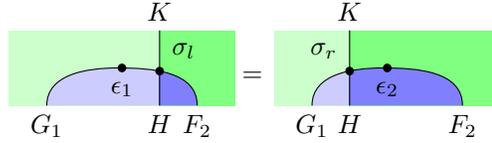
\begin{gathered}
\downfork{\mvleft{\adjsql{\sigma}}}{H}{K}{G_1}{F_2}{1.5}
\end{gathered}
=
\begin{gathered}
\downfork{\mvright{\adjsqr{\sigma}}}{H}{K}{G_1}{F_2}{1.5}
\end{gathered}
\explain[\Leftrightarrow]{definitions}
\begin{gathered}
\slidebex{\adjsql{\sigma}}{$\epsilon_1$}{G_1}{F_2}{H}{K}
\end{gathered}
=
\begin{gathered}
\slideaex{\adjsqr{\sigma}}{$\epsilon_2$}{G_1}{F_2}{H}{K}
\end{gathered}
\end{eqproof*}
That \adjsql{\sigma} and \adjsqr{\tau} determine each other uniquely is then immediate as they are
related by bijections from definition \ref{def:wirebending}.
\end{proof}

\begin{proposition}
\label{prop:adjsqcomp}
Adjoint squares compose vertically and horizontally.
\end{proposition}
\begin{proof}
By lemma \ref{lem:adjsqcond} we will only need to check one of the two equations \eqref{eq:adjsqa} and \eqref{eq:adjsqb} holds 
in order to ensure our composition operations preserve conjugate pairs.
The vertical composition of adjunctions will be assumed to be as in lemma \ref{lem:adjcomp}.
We will denote the vertical composition operation as $\circ$, defined as follows:
\begin{align*}
(
\begin{gathered}
\catstyle{catc}{\catdcol}
\catstyle{catc'}{\catdpcol}
\catstyle{catd}{\catecol}
\catstyle{catd'}{\catepcol}
\drawnatd{\adjsql{\tau}}{$J$}{$K$}{$F_3$}{$F_4$}{1}
\end{gathered}
,
\begin{gathered}
\catstyle{catc}{\catdcol}
\catstyle{catc'}{\catdpcol}
\catstyle{catd}{\catecol}
\catstyle{catd'}{\catepcol}
\drawnatu{\adjsqr{\tau}}{$J$}{$K$}{$G_3$}{$G_4$}{1}
\end{gathered}
)
&\circ 
(
\begin{gathered}
\drawnatd{\adjsql{\sigma}}{$H$}{$J$}{$F_1$}{$F_2$}{1}
\end{gathered}
,
\begin{gathered}
\drawnatu{\adjsqr{\sigma}}{$H$}{$J$}{$G_1$}{$G_2$}{1}
\end{gathered}
)
\\
:=
(
\begin{gathered}
\begin{tikzpicture}[scale=0.5]
\path coordinate[label=above:$K$] (t) ++(0,-4) coordinate[label=below:$H$] (b)
 (t) ++(-1,0) coordinate[label=above:$F_3$] (tl) ++(-1,0) coordinate[label=above:$F_1$] (tll)
 (b) ++(1,0) coordinate[label=below:$F_2$] (br) ++(1,0) coordinate[label=below:$F_4$] (brr);
\draw[name path=vert] (b) -- (t);
\draw[name path=curv1] (br) to[out=90, in=-90] (tll);
\draw[name path=curv2] (brr) to[out=90, in=-90] (tl);
\path[name intersections={of=vert and curv1}] coordinate[dot, label=-135:\adjsql{\sigma}] (sigma) at (intersection-1);
\path[name intersections={of=vert and curv2}] coordinate[dot, label=45:\adjsql{\tau}] (sigma') at (intersection-1);
\begin{pgfonlayer}{background}
\fill[catc] ($(tll) + (-1,0)$) rectangle (b);
\fill[catc'] ($(brr) + (1,0)$) rectangle (t);
\begin{scope}
\clip ($(tll) + (-1,0)$) rectangle (b);
\fill[catd] (br) to[out=90, in=-90] (tll) -- ++(5,0) -- ++(0,-4) -- cycle;
\fill[cate] (brr) to[out=90, in=-90] (tl) -- ++(4,0) -- ++(0,-4) -- cycle;
\end{scope}
\begin{scope}
\clip ($(brr) + (1,0)$) rectangle (t);
\fill[catd'] (br) to[out=90, in=-90] (tll) -- ++(5,0) -- ++(0,-4) -- cycle;
\fill[cate'] (brr) to[out=90, in=-90] (tl) -- ++(4,0) -- ++(0,-4) -- cycle;
\end{scope}
\end{pgfonlayer}
\end{tikzpicture}
\end{gathered}
&,
\begin{gathered}
\begin{tikzpicture}[scale=0.5]
\path coordinate[label=above:$K$] (t) ++(0,-4) coordinate[label=below:$H$] (b)
 (t) ++(1,0) coordinate[label=above:$G_4$] (tr) ++(1,0) coordinate[label=above:$G_2$] (trr)
 (b) ++(-1,0) coordinate[label=below:$G_1$] (bl) ++(-1,0) coordinate[label=below:$G_3$] (bll);
\draw[name path=vert] (b) -- (t);
\draw[name path=curv1] (bll) to[out=90, in=-90] (tr);
\draw[name path=curv2] (bl) to[out=90, in=-90] (trr);
\path[name intersections={of=vert and curv1}] coordinate[dot, label=135:\adjsql{\tau}] (tau') at (intersection-1);
\path[name intersections={of=vert and curv2}] coordinate[dot, label=-45:\adjsqr{\sigma}] (tau) at (intersection-1);
\begin{pgfonlayer}{background}
\fill[cate] ($(bll) + (-1,-0)$) rectangle (t);
\fill[cate'] ($(trr) + (1,0)$) rectangle (b);
\begin{scope}
\clip ($(bll) + (-1,-0)$) rectangle (t);
\fill[catd] (bll) to[out=90, in=-90] (tr) -- ++(2,0) -- ++(0,-4) -- cycle;
\fill[catc] (bl) to[out=90, in=-90] (trr) -- ++(1,0) -- ++(0,-4) -- cycle;
\end{scope}
\begin{scope}
\clip ($(trr) + (1,0)$) rectangle (b);
\fill[catd'] (bll) to[out=90, in=-90] (tr) -- ++(2,0) -- ++(0,-4) -- cycle;
\fill[catc'] (bl) to[out=90, in=-90] (trr) -- ++(1,0) -- ++(0,-4) -- cycle;
\end{scope}
\end{pgfonlayer}
\end{tikzpicture}
\end{gathered}
)
\end{align*}
That the resulting pair of natural transformations are conjugate is trivial as the following
equality holds because the component natural transformations are conjugate:
\begin{equation*}
\begin{gathered}
\begin{tikzpicture}[scale=0.5]
\path coordinate[dot, label=45:$\eta_4$] (eta') 
 +(-2,1) coordinate[label=above:$F_3$] (tl)
 +(2,1) coordinate[label=above:$G_4$] (tr)
 (eta') ++(0,-1) coordinate[label=below:$\eta_2$] (eta)
 (tl) ++(-1,0) coordinate[label=above:$F_1$] (tll)
 (tr) ++(1,0) coordinate[label=above:$G_2$] (trr)
 (tl) ++(1,0) coordinate[label=above:$K$] (t) ++(0,-3) coordinate[label=below:$H$] (b);
\draw[name path=curv] (tll) to[out=-90, in=180] (eta) to[out=0, in=-90] (trr);
\draw[name path=curv'] (tl) to[out=-90, in=180] (eta') to[out=0, in=-90] (tr);
\draw[name path=vert] (b) -- (t);
\path[name intersections={of=vert and curv}] coordinate[dot, label=-135:\adjsql{\sigma}] (sigma) at (intersection-1);
\path[name intersections={of=vert and curv'}] coordinate[dot, label=45:\adjsql{\tau}] (sigma') at (intersection-1);
\begin{pgfonlayer}{background}
\fill[catc] ($(tll) + (-1,0)$) rectangle (b);
\fill[catc'] ($(trr) + (1,0)$) rectangle (b);
\begin{scope}
\clip ($(tll) + (-1,0)$) rectangle (b);
\fill[catd] (tll) to[out=-90, in=180] (eta) to[out=0, in=-90] (trr) -- cycle;
\fill[cate] (tl) to[out=-90, in=180] (eta') to[out=0, in=-90] (tr) -- cycle;
\end{scope}
\begin{scope}
\clip ($(trr) + (1,0)$) rectangle (b);
\fill[catd'] (tll) to[out=-90, in=180] (eta) to[out=0, in=-90] (trr) -- cycle;
\fill[cate'] (tl) to[out=-90, in=180] (eta') to[out=0, in=-90] (tr) -- cycle;
\end{scope} 
\end{pgfonlayer}
\end{tikzpicture}
\end{gathered}
=
\begin{gathered}
\begin{tikzpicture}[scale=0.5]
\path coordinate[dot, label=135:$\eta_2$] (eta') 
 +(-2,1) coordinate[label=above:$F_3$] (tl)
 +(2,1) coordinate[label=above:$G_4$] (tr)
 (eta') ++(0,-1) coordinate[label=below:$\eta_1$] (eta)
 (tl) ++(-1,0) coordinate[label=above:$F_1$] (tll)
 (tr) ++(1,0) coordinate[label=above:$G_2$] (trr)
 (tr) ++(-1,0) coordinate[label=above:$K$] (t) ++(0,-3) coordinate[label=below:$H$] (b);
\draw[name path=curv] (tll) to[out=-90, in=180] (eta) to[out=0, in=-90] (trr);
\draw[name path=curv'] (tl) to[out=-90, in=180] (eta') to[out=0, in=-90] (tr);
\draw[name path=vert] (b) -- (t);
\path[name intersections={of=vert and curv}] coordinate[dot, label=-45:\adjsqr{\sigma}] (tau) at (intersection-1);
\path[name intersections={of=vert and curv'}] coordinate[dot, label=135:\adjsqr{\tau}] (tau') at (intersection-1);
\begin{pgfonlayer}{background}
\fill[catc] ($(tll) + (-1,0)$) rectangle (b);
\fill[catc'] ($(trr) + (1,0)$) rectangle (b);
\begin{scope}
\clip ($(tll) + (-1,0)$) rectangle (b);
\fill[catd] (tll) to[out=-90, in=180] (eta) to[out=0, in=-90] (trr) -- cycle;
\fill[cate] (tl) to[out=-90, in=180] (eta') to[out=0, in=-90] (tr) -- cycle;
\end{scope}
\begin{scope}
\clip ($(trr) + (1,0)$) rectangle (b);
\fill[catd'] (tll) to[out=-90, in=180] (eta) to[out=0, in=-90] (trr) -- cycle;
\fill[cate'] (tl) to[out=-90, in=180] (eta') to[out=0, in=-90] (tr) -- cycle;
\end{scope} 
\end{pgfonlayer}
\end{tikzpicture}
\end{gathered}
\end{equation*}
We will denote the horizontal composition operation $*$, defined as follows:
\begin{align*}
(
\begin{gathered}
\catstyle{catc}{\catcpcol}
\catstyle{catc'}{\catcppcol}
\catstyle{catd}{\catdpcol}
\catstyle{catd'}{\catdppcol}
\drawnatd{\adjsql{\rho}}{$K$}{$L$}{$F_2$}{$F_3$}{1}
\end{gathered}
,
\begin{gathered}
\catstyle{catc}{\catcpcol}
\catstyle{catc'}{\catcppcol}
\catstyle{catd}{\catdpcol}
\catstyle{catd'}{\catdppcol}
\drawnatu{\adjsqr{\rho}}{$K$}{$L$}{$G_2$}{$G_3$}{1}
\end{gathered}
)
&* 
(
\begin{gathered}
\drawnatd{\adjsql{\sigma}}{$H$}{$J$}{$F_1$}{$F_2$}{1}
\end{gathered}
,
\begin{gathered}
\drawnatu{\adjsqr{\sigma}}{$H$}{$J$}{$G_1$}{$G_2$}{1}
\end{gathered}
)
\\
:=
(
\begin{gathered}
\begin{tikzpicture}[auto, x=0.5cm, y=0.5cm]
\path coordinate[label=above:$F_1$] (tl) ++(3,-3) coordinate[label=below:$F_3$] (br)
 (tl) ++(1,0) coordinate[label=above:$J$] (tm) ++(1,0) coordinate[label=above:$L$] (tr)
 (br) ++(-1,0) coordinate[label=below:$K$] (bm) ++(-1,0) coordinate[label=below:$H$] (bl);
\draw[name path=curv] (br) to[out=90, in=-90] (tl);
\draw[name path=vert1] (bm) -- (tr);
\draw[name path=vert2] (bl) -- (tm);
\path[name intersections={of=vert1 and curv}] coordinate[dot, label=45:\adjsql{\rho}] (sigma') at (intersection-1);
\path[name intersections={of=vert2 and curv}] coordinate[dot, label=-135:\adjsql{\sigma}] (sigma) at (intersection-1);
\begin{pgfonlayer}{background}
\fill[catc] ($(tl) + (-1,0)$) rectangle (bl);
\fill[catc'] (tm) rectangle (bm);
\fill[catc''] (tr) rectangle ($(br) + (1,0)$);
\begin{scope}
\clip ($(tl) + (-1,0)$) rectangle (bl);
\fill[catd] (br) to[out=90, in=-90] (tl) -- ++(4,0) -- ++(0,-3) -- cycle;
\end{scope}
\begin{scope}
\clip (tm) rectangle (bm);
\fill[catd'] (br) to[out=90, in=-90] (tl) -- ++(4,0) -- ++(0,-3) -- cycle;
\end{scope}
\begin{scope}
\clip (tr) rectangle ($(br) + (1,0)$);
\fill[catd''] (br) to[out=90, in=-90] (tl) -- ++(4,0) -- ++(0,-3) -- cycle;
\end{scope}
\end{pgfonlayer}
\end{tikzpicture}
\end{gathered}
&,
\begin{gathered}
\begin{tikzpicture}[scale=0.5]
\path coordinate[label=below:$G_1$] (bl) ++(1,0) coordinate[label=below:$H$] (bm) ++(1,0) coordinate[label=below:$K$] (br)
 (bl) ++(3,3) coordinate[label=above:$G_3$] (tr) ++(-1,0) coordinate[label=above:$L$] (tm) ++(-1,0) coordinate[label=above:$J$] (tl);
\draw[name path=curv] (bl) to[out=90, in=-90] (tr);
\draw[name path=vert1] (bm) -- (tl);
\draw[name path=vert2] (br) -- (tm);
\path[name intersections={of=vert1 and curv}] coordinate[dot, label=135:\adjsqr{\sigma}] (tau) at (intersection-1);
\path[name intersections={of=vert2 and curv}] coordinate[dot, label=-45:\adjsqr{\rho}] (tau') at (intersection-1);
\begin{pgfonlayer}{background}
\fill[catd] ($(bl) + (-1,0)$) rectangle (tl);
\fill[catd'] (bm) rectangle (tm);
\fill[catd''] (br) rectangle ($(tr) + (1,0)$);
\begin{scope}
\clip ($(bl) + (-1,0)$) rectangle (tl);
\fill[catc] (bl) to[out=90, in=-90] (tr) -- ++(1,0) -- ++(0,-3) -- cycle;
\end{scope}
\begin{scope}
\clip (bm) rectangle (tm);
\fill[catc'] (bl) to[out=90, in=-90] (tr) -- ++(1,0) -- ++(0,-3) -- cycle;
\end{scope}
\begin{scope}
\clip (br) rectangle ($(tr) + (1,0)$);
\fill[catc''] (bl) to[out=90, in=-90] (tr) -- ++(1,0) -- ++(0,-3) -- cycle;
\end{scope}
\end{pgfonlayer}
\end{tikzpicture}
\end{gathered}
)
\end{align*}
That the resulting pair of natural transformations are conjugate can be seen from the following equalities,
given by applying the adjoint square equations \eqref{eq:adjsqa} and \eqref{eq:adjsqb} for both the component adjoint squares:
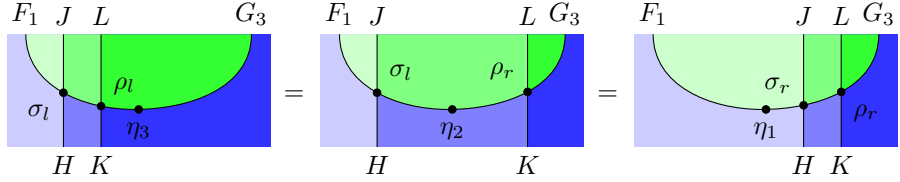
\begin{equation*}
\begin{gathered}
\begin{tikzpicture}[scale=0.5]
\path coordinate[dot, label=below:$\eta_3$] (eta)
 +(-3,2) coordinate[label=above:$F_1$] (tl)
 +(3,2) coordinate[label=above:$G_3$] (tr)
 (tl) ++(1,0) coordinate[label=above:$J$] (vt1) ++(0,-3) coordinate[label=below:$H$] (vb1)
 (tl) ++(2,0) coordinate[label=above:$L$] (vt2) ++(0,-3) coordinate[label=below:$K$] (vb2);
\draw[name path=curv] (tl) to[out=-90, in=180] (eta) to[out=0, in=-90] (tr);
\draw[name path=vert1] (vb1) -- (vt1);
\draw[name path=vert2] (vb2) -- (vt2);
\path[name intersections={of=vert1 and curv}] coordinate[dot, label=-135:\adjsql{\sigma}] (sigma) at (intersection-1);
\path[name intersections={of=vert2 and curv}] coordinate[dot, label=45:\adjsql{\rho}] (sigma') at (intersection-1);
\begin{pgfonlayer}{background}
\fill[catc] ($(tl) + (-0.5,0)$) rectangle (vb1);
\fill[catc'] (vt1) rectangle (vb2);
\fill[catc''] (vb2) rectangle ($(tr) + (0.5,0)$);
\begin{scope}
\clip ($(tl) + (-0.5,0)$) rectangle (vb1);
\fill[catd] (tl) to[out=-90, in=180] (eta) to[out=0, in=-90] (tr) -- cycle;
\end{scope}
\begin{scope}
\clip (vt1) rectangle (vb2);
\fill[catd'] (tl) to[out=-90, in=180] (eta) to[out=0, in=-90] (tr) -- cycle;
\end{scope}
\begin{scope}
\clip (vb2) rectangle ($(tr) + (0.5,0)$);
\fill[catd''] (tl) to[out=-90, in=180] (eta) to[out=0, in=-90] (tr) -- cycle;
\end{scope}
\end{pgfonlayer}
\end{tikzpicture}
\end{gathered}
=
\begin{gathered}
\begin{tikzpicture}[scale=0.5]
\path coordinate[dot, label=below:$\eta_2$] (eta)
 +(-3,2) coordinate[label=above:$F_1$] (tl)
 +(3,2) coordinate[label=above:$G_3$] (tr)
 (tl) ++(1,0) coordinate[label=above:$J$] (vt1) ++(0,-3) coordinate[label=below:$H$] (vb1)
 (tr) ++(-1,0) coordinate[label=above:$L$] (vt2) ++(0,-3) coordinate[label=below:$K$] (vb2);
\draw[name path=curv] (tl) to[out=-90, in=180] (eta) to[out=0, in=-90] (tr);
\draw[name path=vert1] (vb1) -- (vt1);
\draw[name path=vert2] (vb2) -- (vt2);
\path[name intersections={of=vert1 and curv}] coordinate[dot, label=45:\adjsql{\sigma}] (sigma) at (intersection-1);
\path[name intersections={of=vert2 and curv}] coordinate[dot, label=135:\adjsqr{\rho}] (tau') at (intersection-1);
\begin{pgfonlayer}{background}
\fill[catc] ($(tl) + (-0.5,0)$) rectangle (vb1);
\fill[catc'] (vt1) rectangle (vb2);
\fill[catc''] (vb2) rectangle ($(tr) + (0.5,0)$);
\begin{scope}
\clip ($(tl) + (-0.5,0)$) rectangle (vb1);
\fill[catd] (tl) to[out=-90, in=180] (eta) to[out=0, in=-90] (tr) -- cycle;
\end{scope}
\begin{scope}
\clip (vt1) rectangle (vb2);
\fill[catd'] (tl) to[out=-90, in=180] (eta) to[out=0, in=-90] (tr) -- cycle;
\end{scope}
\begin{scope}
\clip (vb2) rectangle ($(tr) + (0.5,0)$);
\fill[catd''] (tl) to[out=-90, in=180] (eta) to[out=0, in=-90] (tr) -- cycle;
\end{scope}
\end{pgfonlayer}
\end{tikzpicture}
\end{gathered}
=
\begin{gathered}
\begin{tikzpicture}[scale=0.5]
\path coordinate[dot, label=below:$\eta_1$] (eta)
 +(-3,2) coordinate[label=above:$F_1$] (tl)
 +(3,2) coordinate[label=above:$G_3$] (tr)
 (tr) ++(-2,0) coordinate[label=above:$J$] (vt1) ++(0,-3) coordinate[label=below:$H$] (vb1)
 (tr) ++(-1,0) coordinate[label=above:$L$] (vt2) ++(0,-3) coordinate[label=below:$K$] (vb2);
\draw[name path=curv] (tl) to[out=-90, in=180] (eta) to[out=0, in=-90] (tr);
\draw[name path=vert1] (vb1) -- (vt1);
\draw[name path=vert2] (vb2) -- (vt2);
\path[name intersections={of=vert1 and curv}] coordinate[dot, label=135:\adjsqr{\sigma}] (tau) at (intersection-1);
\path[name intersections={of=vert2 and curv}] coordinate[dot, label=-45:\adjsqr{\rho}] (tau') at (intersection-1);
\begin{pgfonlayer}{background}
\fill[catc] ($(tl) + (-0.5,0)$) rectangle (vb1);
\fill[catc'] (vt1) rectangle (vb2);
\fill[catc''] (vb2) rectangle ($(tr) + (0.5,0)$);
\begin{scope}
\clip ($(tl) + (-0.5,0)$) rectangle (vb1);
\fill[catd] (tl) to[out=-90, in=180] (eta) to[out=0, in=-90] (tr) -- cycle;
\end{scope}
\begin{scope}
\clip (vt1) rectangle (vb2);
\fill[catd'] (tl) to[out=-90, in=180] (eta) to[out=0, in=-90] (tr) -- cycle;
\end{scope}
\begin{scope}
\clip (vb2) rectangle ($(tr) + (0.5,0)$);
\fill[catd''] (tl) to[out=-90, in=180] (eta) to[out=0, in=-90] (tr) -- cycle;
\end{scope}
\end{pgfonlayer}
\end{tikzpicture}
\end{gathered}
\end{equation*}
\end{proof}
It is straightforward to check that the two forms of composition both give adjoint
squares the structure of a category, with an appropriate choice of identities.
We also have an interchange law that holds between these two types of composition.
\begin{lemma}[Interchange Law]
With the composition operations defined in the proof of proposition \ref{prop:adjsqcomp},
let $P, Q, R, S$ be adjoint squares such that the required composites are well defined,
then the following equality holds:
\begin{equation*}
(S \circ Q) * (R \circ P) = (S * R) \circ (Q * P)
\end{equation*}
\end{lemma}
\begin{proof}
By proposition \ref{prop:adjsqcomp} the composition operations preserve conjugate pairs,
and from lemma \ref{lem:adjsqcond} the members of a conjugate pair determine each other
uniquely. We therefore only need  to consider one of the component natural transformations of the adjoint squares. 
We then have the following sequence of equalities directly from the definitions 
(with a slight abuse of notation we omit the second component of the adjoint squares):
\begin{eqproof*}
\begin{gathered}
\catstyle{catc}{\catdpcol}
\catstyle{catc'}{\catdppcol}
\catstyle{catd}{\catepcol}
\catstyle{catd'}{\cateppcol}
\drawnatd{\adjsql{\lambda}}{}{}{}{}{1}
\end{gathered} 
\circ 
\begin{gathered}
\catstyle{catc}{\catcpcol}
\catstyle{catc'}{\catcppcol}
\catstyle{catd}{\catdpcol}
\catstyle{catd'}{\catdppcol}
\drawnatd{\adjsql{\rho}}{}{}{}{}{1}
\end{gathered} 
) 
* 
(
\begin{gathered}
\catstyle{catc}{\catdcol}
\catstyle{catc'}{\catdpcol}
\catstyle{catd}{\catecol}
\catstyle{catd'}{\catepcol}
\drawnatd{\adjsql{\tau}}{}{}{}{}{1}
\end{gathered} 
\circ 
\begin{gathered}
\drawnatd{\adjsql{\sigma}}{}{}{}{}{1}
\end{gathered} 
)
\explain{definition of vertical composition of adjoint squares}
\begin{gathered}
\begin{tikzpicture}[scale=0.5]
\path coordinate (t) ++(0,-4) coordinate (b)
 (t) ++(-1,0) coordinate (tl) ++(-1,0) coordinate (tll)
 (b) ++(1,0) coordinate (br) ++(1,0) coordinate (brr);
\draw[name path=vert] (b) -- (t);
\draw[name path=curv1] (br) to[out=90, in=-90] (tll);
\draw[name path=curv2] (brr) to[out=90, in=-90] (tl);
\path[name intersections={of=vert and curv1}] coordinate[dot, label=-135:\adjsql{\rho}] (sigma) at (intersection-1);
\path[name intersections={of=vert and curv2}] coordinate[dot, label=45:\adjsql{\lambda}] (sigma') at (intersection-1);
\begin{pgfonlayer}{background}
\fill[catc'] ($(tll) + (-1,0)$) rectangle (b);
\fill[catc''] ($(brr) + (1,0)$) rectangle (t);
\begin{scope}
\clip ($(tll) + (-1,0)$) rectangle (b);
\fill[catd'] (br) to[out=90, in=-90] (tll) -- ++(5,0) -- ++(0,-4) -- cycle;
\fill[cate'] (brr) to[out=90, in=-90] (tl) -- ++(4,0) -- ++(0,-4) -- cycle;
\end{scope}
\begin{scope}
\clip ($(brr) + (1,0)$) rectangle (t);
\fill[catd''] (br) to[out=90, in=-90] (tll) -- ++(5,0) -- ++(0,-4) -- cycle;
\fill[cate''] (brr) to[out=90, in=-90] (tl) -- ++(4,0) -- ++(0,-4) -- cycle;
\end{scope}
\end{pgfonlayer}
\end{tikzpicture}
\end{gathered}
*
\begin{gathered}
\begin{tikzpicture}[scale=0.5]
\path coordinate (t) ++(0,-4) coordinate (b)
 (t) ++(-1,0) coordinate (tl) ++(-1,0) coordinate (tll)
 (b) ++(1,0) coordinate (br) ++(1,0) coordinate (brr);
\draw[name path=vert] (b) -- (t);
\draw[name path=curv1] (br) to[out=90, in=-90] (tll);
\draw[name path=curv2] (brr) to[out=90, in=-90] (tl);
\path[name intersections={of=vert and curv1}] coordinate[dot, label=-135:\adjsql{\sigma}] (sigma) at (intersection-1);
\path[name intersections={of=vert and curv2}] coordinate[dot, label=45:\adjsql{\tau}] (sigma') at (intersection-1);
\begin{pgfonlayer}{background}
\fill[catc] ($(tll) + (-1,0)$) rectangle (b);
\fill[catc'] ($(brr) + (1,0)$) rectangle (t);
\begin{scope}
\clip ($(tll) + (-1,0)$) rectangle (b);
\fill[catd] (br) to[out=90, in=-90] (tll) -- ++(5,0) -- ++(0,-4) -- cycle;
\fill[cate] (brr) to[out=90, in=-90] (tl) -- ++(4,0) -- ++(0,-4) -- cycle;
\end{scope}
\begin{scope}
\clip ($(brr) + (1,0)$) rectangle (t);
\fill[catd'] (br) to[out=90, in=-90] (tll) -- ++(5,0) -- ++(0,-4) -- cycle;
\fill[cate'] (brr) to[out=90, in=-90] (tl) -- ++(4,0) -- ++(0,-4) -- cycle;
\end{scope}
\end{pgfonlayer}
\end{tikzpicture}
\end{gathered}
\explain{definition of horizontal composition of adjoint squares}
\begin{tikzpicture}[scale=0.5]
\path coordinate (a) ++(1,0) coordinate (b) ++(1,0) coordinate (c) ++(1,0) coordinate (d)
 (c) ++(0,-4) coordinate (a') ++(1,0) coordinate (b') ++(1,0) coordinate (c') ++(1,0) coordinate (d');
\draw[name path=vert1] (a') -- (c);
\draw[name path=vert2] (b') -- (d);
\draw[name path=curv1] (c') to[out=90, in=-90] (a);
\draw[name path=curv2] (d') to[out=90, in=-90] (b);
\path[name intersections={of=vert1 and curv1}] coordinate[dot, label=-135:\adjsql{\sigma}] (sigma') at (intersection-1);
\path[name intersections={of=vert2 and curv1}] coordinate[dot, label=-135:\adjsql{\rho}] (sigma') at (intersection-1);
\path[name intersections={of=vert1 and curv2}] coordinate[dot, label=45:\adjsql{\tau}] (sigma') at (intersection-1);
\path[name intersections={of=vert2 and curv2}] coordinate[dot, label=45:\adjsql{\lambda}] (sigma') at (intersection-1);
\begin{pgfonlayer}{background}
\fill[catc] ($(a) + (-1,0)$) rectangle (a');
\fill[catc'] (c) rectangle (b');
\fill[catc''] (d) rectangle ($(d') + (1,0)$);
\begin{scope}
\clip ($(a) + (-1,0)$) rectangle (a');
\fill[catd] (c') to[out=90, in=-90] (a) -- ++(6,0) -- ++(0,-4) -- cycle;
\fill[cate] (d') to[out=90, in=-90] (b) -- ++(5,0) -- ++(0,-4) -- cycle;
\end{scope}
\begin{scope}
\clip (c) rectangle (b');
\fill[catd'] (c') to[out=90, in=-90] (a) -- ++(6,0) -- ++(0,-4) -- cycle;
\fill[cate'] (d') to[out=90, in=-90] (b) -- ++(5,0) -- ++(0,-4) -- cycle;
\end{scope}
\begin{scope}
\clip (d) rectangle ($(d') + (1,0)$);
\fill[catd''] (c') to[out=90, in=-90] (a) -- ++(6,0) -- ++(0,-4) -- cycle;
\fill[cate''] (d') to[out=90, in=-90] (b) -- ++(5,0) -- ++(0,-4) -- cycle;
\end{scope}
\end{pgfonlayer}
\end{tikzpicture}
\explain{definition of vertical composition of adjoint squares}
\begin{gathered}
\begin{tikzpicture}[scale=0.5]
\path coordinate (tl) ++(3,-3) coordinate (br)
 (tl) ++(1,0) coordinate (tm) ++(1,0) coordinate (tr)
 (br) ++(-1,0) coordinate (bm) ++(-1,0) coordinate (bl);
\draw[name path=curv] (br) to[out=90, in=-90] (tl);
\draw[name path=vert1] (bm) -- (tr);
\draw[name path=vert2] (bl) -- (tm);
\path[name intersections={of=vert1 and curv}] coordinate[dot, label=45:\adjsql{\tau}] (sigma') at (intersection-1);
\path[name intersections={of=vert2 and curv}] coordinate[dot, label=-135:\adjsql{\lambda}] (sigma) at (intersection-1);
\begin{pgfonlayer}{background}
\fill[catd] ($(tl) + (-1,0)$) rectangle (bl);
\fill[catd'] (tm) rectangle (bm);
\fill[catd''] (tr) rectangle ($(br) + (1,0)$);
\begin{scope}
\clip ($(tl) + (-1,0)$) rectangle (bl);
\fill[cate] (br) to[out=90, in=-90] (tl) -- ++(4,0) -- ++(0,-3) -- cycle;
\end{scope}
\begin{scope}
\clip (tm) rectangle (bm);
\fill[cate'] (br) to[out=90, in=-90] (tl) -- ++(4,0) -- ++(0,-3) -- cycle;
\end{scope}
\begin{scope}
\clip (tr) rectangle ($(br) + (1,0)$);
\fill[cate''] (br) to[out=90, in=-90] (tl) -- ++(4,0) -- ++(0,-3) -- cycle;
\end{scope}
\end{pgfonlayer}
\end{tikzpicture}
\end{gathered}
\circ
\begin{gathered}
\begin{tikzpicture}[scale=0.5]
\path coordinate (tl) ++(3,-3) coordinate (br)
 (tl) ++(1,0) coordinate (tm) ++(1,0) coordinate (tr)
 (br) ++(-1,0) coordinate (bm) ++(-1,0) coordinate (bl);
\draw[name path=curv] (br) to[out=90, in=-90] (tl);
\draw[name path=vert1] (bm) -- (tr);
\draw[name path=vert2] (bl) -- (tm);
\path[name intersections={of=vert1 and curv}] coordinate[dot, label=45:\adjsql{\rho}] (sigma') at (intersection-1);
\path[name intersections={of=vert2 and curv}] coordinate[dot, label=-135:\adjsql{\sigma}] (sigma) at (intersection-1);
\begin{pgfonlayer}{background}
\fill[catc] ($(tl) + (-1,0)$) rectangle (bl);
\fill[catc'] (tm) rectangle (bm);
\fill[catc''] (tr) rectangle ($(br) + (1,0)$);
\begin{scope}
\clip ($(tl) + (-1,0)$) rectangle (bl);
\fill[catd] (br) to[out=90, in=-90] (tl) -- ++(4,0) -- ++(0,-3) -- cycle;
\end{scope}
\begin{scope}
\clip (tm) rectangle (bm);
\fill[catd'] (br) to[out=90, in=-90] (tl) -- ++(4,0) -- ++(0,-3) -- cycle;
\end{scope}
\begin{scope}
\clip (tr) rectangle ($(br) + (1,0)$);
\fill[catd''] (br) to[out=90, in=-90] (tl) -- ++(4,0) -- ++(0,-3) -- cycle;
\end{scope}
\end{pgfonlayer}
\end{tikzpicture}
\end{gathered}
\explain{definition of horizontal composition of adjoint squares}
(
\begin{gathered}
\catstyle{catc}{\catdpcol}
\catstyle{catc'}{\catdppcol}
\catstyle{catd}{\catepcol}
\catstyle{catd'}{\cateppcol}
\drawnatd{$\sigma^S$}{}{}{}{}{1}
\end{gathered} 
* 
\begin{gathered}
\catstyle{catc}{\catdcol}
\catstyle{catc'}{\catdpcol}
\catstyle{catd}{\catecol}
\catstyle{catd'}{\catepcol}
\drawnatd{$\sigma^R$}{}{}{}{}{1}
\end{gathered} 
) 
\circ 
(
\begin{gathered}
\catstyle{catc}{\catcpcol}
\catstyle{catc'}{\catcppcol}
\catstyle{catd}{\catdpcol}
\catstyle{catd'}{\catdppcol}
\drawnatd{$\sigma^Q$}{}{}{}{}{1}
\end{gathered} 
* 
\begin{gathered}
\drawnatd{$\sigma^P$}{}{}{}{}{1}
\end{gathered} 
)
\end{eqproof*}
\end{proof}
\begin{remark}
In fact, the adjoint squares form a double category \citep{Ehresmann1963}, 
commonly used as a device for organizing information about adjunctions, 
see \citep{Palmquist1971} and \citep{Day1974}. Proving the remaining details
graphically is left for the interested reader.
\end{remark}

\section{Monads}
Monads, also historically referred to as triples and standard constructions,
are another important notion in category theory, and are also of particular interest to
computer scientists with their connections to the semantics of programming languages, see
for example \citep{Moggi1991} and \citep{Wadler1995}. Monads are of course described in \citep{MacLane1998},
but there is much that is not covered. Standard sources for monad theory are \citep{BarrWells2005} and \citep{Manes1976},
a more modern presentation is given in the book in preparation \citep{Jacobs2012b}. We do not attempt to
give a comprehensive treatment here, rather we aim to formulate some basic aspects of monads in graphical terms,
and then provide examples of the string diagrammatic approach by proving some standard results.
As with the description of adjunctions in section \ref{sec:adjunctions}, we will not provide concrete examples of monads, leaving the
interested reader to refer to the various references in the text.

The string diagram notation is particularly effective when working with monads, providing a simple perspective
on many of the mathematical structures under consideration. These ideas will be highlighted in this section.

\subsection{Calculational Properties of Monads}
We now introduce graphical descriptions of monads and some associated mathematical structures.
\begin{definition}[Monad]
Let $\mathcal{C}$ be some category. A \define{monad} on $\mathcal{C}$, consists
of an endofunctor $T$ on $\mathcal{C}$, a \define{unit} natural transformation $\eta$ and
\define{multiplication} natural transformation $\mu$:
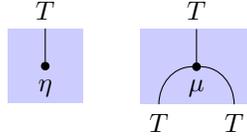
\begin{equation*}
\begin{gathered}
\begin{tikzpicture}[scale=0.5]
\path coordinate[dot, label=below:$\eta$] (eta) ++(0,1) coordinate[label=above:$T$] (t);
\draw (eta) -- (t);
\begin{pgfonlayer}{background}
\fill[catc] ($(t) + (-1,0)$) rectangle ($(eta) + (1,-1)$);
\fill[white] ($(eta) + (-1,-1)$) rectangle ($(eta) + (1,-2)$);
\end{pgfonlayer}
\end{tikzpicture}
\end{gathered}
\qquad 
\begin{gathered}
\begin{tikzpicture}[scale=0.5]
\path coordinate[dot, label=below:$\mu$] (mu)
 +(0,1) coordinate[label=above:$T$] (t)
 +(-1,-1) coordinate[label=below:$T$] (bl)
 +(1,-1) coordinate[label=below:$T$] (br);
\draw (bl) to[out=90, in=180] (mu.west) -- (mu.east) to[out=0, in=90] (br)
 (mu) -- (t);
\begin{pgfonlayer}{background}
\fill[catc] ($(bl) + (-0.5,0)$) rectangle ($(t) + (1.5,0)$);
\end{pgfonlayer}
\end{tikzpicture}
\end{gathered}
\end{equation*}
The following \define{unit axioms} are required to hold:
\begin{equation}
\label{eq:monad-unit}
\begin{gathered}
\begin{tikzpicture}[scale=0.5]
\path coordinate[dot, label=below:$\mu$] (mu)
 +(-1,-1) coordinate[dot, label=below:$\eta$] (eta)
 +(0,1) coordinate[label=above:$T$] (t)
 ++(1,-1) coordinate (a) ++(0,-1) coordinate[label=below:$T$] (b);
\draw (eta) to[out=90, in=180] (mu)
 (b) -- (a) to[out=90, in=0] (mu.east) -- (mu.north) -- (t);
\begin{pgfonlayer}{background}
\fill[catc] ($(b) + (0.5,0)$) rectangle ($(t) + (-1.5,0)$);
\end{pgfonlayer}
\end{tikzpicture}
\end{gathered}
\enskip=\enskip
\begin{gathered}
\onecelldiag{T}{catc}{catc}{3}{1}
\end{gathered}
\enskip=\enskip
\begin{gathered}
\begin{tikzpicture}[auto, scale=0.5]
\path coordinate[dot, label=below:$\mu$] (mu)
 +(1,-1) coordinate[dot, label=below:$\eta$] (eta)
 +(0,1) coordinate[label=above:$T$] (t)
 ++(-1,-1) coordinate (a) ++(0,-1) coordinate[label=below:$T$] (b);
\draw (eta) to[out=90, in=0] (mu)
 (b) -- (a) to[out=90, in=180] (mu.west) -- (mu.north) -- (t);
\begin{pgfonlayer}{background}
\fill[catc] ($(b) + (-0.5,0)$) rectangle ($(t) + (1.5,0)$);
\end{pgfonlayer}
\end{tikzpicture}
\end{gathered}
\end{equation}
Also $\mu$ is required to satisfy the following \define{associativity axiom}:
\begin{equation}
\label{eq:monad-associativity}
\begin{gathered}
\begin{tikzpicture}[scale=0.5]
\path coordinate[dot, label=below:$\mu$] (mu) ++(1,1) coordinate[dot, label=below:$\mu$] 
(mu') +(0,1) coordinate[label=above:$T$] (t) ++(1,-1) coordinate (a) ++(0,-1) coordinate[label=below:$T$] (br)
 (mu) +(-1,-1) coordinate[label=below:$T$] (bl)
 (mu) +(1,-1) coordinate[label=below:$T$] (bm);
\draw (bl) to[out=90, in=180] (mu.west) -- (mu.north) to[out=90, in=180] (mu'.west) -- (mu'.north) -- (t)
 (br) -- (a) to[out=90, in=0] (mu')
 (bm) to[out=90, in=0] (mu);
\begin{pgfonlayer}{background}
\fill[catc] ($(bl) + (-0.5,0)$) rectangle ($(t) + (1.5,0)$);
\end{pgfonlayer}
\end{tikzpicture}
\end{gathered}
\enskip=\enskip
\begin{gathered}
\begin{tikzpicture}[scale=0.5]
\path coordinate[dot, label=below:$\mu$] (mu) ++(-1,1) coordinate[dot, label=below:$\mu$] 
 (mu') +(0,1) coordinate[label=above:$T$] (t) ++(-1,-1) coordinate (a) ++(0,-1) coordinate[label=below:$T$] (bl)
 (mu) +(-1,-1) coordinate[label=below:$T$] (bm)
 (mu) +(1,-1) coordinate[label=below:$T$] (br);
\draw (br) to[out=90, in=0] (mu.east) -- (mu.north) to[out=90, in=0] (mu'.east) -- (mu'.north) -- (t)
 (bm) to[out=90, in=180] (mu.west)
 (bl) -- (a) to[out=90, in=180] (mu'.west);
\begin{pgfonlayer}{background}
\fill[catc] ($(br) + (0.5, 0)$) rectangle ($(t) + (-1.5, 0)$);
\end{pgfonlayer}
\end{tikzpicture}
\end{gathered}
\end{equation}
The string diagram notation makes clear that a monad can be seen as a monoid in
the functor category $[\mathcal{C}, \mathcal{C}]$, as suggested by the use of the names ``unit'' and ``multiplication''.
\end{definition}
There is a close relationship between monads and adjunctions, in fact every adjunction induces a monad.
\begin{lemma}
\label{lem:adjmod}
Every adjunction:
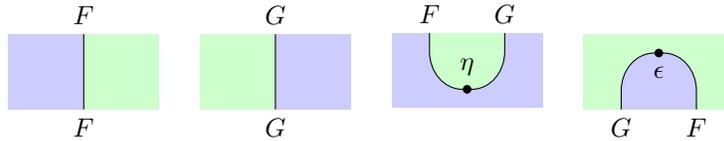
\begin{center}
\begin{tabular}{c c c c}
\onecelldiag{F}{catc}{catd}{2}{2}
&
\onecelldiag{G}{catd}{catc}{2}{2}
&
\cupcell{\eta}{F}{G}{catc}{catd}
&
\capcell{\epsilon}{G}{F}{catd}{catc}
\end{tabular}
\end{center}
induces a monad.
\end{lemma}
\begin{proof}
Take the endofunctor as $G \circ F$, and the unit of the monad is also $\eta$. Take the multiplication $\mu$
of the monad as the composite:
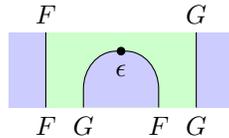
\begin{center}
\begin{tikzpicture}[scale=0.5]
\path coordinate[dot, label=below:$\epsilon$] 
 (epsilon) ++(-1,-1) coordinate (a) ++(0,-0.5) coordinate[label=below:$G$] (bl) ++(-1,0) coordinate[label=below:$F$] (bll) 
  ++(0,2) coordinate[label=above:$F$] (tl)
 (epsilon) ++(1,-1) coordinate (b) ++(0,-0.5) coordinate[label=below:$F$] (br) ++(1,0) coordinate[label=below:$G$] (brr) 
  ++(0,2) coordinate[label=above:$G$] (tr);
\draw (bll) -- (tl)
 (brr) -- (tr)
 (bl) -- (a) to[out=90, in=180] (epsilon.west) -- (epsilon.east) to[out=0, in=90] (b) -- (br);
\begin{pgfonlayer}{background}
\fill[catc] ($(tl) + (-1,0)$) rectangle (bll);
\fill[catc] ($(tr) + (1,0)$) rectangle (brr);
\fill[catd] (bll) rectangle (tr);
\fill[catc] (bl) -- (a) to[out=90, in=180] (epsilon.west) -- (epsilon.east) to[out=0, in=90] (b) -- (br) -- cycle;
\end{pgfonlayer}
\end{tikzpicture}
\end{center}
We must then check the required axioms hold, for the unit axioms we apply the properties of adjunctions:
\begin{equation*}
\begin{gathered}
\begin{tikzpicture}[scale=0.5]
\path coordinate[dot, label=above:$\eta$] (eta) ++(2,1) coordinate[dot, label=below:$\epsilon$] (epsilon)
 ++(1,-1) coordinate (a) ++(0,-0.5) coordinate[label=below:$F$] (bl)
 (eta) ++(-1,1) coordinate (b) ++(0,0.5) coordinate[label=above:$F$] (tl)
 (bl) ++(1,0) coordinate[label=below:$G$] (br) ++(0,2) coordinate[label=above:$G$] (tr);
\draw (tl) -- (b) to[out=-90, in=180] (eta.west) -- (eta.east) to[out=0, in=180] (epsilon.west) -- (epsilon.east) to[out=0, in=90] (a) -- (bl);
\draw (br) -- (tr);
\begin{pgfonlayer}{background}
\fill[catc] ($(tl) + (-1,0)$) rectangle ($(br) + (1,0)$);
\fill[catd] (tl) -- (b) to[out=-90, in=180] (eta.west) -- (eta.east) to[out=0, in=180] (epsilon.west) -- (epsilon.east) to[out=0, in=90] (a) -- (bl) -- (br) -- (tr) -- cycle;
\end{pgfonlayer}
\end{tikzpicture}
\end{gathered}
\stackrel{\eqref{eq:snake-eta-epsilon}}{=}
\begin{gathered}
\begin{tikzpicture}[scale=0.5]
\path coordinate[label=above:$F$] (tl) ++(1,0) coordinate[label=above:$G$] (tr)
 (tl) ++(0,-2) coordinate[label=below:$F$] (bl)
 (tr) ++(0,-2) coordinate[label=below:$G$] (br);
\draw (bl) -- (tl)
 (br) -- (tr);
\begin{pgfonlayer}{background}
\fill[catc] ($(tl) + (-1,0)$) rectangle (bl);
\fill[catc] ($(tr) + (1,0)$) rectangle (br);
\fill[catd] (tl) rectangle (br);
\end{pgfonlayer}
\end{tikzpicture}
\end{gathered}
\stackrel{\eqref{eq:snake-epsilon-eta}}{=}
\begin{gathered}
\begin{tikzpicture}[scale=0.5]
\path coordinate[dot, label=above:$\eta$] (eta) ++(-2,1) coordinate[dot, label=below:$\epsilon$] (epsilon)
 ++(-1,-1) coordinate (a) ++(0,-0.5) coordinate[label=below:$G$] (br)
 (eta) ++(1,1) coordinate (b) ++(0,0.5) coordinate[label=above:$G$] (tr)
 (br) ++(-1,0) coordinate[label=below:$F$] (bl) ++(0,2) coordinate[label=above:$F$] (tl);
\draw (tr) -- (b) to[out=-90, in=0] (eta.east) -- (eta.west) to[out=180, in=0] (epsilon.east) -- (epsilon.west) to[out=180, in=90] (a) -- (br)
 (bl) -- (tl);
\begin{pgfonlayer}{background}
\fill[catc] ($(bl) + (-1,0)$) rectangle ($(tr) + (1,0)$);
\fill[catd] (tr) -- (b) to[out=-90, in=0] (eta.east) -- (eta.west) to[out=180, in=0] (epsilon.east) -- (epsilon.west) to[out=180, in=90] (a) -- (br)
 -- (bl) -- (tl) -- cycle;
\end{pgfonlayer}
\end{tikzpicture}
\end{gathered}
\end{equation*}
and the associativity axioms follow directly from naturality:
\newcommand{\diagelem}[1]{
\begin{gathered}
\begin{tikzpicture}[auto,scale=0.3]
\path (0,-2.5) coordinate (base) 
 (-1.5,-#1) coordinate[dot, label=below:$\epsilon$] 
 (epsilon1) ++(-1,-1) coordinate (a)
 (epsilon1) ++(1,-1) coordinate (b)
 (1.5,#1) coordinate[dot, label=below:$\epsilon$] (epsilon2)
 (epsilon2) ++(-1,-1) coordinate (c)
 (epsilon2) ++(1,-1) coordinate (d);
\path let \p1 = (base) in
      let \p2 = (a) in
      let \p3 = (b) in
      let \p4 = (c) in
      let \p5 = (d) in
 coordinate[label=below:$G$] (bll) at (\x2, \y1)
 coordinate[label=below:$F$] (bl) at (\x3, \y1)
 coordinate[label=below:$G$] (br) at (\x4, \y1)
 coordinate[label=below:$F$] (brr) at (\x5, \y1);
\path 
 (bll) ++(-1,0) coordinate[label=below:$F$] (blll) ++(0,4) coordinate[label=above:$F$] (tlll)
 (brr) ++(1,0) coordinate[label=below:$G$] (brrr) +(0,4) coordinate[label=above:$G$] (trrr);
\draw (bll) -- (a) to[out=90, in=180] (epsilon1.west) -- (epsilon1.east) to[out=0, in=90] (b) -- (bl)
      (br) -- (c) to[out=90, in=180] (epsilon2.west) -- (epsilon2.east) to[out=0, in=90] (d) -- (brr)
 (blll) -- (tlll)
 (brrr) -- (trrr);
\begin{pgfonlayer}{background}
\fill[catd] ($(blll) + (-1,0)$) rectangle ($(trrr) + (1,0)$);
\fill[catc] ($(blll) + (-1,0)$) rectangle (tlll);
\fill[catc] ($(brrr) + (1,0)$) rectangle (trrr);
\fill[catc] (bll) -- (a) to[out=90, in=180] (epsilon1.west) -- (epsilon1.east) to[out=0, in=90] (b) -- (bl) -- cycle;
\fill[catc] (br) -- (c) to[out=90, in=180] (epsilon2.west) -- (epsilon2.east) to[out=0, in=90] (d) -- (brr) -- cycle;
\end{pgfonlayer}
\end{tikzpicture}
\end{gathered}
}
\begin{equation*}
\diagelem{1}
=
\diagelem{0}
=
\diagelem{-1}
\end{equation*}
\end{proof}
We now introduce a suitable notion of morphism for monads on the same category. We will generalize this
situation in section \ref{sec:moncat}.
\begin{definition}[Monad Morphism]
\label{def:monadmorphism}
For a fixed category $\mathcal{C}$, a \define{monad morphism} $(\mathcal{C}, T, \eta, \mu) \rightarrow (\mathcal{C}, T', \eta', \mu')$ is a natural transformation $\sigma : T \Rightarrow T'$ such that the following equalities hold:
\begin{subequations}
\begin{trivlist}\vspace*{-\baselineskip}
\item
\begin{minipage}{0.495\textwidth}
\begin{equation}
\label{eq:mmunit}
\begin{gathered}
\begin{tikzpicture}[scale=0.5]
\path coordinate[dot, label=right:$\eta$] (eta) ++(0,1) coordinate[dot, label=right:$\sigma$] (sigma) ++(0,1) coordinate[label=above:$T'$] (t);
\draw (eta) -- (sigma) -- (t);
\begin{pgfonlayer}{background}
\fill[catc] ($(eta) +(-1,-1)$) rectangle ($(t) + (1.5,0)$);
\fill[white] ($(eta) +(-1,-1)$) rectangle ($(eta) +(1,-1.5)$);
\end{pgfonlayer}
\end{tikzpicture}
\end{gathered}
=
\begin{gathered}
\begin{tikzpicture}[scale=0.5]
\path coordinate[dot, label=right:$\eta'$] (eta) ++(0,2) coordinate[label=above:$T'$] (t);
\draw (eta) -- (t);
\begin{pgfonlayer}{background}
\fill[catc] ($(eta) +(-1,-1)$) rectangle ($(t) + (1.5,0)$);
\fill[white] ($(eta) +(-1,-1)$) rectangle ($(eta) +(1,-1.5)$);
\end{pgfonlayer}
\end{tikzpicture}
\end{gathered}
\end{equation}
\end{minipage}
\begin{minipage}{0.495\textwidth}
\begin{equation}
\label{eq:mmmult}
\begin{gathered}
\begin{tikzpicture}[scale=0.5]
\path coordinate[dot, label=below:$\mu$] (mu)
 +(-1,-1) coordinate[label=below:$T$] (bl)
 +(1,-1) coordinate[label=below:$T$] (br)
 ++(0,1) coordinate[dot, label=right:$\sigma$] (sigma)
 ++(0,1) coordinate[label=above:$T'$] (t);
\draw (mu) -- (sigma) -- (t)
 (bl) to[out=90, in=180] (mu.west) -- (mu.east) to[out=0, in=90] (br);
\begin{pgfonlayer}{background}
\fill[catc] ($(bl) + (-1,0)$) rectangle ($(t) + (2,0)$);
\end{pgfonlayer}
\end{tikzpicture}
\end{gathered}
=
\begin{gathered}
\begin{tikzpicture}[scale=0.5]
\path coordinate[dot, label=below:$\mu'$] (mu) 
 +(0,1) coordinate[label=above:$T'$] (t)
 (mu) ++(-1,-1) coordinate[dot, label=left:$\sigma$] (sigma1) ++(0,-1) coordinate[label=below:$T$] (bl)
 (mu) ++(1,-1) coordinate[dot, label=right:$\sigma$] (sigma2) ++(0,-1) coordinate[label=below:$T$] (br);
\draw (mu) -- (t)
 (bl) -- (sigma1) to[out=90, in=180] (mu.west) -- (mu.east) to[out=0, in=90] (sigma2) -- (br);
\begin{pgfonlayer}{background}
\fill[catc] ($(bl) + (-1,0)$) rectangle ($(t) + (2,0)$);
\end{pgfonlayer}
\end{tikzpicture}
\end{gathered}
\end{equation}
\end{minipage}
\end{trivlist}
\end{subequations}
Again the string notation is instructive, if we consider the two monads as monoids in $[\mathcal{C}, \mathcal{C}]$ then
$\sigma$ is clearly a monoid homomorphism, commuting appropriately with the unit and multiplication.
\end{definition}
\begin{lemma}
Let $(\mathcal{C}, T, \eta, \mu)$, $(\mathcal{C}, T', \eta', \mu')$ and $(\mathcal{C}, T'', \eta'', \mu'')$ be monads.
Also let $\sigma : T \Rightarrow T'$ and $\tau : T' \Rightarrow T''$ be monad morphisms. Then $\tau \circ \sigma$ is a monad morphism. Also for any monad $(\mathcal{C}, T, \eta, \mu)$, the identity natural transformation on $T$ is a monad morphism.
\end{lemma}
\begin{proof}
For the unit we have:
\begin{equation*}
\begin{gathered}
\begin{tikzpicture}[scale=0.5]
\path coordinate[dot, label=right:$\eta$] (eta) ++(0,1) coordinate[dot, label=right:$\sigma$] (sigma) ++(0,1) coordinate[dot, label=right:$\tau$] (tau) 
 ++(0,1) coordinate[label=above:$T''$] (t);
\draw (eta) -- (sigma) -- (tau) -- (t);
\begin{pgfonlayer}{background}
\fill[catc] ($(eta) + (-1,-1)$) rectangle ($(t) + (2,0)$);
\end{pgfonlayer}
\end{tikzpicture}
\end{gathered}
\stackrel{\eqref{eq:mmunit}}{=}
\begin{gathered}
\begin{tikzpicture}[scale=0.5]
\path coordinate[dot, label=right:$\eta'$] (eta) ++(0,1) coordinate[dot, label=right:$\tau$] (tau) ++(0,1) coordinate[label=above:$T''$] (t);
\draw (eta) -- (tau) -- (t);
\begin{pgfonlayer}{background}
\fill[catc] ($(eta) + (-1,-2)$) rectangle ($(t) + (2,0)$);
\end{pgfonlayer}
\end{tikzpicture}
\end{gathered}
\stackrel{\eqref{eq:mmunit}}{=}
\begin{gathered}
\begin{tikzpicture}[scale=0.5]
\path coordinate[dot, label=right:$\eta''$] (eta) ++(0,1) coordinate[label=above:$T''$] (t);
\draw (eta) -- (t);
\begin{pgfonlayer}{background}
\fill[catc] ($(eta) + (-1,-3)$) rectangle ($(t) + (2,0)$);
\end{pgfonlayer}
\end{tikzpicture}
\end{gathered}
\end{equation*}
Similarly for the multiplication:
\begin{equation*}
\begin{gathered}
\begin{tikzpicture}[scale=0.5]
\path coordinate[dot, label=below:$\mu''$] (mu) ++(0,1) coordinate[label=above:$T''$] (t)
 (mu) ++(-1,-1) coordinate[dot, label=left:$\tau$] (tau1) ++(0,-1) coordinate[dot, label=left:$\sigma$] (sigma1) ++(0,-1) coordinate[label=below:$T$] (bl)
 (mu) ++(1,-1) coordinate[dot, label=right:$\tau$] (tau2) ++(0,-1) coordinate[dot, label=right:$\sigma$] (sigma2) ++(0,-1) coordinate[label=below:$T$] (br);
\draw (mu) -- (t)
 (bl) -- (sigma1) -- (tau1) to[out=90, in=180] (mu.west) -- (mu.east) to[out=0, in=90] (tau2) -- (sigma2) -- (br);
\begin{pgfonlayer}{background}
\fill[catc] ($(bl) + (-1,0)$) rectangle ($(t) + (2,0)$);
\end{pgfonlayer} 
\end{tikzpicture}
\end{gathered}
\stackrel{\eqref{eq:mmmult}}{=}
\begin{gathered}
\begin{tikzpicture}[scale=0.5]
\path coordinate[dot, label=below:$\mu'$] (mu) ++(0,1) coordinate[dot, label=right:$\tau$] (tau) ++(0,1) coordinate[label=above:$T''$] (t)
 (mu) ++(-1,-1) coordinate[dot, label=left:$\sigma$] (sigma1) ++(0,-1) coordinate[label=below:$T$] (bl)
 (mu) ++(1,-1) coordinate[dot, label=right:$\sigma$] (sigma2) ++(0,-1) coordinate[label=below:$T$] (br);
\draw (mu) -- (tau) -- (t)
 (bl) -- (sigma1) to[out=90, in=180] (mu.west) -- (mu.east) to[out=0, in=90] (sigma2) -- (br);
\begin{pgfonlayer}{background}
\fill[catc] ($(bl) + (-1,0)$) rectangle ($(t) + (2,0)$);
\end{pgfonlayer}
\end{tikzpicture}
\end{gathered}
\stackrel{\eqref{eq:mmmult}}{=}
\begin{gathered}
\begin{tikzpicture}[scale=0.5]
\path coordinate[dot, label=below:$\mu$] (mu) ++(0,1) coordinate[dot, label=right:$\sigma$] (sigma) ++(0,1) coordinate[dot, label=right:$\tau$] (tau) 
 ++(0,1) coordinate[label=above:$T''$] (t)
 (mu) +(-1,-1) coordinate[label=below:$T$] (bl) +(1,-1) coordinate[label=below:$T$] (br);
\draw (bl) to[out=90, in=180] (mu.west) -- (mu.east) to[out=0, in=90] (br);
\draw (mu) -- (sigma) -- (tau) -- (t);
\begin{pgfonlayer}{background}
\fill[catc] ($(bl) + (-1,0)$) rectangle ($(t) + (2,0)$);
\end{pgfonlayer}
\end{tikzpicture}
\end{gathered}
\end{equation*}
The second part is trivial.
\end{proof}
We could have seen the above lemma immediately as monoid homomorphisms compose.
Clearly for a category $\mathcal{C}$, the monads on $\mathcal{C}$ and the monad
morphisms form a category. We now investigate two important constructions for
a given monad.
\begin{definition}[Eilenberg-Moore Category]
The \define{Eilenberg-Moore category} for a monad $(\mathcal{C}, T, \eta, \mu)$, denoted \eilmo{T},
 is defined as follows:
\begin{itemize}
 \item {\bf Objects}: Pairs consisting of an object $X$ of $\mathcal{C}$ and a morphism $a : T(X) \rightarrow X$
   satisfying the following equalities:
\begin{subequations}
\begin{trivlist}
\item
\begin{minipage}{0.395\linewidth}
\begin{equation}
\label{eq:emunit}
\begin{gathered}
\begin{tikzpicture}[scale=0.5]
\path coordinate[dot, label=left:$a$] (a) +(0,1) coordinate[label=above:$X$] (t) +(0,-2) coordinate[label=below:$X$] (b)
 +(1,-1) coordinate[dot, label=below:$\eta$] (eta);
\draw (b) -- (a) -- (t)
 (eta) to[out=90, in=0] (a);
\begin{pgfonlayer}{background}
\fill[catterm] ($(t) + (-1,0)$) rectangle (b);
\fill[catc] (t) rectangle ($(b) + (1.5,0)$);
\end{pgfonlayer}
\end{tikzpicture}
\end{gathered}
=
\begin{gathered}
\onecelldiag{X}{catterm}{catc}{3}{1}
\end{gathered}
\end{equation}
\end{minipage}
\begin{minipage}{0.595\linewidth}
\begin{equation}
\label{eq:emmult}
\begin{gathered}
\begin{tikzpicture}[scale=0.5]
\path coordinate[dot, label=left:$a$] (a) 
 +(0,1) coordinate[label=above:$X$] (t)
 +(0,-2) coordinate[label=below:$X$] (bl)
 ++(2,-1) coordinate[dot, label=below:$\mu$] (mu)
 +(-1,-1) coordinate[label=below:$T$] (bm)
 +(1,-1) coordinate[label=below:$T$] (br);
\draw (bl) -- (a) -- (t)
 (bm) to[out=90, in=180] (mu.west) -- (mu.east) to[out=0, in=90] (br)
 (mu) to[out=90, in=0] (a);
\begin{pgfonlayer}{background}
\fill[catterm] ($(t) + (-1,0)$) rectangle (bl);
\fill[catc] (t) rectangle ($(br) + (0.5,0)$);
\end{pgfonlayer}
\end{tikzpicture}
\end{gathered}
=
\begin{gathered}
\begin{tikzpicture}[scale=0.5]
\path coordinate[dot, label=left:$a$] (a1) +(0,-1) coordinate[label=below:$X$] (bl)
 ++(0,1) coordinate[dot, label=left:$a$] (a2) ++(0,1) coordinate[label=above:$X$] (t)
 (bl) ++(1,0) coordinate[label=below:$T$] (bm) ++(1,0) coordinate[label=below:$T$] (br);
\draw (bl) -- (a1) -- (a2) -- (t)
 (bm) to[out=90, in=0] (a1)
 (br) to[out=90, in=0] (a2);
\begin{pgfonlayer}{background}
\fill[catterm] ($(t) + (-1,0)$) rectangle (bl);
\fill[catc] (t) rectangle ($(br) + (0.5,0)$);
\end{pgfonlayer}
\end{tikzpicture}
\end{gathered}
\end{equation}
\end{minipage}
\end{trivlist}
\end{subequations}
 \item {\bf Morphisms}: A morphism $(X, a) \rightarrow (Y, b)$ is an algebra homomorphism, as defined in example \ref{ex:alghom}.
  Composition of morphisms is as in $\mathcal{C}$.
\end{itemize}
The graphical notation shows that if we consider the monad as a monoid in $[\mathcal{C}, \mathcal{C}]$ then
we can view the objects in the Eilenberg-Moore category as monoid actions on objects in $\mathcal{C}$.
In fact $X$ is a constant functor, and we can generalize the above definition to arbitrary functors. 
We will not require this extra generality in our setting, details can be found in \citep{KellyStreet1974}.
\end{definition}
In lemma \ref{lem:adjmod} we saw that every adjunction induces a monad. We now see that every monad arises
in this way.
\begin{proposition}[\cite{EilenbergMoore1965}]
\label{prop:emadj}
For a monad $(\mathcal{C}, T, \eta, \mu)$ there are functors:
\begin{equation*}
U : \eilmo{T} \rightarrow \mathcal{C}
\end{equation*}
and:
\begin{equation*}
J : \mathcal{C} \rightarrow \eilmo{T}
\end{equation*}
with $J \dashv U$ and the monad induced as in lemma \ref{lem:adjmod} is equal to $T$.
\end{proposition}
\begin{proof}
We define:
\begin{align*}
U : \eilmo{T} &\rightarrow \mathcal{C}\\
(X, a : T(X) \rightarrow X) &\mapsto X\\
h : X \rightarrow Y &\mapsto h
\end{align*}
and:
\begin{align*}
J : \mathcal{C} &\rightarrow \eilmo{T}\\
X &\mapsto \mu_X\\
h : X \rightarrow Y &\mapsto T(h)
\end{align*}
That these are valid functors is easy to check. Now we require a bijection:
\begin{center}
\begin{prooftree}
\AxiomC{$J(X) \rightarrow (Y, a : T(Y) \rightarrow Y)$}
\doubleLine
\UnaryInfC{$X \rightarrow U(Y, a : T(Y) \rightarrow Y)$}
\end{prooftree}
\end{center}
Expanding definitions this becomes:
\begin{center}
\begin{prooftree}
\newcommand{\emleft}{
\begin{gathered}
\begin{tikzpicture}[scale=0.5]
\path coordinate[dot, label=below:$\mu$] (mu) +(-1,-1) coordinate[label=below:$T$] (bm) +(1,-1) coordinate[label=below:$T$] (br)
 +(0,1) coordinate[label=above:$T$] (tr)
 (bm) ++(-1,0) coordinate[label=below:$X$] (bl) ++(0,2) coordinate[label=above:$X$] (tl);
\draw (bl) -- (tl)
 (mu) -- (tr)
 (bm) to[out=90, in=180] (mu.west) -- (mu.east) to[out=0, in=90] (br);
\begin{pgfonlayer}{background}
\fill[catterm] ($(tl) + (-1,0)$) rectangle (bl);
\fill[catc] (tl) rectangle ($(br) + (1,0)$);
\end{pgfonlayer}
\end{tikzpicture}
\end{gathered}
}
\newcommand{\emright}{
\begin{gathered}
\gmorph{a}{T}{Y}{Y}{catterm}{catc}{catc}{1}{1}
\end{gathered}
}
\AxiomC{$\emleft \rightarrow \emright$}
\doubleLine
\UnaryInfC{$X \rightarrow Y$}
\end{prooftree}
\end{center}
We define a map in the downward direction as follows:
\begin{equation*}
\begin{gathered}
\gmorph{h}{T}{X}{Y}{catterm}{catc}{catc}{1.5}{1}
\end{gathered}
\mapsto
\begin{gathered}
\begin{tikzpicture}[scale=0.5]
\path coordinate[dot, label=left:$h$] (a)
 +(0,1) coordinate[label=above:$Y$] (t)
 +(0,-2) coordinate[label=below:$X$] (b)
 +(1,-1) coordinate[dot, label=below:$\eta$] (eta);
\draw (b) -- (a) -- (t)
 (eta) to[out=90, in=0] (a);
\begin{pgfonlayer}{background}
\fill[catterm] ($(t) + (-1,0)$) rectangle (b);
\fill[catc] (t) rectangle ($(eta) + (1,-1)$);
\end{pgfonlayer}
\end{tikzpicture}
\end{gathered}
\end{equation*}
In the upward direction we define the map:
\begin{equation*}
\begin{gathered}
\twocelldiag{f}{X}{Y}{catterm}{catc}{1}{1}{1.5}
\end{gathered}
\mapsto
\begin{gathered}
\begin{tikzpicture}[scale=0.5]
\path coordinate[dot, label=left:$f$] (f) 
 +(0,-1) coordinate[label=below:$X$] (bl)
 ++(0,1) coordinate[dot, label=left:$a$] (a)
 ++(0,1) coordinate[label=above:$Y$] (t)
 (bl) ++(1,0) coordinate[label=below:$T$] (br);
\draw (bl) -- (f) -- (a) -- (t)
 (br) to[out=90, in=0] (a);
\begin{pgfonlayer}{background}
\fill[catterm] ($(t) + (-1,0)$) rectangle (bl);
\fill[catc] (t) rectangle ($(br) + (1,0)$);
\end{pgfonlayer}
\end{tikzpicture}
\end{gathered}
\end{equation*}
That this results in a valid algebra homomorphism follows immediately as $a$ is an Eilenberg-Moore algebra and so:
\begin{equation*}
\begin{gathered}
\begin{tikzpicture}[scale=0.5]
\path coordinate[dot, label=left:$f$] (f) +(0,-1) coordinate[label=below:$X$] (bl) +(2,0) coordinate[dot, label=below:$\mu$] (mu)
 ++(0,2) coordinate[dot, label=left:$a$] (a) ++(0,1) coordinate[label=above:$Y$] (t)
 (mu) +(-1,-1) coordinate[label=below:$T$] (bm) +(1,-1) coordinate[label=below:$T$] (br);
\draw (bl) -- (f) -- (a) -- (t)
 (mu) to[out=90, in=0] (a)
 (bm) to[out=90, in=180] (mu.west) -- (mu.east) to[out=0, in=90] (br);
\begin{pgfonlayer}{background}
\fill[catterm] ($(t) + (-1,0)$) rectangle (bl);
\fill[catc] (t) rectangle ($(br) + (1,0)$);
\end{pgfonlayer}
\end{tikzpicture}
\end{gathered}
\stackrel{\eqref{eq:emmult}}{=}
\begin{gathered}
\begin{tikzpicture}[scale=0.5]
\path coordinate[dot, label=left:$f$] (f) +(0,-1) coordinate[label=below:$X$] (bl)
 ++(0,1) coordinate[dot, label=left:$a$] (a1) ++(0,1) coordinate[dot, label=left:$a$] (a2) ++(0,1) coordinate[label=above:$Y$] (t)
 (bl) ++(1,0) coordinate[label=below:$T$] (bm) ++(1,0) coordinate[label=below:$T$] (br);
\draw (bl) -- (f) -- (a1) -- (a2) -- (t)
 (bm) to[out=90, in=0] (a1)
 (br) to[out=90, in=0] (a2);
\begin{pgfonlayer}{background}
\fill[catterm] ($(t) + (-1,0)$) rectangle (bl);
\fill[catc] (t) rectangle ($(br) + (1,0)$);
\end{pgfonlayer}
\end{tikzpicture}
\end{gathered}
\end{equation*}
We now check these maps constitute a bijection. Firstly, as $a$ is an Eilenberg-Moore algebra homomorphism we immediately have:
\begin{equation*}
\begin{gathered}
\begin{tikzpicture}[scale=0.5]
\path coordinate[dot, label=left:$f$] (f) 
 +(0,-1) coordinate[label=below:$X$] (b)
 +(1,0) coordinate[dot, label=below:$\eta$] (eta)
 ++(0,1) coordinate[dot, label=left:$a$] (a)
 ++(0,1) coordinate[label=above:$Y$] (t);
\draw (b) -- (f) -- (a) -- (t)
 (eta) to[out=90, in=0] (a);
\begin{pgfonlayer}{background}
\fill[catterm] ($(t) + (-1,0)$) rectangle (b);
\fill[catc] (t) rectangle ($(eta) + (1,-1)$);
\end{pgfonlayer}
\end{tikzpicture}
\end{gathered}
\stackrel{\eqref{eq:emunit}}{=}
\begin{gathered}
\twocelldiag{f}{X}{Y}{catterm}{catc}{1}{1}{1.5}
\end{gathered}
\end{equation*}
In the opposite direction:
\begin{equation*}
\begin{gathered}
\begin{tikzpicture}[scale=0.5]
\path coordinate[dot, label=left:$h$] (h) +(1,-1) coordinate[dot, label=below:$\eta$] (eta)
 +(0,-2) coordinate[label=below:$X$] (bl) ++(0,1) coordinate[dot, label=left:$a$] (a) +(0,1) coordinate[label=above:$Y$] (t)
 (eta) +(1,-1) coordinate[label=below:$T$] (br);
\draw (bl) -- (h) -- (a) -- (t)
 (eta) to[out=90, in=0] (h)
 (br) to[out=90, in=0] (a);
\begin{pgfonlayer}{background}
\fill[catterm] ($(t) + (-1,0)$) rectangle (bl);
\fill[catc] (t) rectangle ($(br) + (1,0)$);
\end{pgfonlayer}
\end{tikzpicture}
\end{gathered}
\enskip=\enskip
\begin{gathered}
\begin{tikzpicture}[scale=0.5]
\path coordinate[dot, label=left:$h$] (h)
 +(0,1) coordinate[label=above:$Y$] (t)
 +(0,-3) coordinate[label=below:$X$] (bl)
 ++(2,-1) coordinate[dot, label=below:$\mu$] (mu)
 +(-1,-1) coordinate[dot, label=below:$\eta$] (eta)
 ++(1,-2) coordinate[label=below:$T$] (br);
\draw (bl) -- (h) -- (t)
 (eta) to[out=90, in=180] (mu)
 (br) to[out=90, in=0] (mu)
 (mu) to[out=90, in=0] (h);
\begin{pgfonlayer}{background}
\fill[catterm] ($(t) + (-1,0)$) rectangle (bl);
\fill[catc] (t) rectangle ($(br) + (1,0)$);
\end{pgfonlayer}
\end{tikzpicture}
\end{gathered}
\enskip\stackrel{\eqref{eq:monad-unit}}{=}\enskip
\begin{gathered}
\gmorph{h}{T}{X}{Y}{catterm}{catc}{catc}{2.0}{1}
\end{gathered}
\end{equation*}
The first equality above follows by the assumption $h$ a homomorphism from $\mu_X$ to $a$.

Checking the naturality of these mappings is left as an exercise. That $T$ is the monad induced by
this adjunction follows from the construction in lemma \ref{lem:adjmod}.
\end{proof}
\begin{proposition}
\label{prop:emlift}
Every monad morphism:
\begin{equation*}
\sigma : (\mathcal{C}, T, \eta, \mu) \rightarrow (\mathcal{C}, T', \eta', \mu')
\end{equation*}
induces a functor:
\begin{equation*}
\eilmo{\sigma} : \eilmo{T'} \rightarrow \eilmo{T}
\end{equation*}
\end{proposition}
\begin{proof}
On objects, a suitable functor acts as follows:
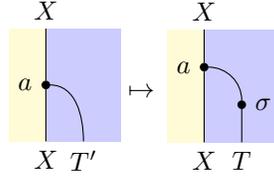
\begin{equation*}
\begin{gathered}
\gmorph{a}{T'}{X}{X}{catterm}{catc}{catc}{1.5}{1}
\end{gathered}
\mapsto
\begin{gathered}
\begin{tikzpicture}[scale=0.5]
\path coordinate[dot, label=left:$a$] (a) +(0,1) coordinate[label=above:$X$] (t)
 +(0,-2) coordinate[label=below:$X$] (bl)
 ++(1,-1) coordinate[dot, label=right:$\sigma$] (sigma)
 ++(0,-1) coordinate[label=below:$T$] (br);
\draw (bl) -- (a) -- (t)
 (br) -- (sigma) to[out=90, in=0] (a);
\begin{pgfonlayer}{background}
\fill[catterm] ($(t) + (-1,0)$) rectangle (bl);
\fill[catc] (t) rectangle ($(br) + (1,0)$);
\end{pgfonlayer}
\end{tikzpicture}
\end{gathered}
\end{equation*}
On morphisms, the functor acts as the identity.
We must show the resulting algebra satisfies the unit and multiplication axioms.
The unit axiom follows from the equalities:
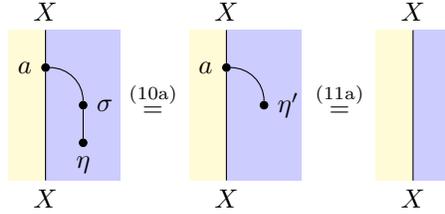
\begin{equation*}
\begin{gathered}
\begin{tikzpicture}[scale=0.5]
\path coordinate[dot, label=left:$a$] (a)
 +(0,1) coordinate[label=above:$X$] (t)
 +(0,-3) coordinate[label=below:$X$] (b)
 ++(1,-1) coordinate[dot, label=right:$\sigma$] (sigma)
 ++(0,-1) coordinate[dot, label=below:$\eta$] (eta);
\draw (b) -- (a) -- (t)
 (eta) -- (sigma) to[out=90, in=0] (a);
\begin{pgfonlayer}{background}
\fill[catterm] ($(t) + (-1,0)$) rectangle (b);
\fill[catc] (t) rectangle ($(eta) + (1,-1)$);
\end{pgfonlayer}
\end{tikzpicture}
\end{gathered}
\stackrel{\eqref{eq:mmunit}}{=}
\begin{gathered}
\begin{tikzpicture}[scale=0.5]
\path coordinate[dot, label=left:$a$] (a)
 +(0,1) coordinate[label=above:$X$] (t)
 +(0,-3) coordinate[label=below:$X$] (b)
 ++(1,-1) coordinate[dot, label=right:$\eta'$] (eta);
\draw (b) -- (a) -- (t)
 (eta) to[out=90, in=0] (a);
\begin{pgfonlayer}{background}
\fill[catterm] ($(t) + (-1,0)$) rectangle (b);
\fill[catc] (t) rectangle ($(eta) + (1,-2)$);
\end{pgfonlayer}
\end{tikzpicture}
\end{gathered}
\stackrel{\eqref{eq:emunit}}{=}
\begin{gathered}
\onecelldiag{X}{catterm}{catc}{4}{1}
\end{gathered}
\end{equation*}
That the multiplication axiom is satisfied is shown as follows:
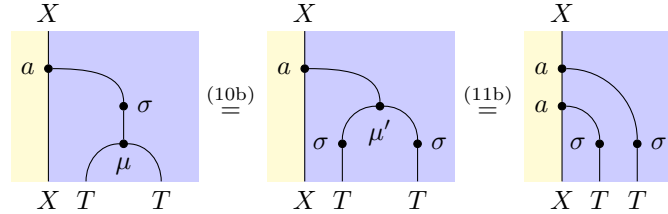
\begin{equation*}
\begin{gathered}
\begin{tikzpicture}[scale=0.5]
\path coordinate[dot, label=left:$a$] (a)
 +(0,1) coordinate[label=above:$X$] (t)
 +(0,-3) coordinate[label=below:$X$] (bl)
 ++(2,-1) coordinate[dot, label=right:$\sigma$] (sigma)
 ++(0, -1) coordinate[dot, label=below:$\mu$] (mu)
 +(-1,-1) coordinate[label=below:$T$] (bm)
 +(1,-1) coordinate[label=below:$T$] (br);
\draw (bl) -- (a) -- (t)
 (mu) -- (sigma) to[out=90, in=0] (a)
 (bm) to[out=90, in=180] (mu.west) -- (mu.east) to[out=0, in=90] (br);
\begin{pgfonlayer}{background}
\fill[catterm] ($(t) + (-1,0)$) rectangle (bl);
\fill[catc] (t) rectangle ($(br) + (1,0)$);
\end{pgfonlayer}
\end{tikzpicture}
\end{gathered}
\stackrel{\eqref{eq:mmmult}}{=}
\begin{gathered}
\begin{tikzpicture}[scale=0.5]
\path coordinate[dot, label=left:$a$] (a)
 +(0,1) coordinate[label=above:$X$] (t)
 +(0,-3) coordinate[label=below:$X$] (bl)
 ++(2,-1) coordinate[dot, label=below:$\mu'$] 
 (mu) ++(-1,-1) coordinate[dot, label=left:$\sigma$] (sigma1) ++(0,-1) coordinate[label=below:$T$] (bm)
 (mu) ++(1,-1) coordinate[dot, label=right:$\sigma$] (sigma2) ++(0,-1) coordinate[label=below:$T$] (br);
\draw (bl) -- (a) -- (t)
 (mu) to[out=90, in=0] (a)
 (bm) -- (sigma1) to[out=90, in=180] (mu.west) -- (mu.east) to[out=0, in=90] (sigma2) -- (br);
\begin{pgfonlayer}{background}
\fill[catterm] ($(t) + (-1,0)$) rectangle (bl);
\fill[catc] (t) rectangle ($(br) + (1,0)$);
\end{pgfonlayer}
\end{tikzpicture}
\end{gathered}
\stackrel{\eqref{eq:emmult}}{=}
\begin{gathered}
\begin{tikzpicture}[scale=0.5]
\path coordinate[dot, label=left:$a$] (a1)
 +(2,-2) coordinate[dot, label=right:$\sigma$] (sigma2)
 +(0,1) coordinate[label=above:$X$] (t)
 ++(0,-1) coordinate[dot, label=left:$a$] (a2)
 +(1,-1) coordinate[dot, label=left:$\sigma$] (sigma1)
 ++(0,-2) coordinate[label=below:$X$] (bl)
 (sigma1) ++(0,-1) coordinate[label=below:$T$] (bm)
 (sigma2) ++(0,-1) coordinate[label=below:$T$] (br);
\draw (bl) -- (a2) -- (a1) -- (t)
 (bm) -- (sigma1) to[out=90, in=0] (a2)
 (br) -- (sigma2) to[out=90, in=0] (a1);
\begin{pgfonlayer}{background}
\fill[catterm] ($(t) + (-1,0)$) rectangle (bl);
\fill[catc] (t) rectangle ($(br) + (1,0)$);
\end{pgfonlayer}
\end{tikzpicture}
\end{gathered}
\end{equation*}
That this is functorial is easy to show by applying a special case of example \ref{ex:alghom}.
\end{proof}

\begin{definition}[Kleisli Category]
For a monad $(\mathcal{C}, T, \eta, \mu)$ the \define{Kleisli category} \klei{T} is defined as follows:
\begin{itemize}
 \item {\bf Objects}: Objects of $\mathcal{C}$
 \item {\bf Morphisms}: A morphism of type $f: X \rightarrow Y$ in \klei{T} is a $\mathcal{C}$ morphism $f : X \rightarrow T(Y)$.
 Given another morphism $g:Y \rightarrow Z$ in \klei{T}, their (Kleisli) composite, written $g \kleicomp f$ is given
 by the following composite in $\mathcal{C}$:
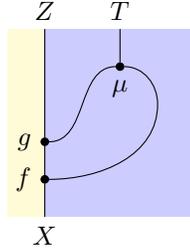
\begin{center}
\klcomp
\end{center}
The identity morphism on $X$ in \klei{T} is given by $\eta_X$, that this is a valid identity morphism
follows immediately from the monad axioms.
\end{itemize}
\end{definition}
We now see that the Kleisli category can be used to give an alternative construction of an adjunction
inducing a particular monad, as was done with the Eilenberg-Moore category in proposition \ref{prop:emadj}
\begin{proposition}[\cite{Kleisli1965}]
\label{prop:kladj}
For a monad $(\mathcal{C}, T, \eta, \mu)$ there are functors:
\begin{equation*}
V : \klei{T} \rightarrow \mathcal{C}
\end{equation*}
and:
\begin{equation*}
H : \mathcal{C} \rightarrow \klei{T}
\end{equation*}
with $H \dashv V$ and the monad induced as in lemma \ref{lem:adjmod} is equal to $T$.
\end{proposition}
\begin{proof}
We define:
\begin{align*}
V : \klei{T} &\rightarrow \mathcal{C}\\
X \mapsto T(X)\\
f : X \rightarrow T(X) &\mapsto \mu_Y \circ T(f)
\end{align*}
That this preserves identities is easy to check. For functoriality of composition, we have the following equalities:
\begin{eqproof*}
V(
\begin{gathered}
\fmorph{g}{T}{Y}{Z}{catterm}{catc}{catc}{1.5}{1}
\end{gathered}
\kleicomp
\begin{gathered}
\fmorph{f}{T}{X}{Y}{catterm}{catc}{catc}{1.5}{1}
\end{gathered}
)
\explain{definition of Kleisli Composition}
V(
\begin{gathered}
\klcomp
\end{gathered}
)
\explain{definition of $V$}
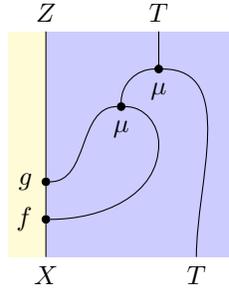
\begin{tikzpicture}[scale=0.5]
\path coordinate[dot, label=left:$f$] (f) +(0,-1) coordinate[label=below:$X$] (bl)
 ++(0,1) coordinate[dot, label=left:$g$] (g) +(0,4) coordinate[label=above:$Z$] (tl)
 ++(2,2) coordinate[dot, label=below:$\mu$] (mu1) ++(1,1) coordinate[dot, label=below:$\mu$] (mu2) ++(0,1) coordinate[label=above:$T$] (tr)
 (mu) ++(1,-1) coordinate (a)
 (bl) ++(4,0) coordinate[label=below:$T$] (br);
\draw (bl) -- (f) -- (g) -- (tl)
 (mu1) to[out=90, in=180] (mu2.west) -- (mu2.north) -- (tr)
 (mu2) to[out=0, in=90] (br)
 (f) to[out=0, in=-90] (a) to[out=90, in=0] (mu1)
 (g) to[out=0, in=180] (mu1);
\begin{pgfonlayer}{background}
\fill[catterm] ($(bl) + (-1,0)$) rectangle (tl);
\fill[catc] (bl) rectangle ($(tr) + (2,0)$);
\end{pgfonlayer}
\end{tikzpicture}
\explain{monad associativity axiom \eqref{eq:monad-associativity}}
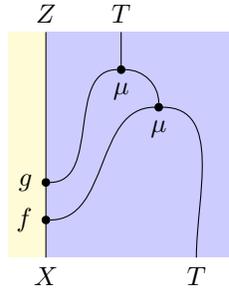
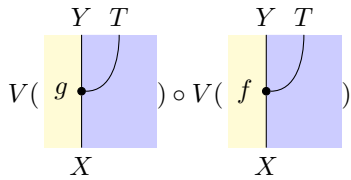
\begin{tikzpicture}[scale=0.5]
\path coordinate[dot, label=left:$f$] (f) +(0,-1) coordinate[label=below:$X$] (bl)
 ++(0,1) coordinate[dot, label=left:$g$] (g) +(0,4) coordinate[label=above:$Z$] (tl)
 ++(2,3) coordinate[dot, label=below:$\mu$] (mu1) ++(1,-1) coordinate[dot, label=below:$\mu$] (mu2) 
 (mu1) ++(0,1) coordinate[label=above:$T$] (tr)
 (mu) ++(1,-1) coordinate (a)
 (bl) ++(4,0) coordinate[label=below:$T$] (br);
\draw (bl) -- (f) -- (g) -- (tl)
 (mu1) to[out=0, in=90] (mu2)
 (mu1) -- (tr)
 (mu2) to[out=0, in=90] (br)
 (f) to[out=0, in=180] (mu2)
 (g) to[out=0, in=180] (mu1);
\begin{pgfonlayer}{background}
\fill[catterm] ($(bl) + (-1,0)$) rectangle (tl);
\fill[catc] (tl) rectangle ($(br) + (1,0)$);
\end{pgfonlayer}
\end{tikzpicture}
\explain{definition of $V$}
V(
\begin{gathered}
\fmorph{g}{T}{X}{Y}{catterm}{catc}{catc}{1.5}{1}
\end{gathered}
)
\circ
V(
\begin{gathered}
\fmorph{f}{T}{X}{Y}{catterm}{catc}{catc}{1.5}{1}
\end{gathered}
)
\end{eqproof*}
We also define:
\begin{align*}
H : \mathcal{C} &\rightarrow \klei{T}\\
X &\mapsto X\\
f : X \rightarrow Y &\mapsto \eta_Y \circ f
\end{align*}
Again that this preserves identities is easy to check. For functoriality of composition
we have:
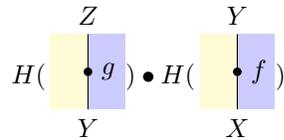
\begin{eqproof*}
H(
\begin{gathered}
\twocelldiag{g}{Y}{Z}{catterm}{catc}{1}{1}{1}
\end{gathered}
)
\kleicomp
H(
\begin{gathered}
  \twocelldiag{f}{X}{Y}{catterm}{catc}{1}{1}{1}
\end{gathered}
)
\explain{definition of $H$}
\begin{gathered}
\hmorph{g}{Y}{Z}
\end{gathered}
\kleicomp
\begin{gathered}
\hmorph{f}{X}{Y}
\end{gathered}
\explain{definition of Kleisli composition}
\begin{tikzpicture}[scale=0.5]
\path coordinate[dot, label=left:$f$] +(0,-1) coordinate[label=below:$X$] (b)
 ++(0,1) coordinate[dot, label=left:$g$] (g)
 ++(0,1) coordinate[label=above:$Z$] (tl)
 ++(2,0) coordinate[label=above:$T$] (tr)
 ++(0,-1) coordinate[dot, label=below:$\mu$] (mu)
 +(-1,-1) coordinate[dot, label=below:$\eta$] (eta1)
 +(1,-1) coordinate[dot, label=below:$\eta$] (eta2);
\draw (mu) -- (tr)
 (b) -- (f) -- (g) -- (tl)
 (eta1) to[out=90, in=180] (mu.west) -- (mu.east) to[out=0, in=90] (eta2);
\begin{pgfonlayer}{background}
\fill[catterm] ($(b) + (-1,0)$) rectangle (tl);
\fill[catc] (b) rectangle ($(tr) + (2,0)$);
\end{pgfonlayer}
\end{tikzpicture}
\explain{monad unit axiom \eqref{eq:monad-unit}}
\begin{tikzpicture}[scale=0.5]
\path coordinate[dot, label=left:$f$] (f) +(0,-1) coordinate[label=below:$X$] (b)
 ++(0,1) coordinate[dot, label=left:$g$] (g) ++(0,1) coordinate[label=above:$Z$] (tl)
 ++(1,0) coordinate[label=above:$T$] (tr)
 ++(0,-1) coordinate[dot, label=below:$\eta$] (eta);
\draw (eta) -- (tr)
 (b) -- (f) -- (g) -- (tl);
\begin{pgfonlayer}{background}
\fill[catterm] ($(b) + (-1,0)$) rectangle (tl);
\fill[catc] (b) rectangle ($(tr) + (1,0)$);
\end{pgfonlayer}
\end{tikzpicture}
\explain{definition of $H$}
H(
\begin{gathered}
\twocellpairdiag{f}{g}{X}{Z}{catterm}{catc}{1}{1}{1}
\end{gathered}
)
\end{eqproof*}
Now we require a bijection:
\begin{center}
\begin{prooftree}
\AxiomC{$H(X) \rightarrow Y$}
\doubleLine
\UnaryInfC{$X \rightarrow V(Y)$}
\end{prooftree}
\end{center}
Expanding definitions, and noting that a Kleisli morphism $X \rightarrow Y$ is a $\mathcal{C}$
morphism $X \rightarrow T(Y)$ by definition, the required bijection reduces to:
\begin{center}
\begin{prooftree}
\AxiomC{$X \rightarrow T(Y)$}
\doubleLine
\UnaryInfC{$X \rightarrow T(Y)$}
\end{prooftree}
\end{center}
Naturality of this bijection under the two forms of composition is easy to check. That we recover
the original monad from this adjunction follows from the construction in lemma \ref{lem:adjmod}.
\end{proof}
\begin{proposition}
\label{prop:kleislilift}
Every monad morphism:
\begin{equation*}
\sigma : (\mathcal{C}, T, \eta, \mu) \rightarrow (\mathcal{C}, T', \eta', \mu')
\end{equation*}
induces a functor:
\begin{equation*}
\klei{\sigma} : \klei{T} \rightarrow \klei{T'}
\end{equation*}
\end{proposition}
\begin{proof}
On objects, the functor acts as the identity. The action on morphisms is defined in $\mathcal{C}$
as follows:
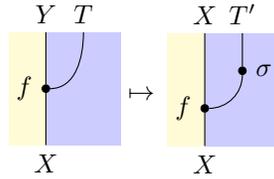
\begin{equation*}
\begin{gathered}
\fmorph{f}{T}{X}{Y}{catterm}{catc}{catc}{1.5}{1}
\end{gathered}
\mapsto
\begin{gathered}
\begin{tikzpicture}[scale=0.5]
\path coordinate[dot, label=left:$f$] (f) 
 +(0,2) coordinate[label=above:$X$] (tl)
 +(0,-1) coordinate[label=below:$X$] (b)
 ++(1,1) coordinate[dot, label=right:$\sigma$] (sigma) ++(0,1) coordinate[label=above:$T'$] (tr);
\draw (b) -- (f) -- (tl)
 (tr) -- (sigma) to[out=-90, in=0] (f);
\begin{pgfonlayer}{background}
\fill[catterm] ($(tl) + (-1,0)$) rectangle (b);
\fill[catc] (b) rectangle ($(tr) + (1,0)$);
\end{pgfonlayer}
\end{tikzpicture}
\end{gathered}
\end{equation*}
That this preserves identities is immediate from equation \eqref{eq:mmunit} as: 
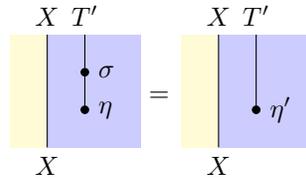
\begin{equation*}
\begin{gathered}
\begin{tikzpicture}[scale=0.5]
\path coordinate[dot, label=right:$\eta$] (eta) ++(0,1) coordinate[dot, label=right:$\sigma$] (sigma) ++(0,1) coordinate[label=above:$T'$] (t)
 (t) ++(-1,0) coordinate[label=above:$X$] (tl)
 (eta) ++(-1,-1) coordinate[label=below:$X$] (bl);
\draw (eta) -- (sigma) -- (t)
 (bl) -- (tl);
\begin{pgfonlayer}{background}
\fill[catterm] ($(bl) + (-1,0)$) rectangle (tl);
\fill[catc] (bl) rectangle ($(t) + (1.5,0)$);
\end{pgfonlayer}
\end{tikzpicture}
\end{gathered}
=
\begin{gathered}
\begin{tikzpicture}[scale=0.5]
\path coordinate[dot, label=right:$\eta'$] (eta) ++(0,2) coordinate[label=above:$T'$] (t)
 (t) ++(-1,0) coordinate[label=above:$X$] (tl)
 (eta) ++(-1,-1) coordinate[label=below:$X$] (bl);
\draw (eta) -- (t)
 (bl) -- (tl);
\begin{pgfonlayer}{background}
\fill[catterm] ($(bl) + (-1,0)$) rectangle (tl);
\fill[catc] (bl) rectangle ($(t) + (1.5,0)$);
\end{pgfonlayer}
\end{tikzpicture}
\end{gathered}
\end{equation*}
Similarly, composition follows from equation \eqref{eq:mmmult} as:
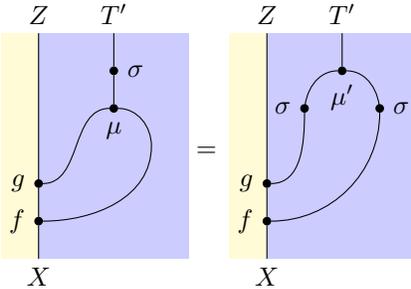
\begin{equation*}
\begin{gathered}
\begin{tikzpicture}[scale=0.5]
\path coordinate[dot, label=left:$f$] (f) +(0,-1) coordinate[label=below:$X$] (b)
 ++(0,1) coordinate[dot, label=left:$g$] (g)
 +(0,4) coordinate[label=above:$Z$] (tl)
 ++(2,2) coordinate[dot, label=below:$\mu$] (mu)
 +(0,1) coordinate[dot, label=right:$\sigma$] (sigma)
 +(1,-1) coordinate (a)
 (sigma) ++(0,1) coordinate[label=above:$T'$] (tr);
\draw (b) -- (f) -- (g) -- (tl)
 (g) to[out=0, in=180] (mu)
 (f) to[out=0, in=-90] (a) to[out=90, in=0] (mu)
 (mu) -- (sigma) -- (tr);
\begin{pgfonlayer}{background}
\fill[catterm] ($(b) + (-1,0)$) rectangle (tl);
\fill[catc] (b) rectangle ($(tr) + (2,0)$);
\end{pgfonlayer}
\end{tikzpicture}
\end{gathered}
=
\begin{gathered}
\begin{tikzpicture}[scale=0.5]
\path coordinate[dot, label=left:$f$] (f) +(0,-1) coordinate[label=below:$X$] (b)
 ++(0,1) coordinate[dot, label=left:$g$] (g)
 +(0,4) coordinate[label=above:$Z$] (tl)
 (g) ++(1,2) coordinate[dot, label=left:$\sigma$] (sigma1) ++(1,1) coordinate[dot, label=below:$\mu'$] (mu) ++(0,1) coordinate[label=above:$T'$] (tr)
 (f) ++(3,3) coordinate[dot, label=right:$\sigma$] (sigma2);
\draw (b) -- (f) -- (g) -- (tl)
 (g) to[out=0, in=-90] (sigma1) to[out=90, in=180] (mu)
 (f) to[out=0, in=-90] (sigma2) to[out=90, in=0] (mu)
 (mu) -- (tr);
\begin{pgfonlayer}{background}
\fill[catterm] ($(b) + (-1,0)$) rectangle (tl);
\fill[catc] (b) rectangle ($(tr) + (2,0)$);
\end{pgfonlayer}
\end{tikzpicture}
\end{gathered}
\end{equation*}
\end{proof}
\subsection{Categories of Monads}
\label{sec:moncat}
We now consider some appropriate types of morphisms between monads on different base categories.
\begin{definition}
The category \defcmonad is defined as having:
\begin{itemize}
 \item {\bf Objects}: Monads
 \item {\bf Morphisms}: A morphism of type $(\mathcal{C}, T, \eta, \mu) \rightarrow (\mathcal{C}', T', \eta', \mu')$
  is a pair consisting of a functor $F : \mathcal{C} \rightarrow \mathcal{C}'$ and a natural transformation $f : T' F \Rightarrow F T$
  satisfying the following equations:
\begin{subequations}
\begin{trivlist}
\item
\begin{minipage}{0.445\linewidth}
\begin{equation}
\label{eq:mcatmor-unit}
\begin{gathered}
\unitpulledleft{\eta}{T}{F}{catc}{catc'}
\end{gathered}
=
\begin{gathered}
\unitpullleft{\eta'}{f}{T}{F}{catc}{catc'}
\end{gathered}
\end{equation}
\end{minipage}
\begin{minipage}{0.545\linewidth}
\begin{equation}
\label{eq:mcatmor-mult}
\begin{gathered}
\multpulledleft{\mu}{f}{T}{T'}{F}{catc}{catc'}
\end{gathered}
=
\begin{gathered}
\multpullleft{\mu'}{f}{T}{T'}{F}{catc}{catc'}
\end{gathered}
\end{equation}
\end{minipage}
\end{trivlist}
\end{subequations}

\end{itemize}
\end{definition}
\begin{proposition}
\label{prop:cmonadlift}
Every \refcmonad morphism:
\begin{equation*}
(\mathcal{C}, T, \eta, \mu) \rightarrow (\mathcal{C}', T', \eta', \mu')
\end{equation*}
induces a functor:
\begin{equation*}
\eilmo{T} \rightarrow \eilmo{T'}
\end{equation*}
\end{proposition}
\begin{proof}
For \refcmonad morphism $(F, f)$ the action on objects of our functor is given by:
\begin{equation*}
\begin{gathered}
\gmorph{a}{T}{X}{X}{catterm}{catc}{catc}{1.5}{1}
\end{gathered}
\mapsto
\begin{gathered}
\begin{tikzpicture}[scale=0.5]
\path coordinate[dot, label=left:$a$] (a) +(0,1) coordinate[label=above:$X$] (tl)
 ++(-0,-2) coordinate[label=below:$X$] (bl) ++(1,0) coordinate[label=below:$F$] (bm) ++(1,0) coordinate[label=below:$T'$] (br)
 (a) ++(3,0) coordinate (r) ++(0,1) coordinate[label=above:$F$] (tr);
\draw (bl) -- (a) -- (tl);
\draw[name path=vert] (br) to[out=90, in=0] (a);
\draw[name path=curv] (bm) to[out=90, in=-90] (r) -- (tr);
\path[name intersections={of=vert and curv}] coordinate[dot, label={[label distance=0.25cm]right:$f$}] (sigma) at (intersection-1);
\begin{pgfonlayer}{background}
\fill[catterm] ($(bl) + (-1,0)$) rectangle (tl);
\fill[catc] (bl) rectangle ($(tr) + (1,0)$);
\fill[catc'] (bm) to[out=90, in=-90] (r) -- (tr) -- ++(1,0) -- ++(0,-3) -- cycle;
\end{pgfonlayer}
\end{tikzpicture}
\end{gathered}
\end{equation*}
That this extends to a functor to the category of $T'$-algebras follows from example \ref{ex:alghom}.
We must then confirm that the resulting algebra is in \eilmo{T'}. For the unit axiom we have:
\begin{equation*}
\begin{gathered}
\begin{tikzpicture}[scale=0.5]
\path coordinate[dot, label=left:$a$] (a) +(0,1) coordinate[label=above:$X$] (tl)
 +(0,-3) coordinate[label=below:$X$] (bl)
 ++(2,-2) coordinate[dot, label=right:$\eta'$] (eta)
 ++(-1,0) coordinate (l) ++(0,-1) coordinate[label=below:$F$] (br)
 (tl) ++(3,0) coordinate[label=above:$F$] (tr);
\draw (bl) -- (a) -- (tl);
\draw[name path=vert] (eta) to[out=90, in=0] (a);
\draw[name path=curv] (br) -- (l) to[out=90, in=-90] (tr);
\path[name intersections={of=vert and curv}] coordinate[dot, label=right:$f$] (sigma) at (intersection-1);
\begin{pgfonlayer}{background}
\fill[catterm] ($(bl) + (-1,0)$) rectangle (tl);
\fill[catc] (bl) rectangle ($(tr) + (1,0)$);
\fill[catc'] (br) -- (l) to[out=90, in=-90] (tr) -- ++(1,0) -- ++(0,-4) -- cycle;
\end{pgfonlayer}
\end{tikzpicture}
\end{gathered}
\stackrel{\eqref{eq:mcatmor-unit}}{=}
\begin{gathered}
\begin{tikzpicture}[scale=0.5]
\path coordinate[dot, label=left:$a$] (a) +(0,1) coordinate[label=above:$X$] (tl)
 +(0,-3) coordinate[label=below:$X$] (bl)
 ++(2,-1) coordinate[dot, label=left:$\eta$] (eta)
 (bl) ++(1,0) coordinate[label=below:$F$] (br) ++(0,1) coordinate (l) ++(2,1) coordinate (r) ++(0,2) coordinate[label=above:$F$] (tr);
\draw (bl) -- (a) -- (tl)
 (eta) to[out=90, in=0] (a)
 (br) -- (l) to[out=90, in=-90] (r) -- (tr);
\begin{pgfonlayer}{background}
\fill[catterm] ($(bl) + (-1,0)$) rectangle (tl);
\fill[catc] (bl) rectangle ($(tr) + (1,0)$);
\fill[catc'] (br) -- (l) to[out=90, in=-90] (r) -- (tr) -- ++(1,0) -- ++(0,-4) -- cycle;
\end{pgfonlayer}
\end{tikzpicture}
\end{gathered}
\stackrel{\eqref{eq:emunit}}{=}
\begin{gathered}
\begin{tikzpicture}[scale=0.5]
\path coordinate[label=below:$X$] (bl) +(1,0) coordinate[label=below:$F$] (br)
 ++(0,4) coordinate[label=above:$X$] (tl) ++(1,0) coordinate[label=above:$F$] (tr);
\draw (bl) -- (tl)
 (br) -- (tr);
\begin{pgfonlayer}{background}
\fill[catterm] ($(bl) + (-1,0)$) rectangle (tl);
\fill[catc] (bl) rectangle (tr);
\fill[catc'] (br) rectangle ($(tr) + (1,0)$);
\end{pgfonlayer}
\end{tikzpicture}
\end{gathered}
\end{equation*}
For the multiplication axiom:
\begin{equation*}
\begin{gathered}
\begin{tikzpicture}[scale=0.5]
\path coordinate[dot, label=left:$a$] (a) +(0,1) coordinate[label=above:$X$] (tl)
 +(0,-4) coordinate[label=below:$X$] (bl)
 ++(3,-3) coordinate[dot, label={[label distance=0.25cm]right:$\mu'$}] (mu)
 +(-1,-1) coordinate[label=below:$T'$] (bm)
 +(1,-1) coordinate[label=below:$T'$] (br);
\path (bl) ++(1,0) coordinate[label=below:$F$] (cb) ++(0,1) coordinate (l)
 ++(4,1.5) coordinate (r) ++(0,2.5) coordinate[label=above:$F$] (ct);
\draw (bl) -- (a) -- (tl);
\draw[name path=vert] (mu) to[out=90, in=0] (a);
\draw (bm) to[out=90, in=180] (mu.west) -- (mu.east) to[out=0, in=90] (br);
\draw[name path=curv] (cb) -- (l) to[out=90, in=-90] (r) -- (ct);
\path[name intersections={of=vert and curv}] coordinate[dot, label=45:$f$] (sigma) at (intersection-1);
\begin{pgfonlayer}{background}
\fill[catterm] ($(bl) + (-1,0)$) rectangle (tl);
\fill[catc] (bl) rectangle ($(ct) + (0.5,0)$);
\fill[catc'] (cb) -- (l) to[out=90, in=-90] (r) -- (ct) -- ++(0.5,0) -- ++(0,-5) -- cycle;
\end{pgfonlayer}
\end{tikzpicture}
\end{gathered}
\stackrel{\eqref{eq:mcatmor-mult}}{=}
\begin{gathered}
\begin{tikzpicture}[scale=0.5]
\path coordinate[dot, label=left:$a$] (a) +(0,1) coordinate[label=above:$X$] (tl)
 +(0,-4) coordinate[label=below:$X$] (bl)
 ++(3,-1) coordinate[dot, label=below:$\mu$] (mu)
 +(-1,-3) coordinate[label=below:$T'$] (bm)
 +(1,-3) coordinate[label=below:$T'$] (br);
\path (bl) ++(1,0) coordinate[label=below:$F$] (cb) ++(0,1) coordinate (l)
 ++(4,1.5) coordinate (r) ++(0,2.5) coordinate[label=above:$F$] (ct);
\draw (bl) -- (a) -- (tl);
\draw (mu) to[out=90, in=0] (a);
\draw[name path=vert1] (bm) to[out=90, in=180] (mu.west); 
\draw[name path=vert2] (br) to[out=90, in=0] (mu.east);
\draw[name path=curv] (cb) -- (l) to[out=90, in=-90] (r) -- (ct);
\path[name intersections={of=vert1 and curv}] coordinate[dot, label=-45:$f$] (sigma1) at (intersection-1);
\path[name intersections={of=vert2 and curv}] coordinate[dot, label=-45:$f$] (sigma2) at (intersection-1);
\begin{pgfonlayer}{background}
\fill[catterm] ($(bl) + (-1,0)$) rectangle (tl);
\fill[catc] (bl) rectangle ($(ct) + (0.5,0)$);
\fill[catc'] (cb) -- (l) to[out=90, in=-90] (r) -- (ct) -- ++(0.5,0) -- ++(0,-5) -- cycle;
\end{pgfonlayer}
\end{tikzpicture}
\end{gathered}
\stackrel{\eqref{eq:emmult}}{=}
\begin{gathered}
\begin{tikzpicture}[scale=0.5]
\path coordinate[dot, label=left:$a$] (a2) +(0,1) coordinate[label=above:$X$] (tl)
 ++(0,-1) coordinate[dot, label=left:$a$] (a1) ++(0,-3) coordinate[label=below:$X$] (bl)
 ++(2,0) coordinate[label=below:$T'$] (bm) ++(2,0) coordinate[label=below:$T'$] (br);
\path (bl) ++(1,0) coordinate[label=below:$F$] (cb) ++(0,1) coordinate (l)
 ++(4,1.5) coordinate (r) ++(0,2.5) coordinate[label=above:$F$] (ct);
\draw (bl) -- (a1) -- (a2) -- (tl);
\draw[name path=vert1] (bm) to[out=90, in=0] (a1);
\draw[name path=vert2] (br) to[out=90, in=0] (a2);
\draw[name path=curv] (cb) -- (l) to[out=90, in=-90] (r) -- (ct);
\path[name intersections={of=vert1 and curv}] coordinate[dot, label=-45:$f$] (sigma1) at (intersection-1);
\path[name intersections={of=vert2 and curv}] coordinate[dot, label=-45:$f$] (sigma2) at (intersection-1);
\begin{pgfonlayer}{background}
\fill[catterm] ($(bl) + (-1,0)$) rectangle (tl);
\fill[catc] (bl) rectangle ($(ct) + (0.5,0)$);
\fill[catc'] (cb) -- (l) to[out=90, in=-90] (r) -- (ct) -- ++(0.5,0) -- ++(0,-5) -- cycle;
\end{pgfonlayer}
\end{tikzpicture}
\end{gathered}
\end{equation*}
\end{proof}
\begin{remark}
Every monad morphism (definition \ref{def:monadmorphism}) is an endomorphism in \refcmonad with the functor
part the identity. Proposition \ref{prop:emlift} is then a special case of proposition \ref{prop:cmonadlift}.
\end{remark}
\begin{definition}
The category \defcmonadstar is defined as having:
\begin{itemize}
 \item {\bf Objects}: Monads
 \item {\bf Morphism}: A morphism of type $(\mathcal{C}, T, \eta, \mu) \rightarrow (\mathcal{C}', T', \eta', \mu')$
  is a pair consisting of a functor $F: \mathcal{C} \rightarrow \mathcal{C}'$ and a natural transformation $f : F T \Rightarrow T' F$
  satisfying the following equations:
\begin{subequations}
\begin{trivlist}
\item
\begin{minipage}{0.445\linewidth}
\begin{equation}
\label{eq:mstarcatmor-unit}
\begin{gathered}
\unitpullright{\eta}{f}{T'}{F}{catc}{catc'}
\end{gathered}
=
\begin{gathered}
\unitpulledright{\eta'}{T'}{F}{catc}{catc'}
\end{gathered}
\end{equation}
\end{minipage}
\begin{minipage}{0.545\linewidth}
\begin{equation}
\label{eq:mstarcatmor-mult}
\begin{gathered}
\multpullright{\mu}{f}{T'}{T}{F}{catc}{catc'}
\end{gathered}
=
\begin{gathered}
\multpulledright{\mu'}{f}{T'}{T}{F}{catc}{catc'}
\end{gathered}
\end{equation}
\end{minipage}
\end{trivlist}
\end{subequations}
\end{itemize}
\end{definition}
\begin{proposition}
\label{prop:cmonadstarlift}
Every \refcmonadstar morphism:
\begin{equation*}
(\mathcal{C}, T, \eta, \mu) \rightarrow (\mathcal{C}', T', \eta', \mu')
\end{equation*}
induces a functor:
\begin{equation*}
\klei{T} \rightarrow \klei{T'}
\end{equation*}
\end{proposition}
\begin{proof}
On objects, the induced functor maps $X$ to $F(X)$. On morphisms the action is:
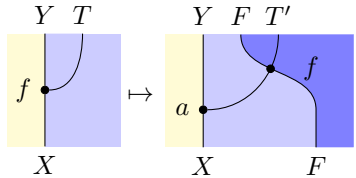
\begin{equation*}
\begin{gathered}
\fmorph{f}{T}{X}{Y}{catterm}{catc}{catc}{1.5}{1}
\end{gathered}
\mapsto
\begin{gathered}
\begin{tikzpicture}[scale=0.5]
\path coordinate[dot, label=left:$a$] (a) +(0,-1) coordinate[label=below:$X$] (bl)
 +(0,2) coordinate[label=above:$Y$] (tl)
 ++(2,2) coordinate[label=above:$T'$] (tr)
 (tl) ++(1,0) coordinate[label=above:$F$] (tm)
 (bl) ++(3,0) coordinate[label=below:$F$] (br) ++(0,1) coordinate (r);
\draw (bl) -- (a) -- (tl);
\draw[name path=vert] (a) to[out=0, in=-90] (tr);
\draw[name path=curv] (br) -- (r) to[out=90, in=-90] (tm);
\path[name intersections={of=vert and curv}] coordinate[dot, label={[label distance=0.25cm]right:$f$}] (f) at (intersection-1);
\begin{pgfonlayer}{background}
\fill[catterm] ($(tl) + (-1,0)$) rectangle (bl);
\fill[catc] (tl) rectangle ($(br) + (1,0)$);
\fill[catc'] (br) -- (r) to[out=90, in=-90] (tm) -- ++(3,0) -- ++(0,-3) -- cycle;
\end{pgfonlayer}
\end{tikzpicture}
\end{gathered}
\end{equation*}
That this preserves identities follows immediately from the \refcmonadstar axiom \eqref{eq:mstarcatmor-unit} as:
\begin{equation*}
\begin{gathered}
\begin{tikzpicture}[scale=0.5]
\path coordinate[dot, label=below:$\eta$] (eta) ++(0,2) coordinate[label=above:$T'$] (tr)
 ++(-1,0) coordinate[label=above:$F$] (tm)
 ++(-1,0) coordinate[label=above:$X$] (tl)
 ++(0,-3) coordinate[label=below:$X$] (bl)
 ++(3,0) coordinate[label=below:$F$] (br) ++(0,1) coordinate (r);
\draw (bl) -- (tl);
\draw[name path=vert] (eta) -- (tr);
\draw[name path=curv] (br) -- (r) to[out=90, in=-90] (tm);
\path[name intersections={of=vert and curv}] coordinate[dot, label={[label distance=0.25cm]right:$f$}] (f) at (intersection-1);
\begin{pgfonlayer}{background}
\fill[catterm] ($(tl) + (-1,0)$) rectangle (bl);
\fill[catc] (tl) rectangle ($(br) + (1,0)$);
\fill[catc'] (br) -- (r) to[out=90, in=-90] (tm) -- ++(3,0) -- ++(0,-3) -- cycle;
\end{pgfonlayer}
\end{tikzpicture}
\end{gathered}
=
\begin{gathered}
\begin{tikzpicture}[scale=0.5]
\path coordinate[dot, label=right:$\eta'$] (eta') ++(0,1) coordinate[label=above:$T'$] (tr)
 ++(-1,0) coordinate[label=above:$F$] (tm) +(0,-1) coordinate (l)
 ++(-1,0) coordinate[label=above:$X$] (tl)
 ++(0,-3) coordinate[label=below:$X$] (bl)
 ++(3,0) coordinate[label=below:$F$] (br);
\draw (bl) -- (tl)
 (eta) -- (tr)
 (br) to[out=90, in=-90] (l) -- (tm);
\begin{pgfonlayer}{background}
\fill[catterm] ($(tl) + (-1,0)$) rectangle (bl);
\fill[catc] (tl) rectangle ($(br) + (1,0)$);
\fill[catc'] (br) to[out=90, in=-90] (l) -- (tm) -- ++(3,0) -- ++(0,-3) -- cycle;
\end{pgfonlayer}
\end{tikzpicture}
\end{gathered}
\end{equation*}
Functoriality of composition also follows immediately from the \refcmonadstar axiom \eqref{eq:mstarcatmor-mult} as we have:
\begin{equation*}
\begin{gathered}
\begin{tikzpicture}[scale=0.5]
\path coordinate[dot, label=left:$p$] (p) +(0,-1) coordinate[label=below:$X$] (bl)
 ++(0,1) coordinate[dot, label=left:$q$] (q) +(0,4) coordinate[label=above:$Z$] (tl)
 ++(3,2) coordinate[dot, label=below:$\mu$] (mu) +(0,2) coordinate[label=above:$T'$] (tr)
 ++(1,-1) coordinate (a);
\path (tl) ++(1,0) coordinate[label=above:$F$] (tm) ++(0,-1)
 (bl) ++(5,0) coordinate[label=below:$F$] (br) ++(0,4) coordinate (r);
\draw (bl) -- (p) -- (q) -- (tl);
\draw (p) to[out=0, in=-90] (a) to[out=90, in=0] (mu.east);
\draw (q) to[out=0, in=180] (mu.west);
\draw[name path=vert] (mu) -- (tr);
\draw[name path=curv] (br) -- (r) to[out=90, in=-90] (tm);
\path[name intersections={of=vert and curv}] coordinate[dot, label=45:$f$] (f) at (intersection-1);
\begin{pgfonlayer}{background}
\fill[catterm] ($(tl) + (-1,0)$) rectangle (bl);
\fill[catc] (tl) rectangle ($(br) + (1,0)$);
\fill[catc'] (br) -- (r) to[out=90, in=-90] (tm) -- ++(5,0) -- ++(0,-6) -- cycle;
\end{pgfonlayer}
\end{tikzpicture}
\end{gathered}
=
\begin{gathered}
\begin{tikzpicture}[scale=0.5]
\path coordinate[dot, label=left:$p$] (p) +(0,-1) coordinate[label=below:$X$] (bl)
 ++(0,1) coordinate[dot, label=left:$q$] (q) +(0,4) coordinate[label=above:$Z$] (tl)
 ++(3,3) coordinate[dot, label=below:$\mu'$] (mu) +(0,1) coordinate[label=above:$T'$] (tr)
 ++(1,-1) coordinate (a);
\path (tl) ++(1,0) coordinate[label=above:$F$] (tm) ++(0,-1) coordinate (l)
 (bl) ++(5,0) coordinate[label=below:$F$] (br);
\draw (bl) -- (p) -- (q) -- (tl);
\draw[name path=vert2] (p) to[out=0, in=-90] (a) to[out=90, in=0] (mu.east);
\draw[name path=vert1] (q) to[out=0, in=180] (mu.west);
\draw (mu) -- (tr);
\draw[name path=curv] (br) to[out=90, in=-90] (l) -- (tm);
\path[name intersections={of=vert1 and curv}] coordinate[dot, label=left:$f$] (f1) at (intersection-1);
\path[name intersections={of=vert2 and curv}] coordinate[dot, label={[label distance=0.25cm]right:$f$}] (f2) at (intersection-1);
\begin{pgfonlayer}{background}
\fill[catterm] ($(tl) + (-1,0)$) rectangle (bl);
\fill[catc] (tl) rectangle ($(br) + (1,0)$);
\fill[catc'] (br) to[out=90, in=-90] (l) -- (tm) -- ++(5,0) -- ++(0,-6) -- cycle;
\end{pgfonlayer}
\end{tikzpicture}
\end{gathered}
\end{equation*}
\end{proof}
\begin{remark}
Every monad morphism (definition \ref{def:monadmorphism}) is an endomorphism in \refcmonadstar
with the functor part the identity. Proposition \ref{prop:kleislilift} is then a special case
of proposition \ref{prop:cmonadstarlift}.
\end{remark}
\begin{remark}
The categories \refcmonad and \refcmonadstar can be described more generally for an arbitrary 2-category,
and in fact can be given the structure of 2-categories themselves, although we will not require this additional
structure for our examples. 
For more details, see \citep{Street1972} and \citep{StreetLack2002} 
and also the excellent exposition in the early parts of \cite{PowerWatanabe2002}.
\end{remark}

\subsection{Distributive Laws}
In lemma \ref{lem:adjcomp} we saw that adjunctions compose in a straightforward manner. Monads do
not in general compose, but they can be composed in the presence of a suitable mediating natural
transformation, referred to as a distributive law.
\begin{definition}[Distributive Law]
\label{def:distrlaw}
For monads $(\mathcal{C}, T, \eta, \mu)$ and $(\mathcal{C}, T', \eta', \mu')$ 
a \define{distributive law} \citep{Beck1969} is a natural
transformation:
\begin{equation}
\label{eq:distlaw}
\begin{gathered}
\begin{tikzpicture}[scale=0.5]
\path coordinate[dot, label=below:$\delta$] (delta)
 +(-1,1) coordinate[label=above:$T$] (tl) +(1,1) coordinate[label=above:$T'$] (tr)
 +(-1,-1) coordinate[label=below:$T'$] (bl) +(1,-1) coordinate[label=below:$T$] (br);
\draw (bl) to[out=90, in=180] (delta.west) -- (delta.east) to[out=0, in=90] (br)
 (tl) to[out=-90, in=180] (delta.west) -- (delta.east) to[out=0, in=-90] (tr);
\begin{pgfonlayer}{background}
\fill[catc] ($(tl) + (-1,0)$) rectangle ($(br) + (1,0)$);
\end{pgfonlayer}
\end{tikzpicture}
\end{gathered}
\end{equation}
satisfying the following equations:
\begin{subequations}
\begin{trivlist}
\item
\begin{minipage}{0.495\linewidth}
\begin{equation}
\label{eq:distretal}
\begin{gathered}
\unitpullright{\eta'}{\delta}{T'}{T}{catc}{catc}
\end{gathered}
=
\begin{gathered}
\unitpulledright{\eta'}{T'}{T}{catc}{catc}
\end{gathered}
\end{equation}
\end{minipage}
\begin{minipage}{0.495\linewidth}
\begin{equation}
\label{eq:distretar}
\begin{gathered}
\unitpullleft{\eta}{\delta}{T}{T'}{catc}{catc}
\end{gathered}
=
\begin{gathered}
\unitpulledleft{\eta}{T}{T'}{catc}{catc}
\end{gathered}
\end{equation}
\end{minipage}
\item
\begin{minipage}{0.495\linewidth}
\begin{equation}
\label{eq:distrmul}
\begin{gathered}
\multpullright{\mu'}{\delta}{T'}{T'}{T}{catc}{catc}
\end{gathered}
=
\begin{gathered}
\multpulledright{\mu'}{\delta}{T'}{T'}{T}{catc}{catc}
\end{gathered}
\end{equation}
\end{minipage}
\begin{minipage}{0.495\linewidth}
\begin{equation}
\label{eq:distrmur}
\begin{gathered}
\multpullleft{\mu}{\delta}{T}{T}{T'}{catc}{catc}
\end{gathered}
=
\begin{gathered}
\multpulledleft{\mu}{\delta}{T}{T}{T'}{catc}{catc}
\end{gathered}
\end{equation}
\end{minipage}
\end{trivlist}
\end{subequations}
\end{definition}
\begin{remark}[Artistic Values]
The axioms presented for a distributive law in definition \ref{def:distrlaw} perhaps best
illustrate the importance of how we choose to draw our string diagrams. In the axioms of equations
\eqref{eq:distretal}, \eqref{eq:distretar}, \eqref{eq:distrmul} and \eqref{eq:distrmur} we choose
to draw the natural transformation $\delta$ in different ways. The different choices of
how $\delta$ is drawn allow us to emphasise the nature of the transformations as ``sliding''
various $\eta$ and $\epsilon$ natural transformations across an appropriate line in the diagram.
For example, we could instead have persisted with the neutral depiction of diagram \eqref{eq:distlaw},
and drawn axiom \eqref{eq:distrmur} as:
\begin{equation*}
\begin{gathered}
\begin{tikzpicture}[scale=0.5]
\path coordinate[dot, label=$\delta$] (dl)
 +(-1,-1) coordinate[label=below:$T'$] (bl)
 +(1,-1) coordinate[label=below:$T$] (bm)
 +(-1,1) coordinate (a)
 ++(1,1) coordinate (b) ++(1,1) coordinate[dot, label=above:$\delta$] (dr)
 +(1,-1) coordinate (c)
 +(1,1) coordinate (d)
 ++(-1,1) coordinate (e) ++(-1,1) coordinate[dot, label=below:$\mu$] (mu)
 +(-1,-1) coordinate (f)
 +(0,1) coordinate[label=above:$T$] (tl)
 (d) ++(0,2) coordinate[label=above:$T'$] (tr)
 (c) ++(0,-2) coordinate[label=below:$T$] (br);
\draw (bl) to[out=90, in=180] (dl.west) -- (dl.east) to[out=0, in=-90] (b) to[out=90, in=180] (dr.west) -- (dr.east) to[out=0, in=-90] (d) -- (tr)
 (bm) to[out=90, in=0] (dl)
 (br) -- (c) to[out=90, in=0] (dr)
 (dl.west) to[out=180, in=-90] (a) -- (f) to[out=90, in=180] (mu.west) -- (mu.north) -- (tl)
 (dr.west) to[out=180, in=-90] (e) to[out=90, in=0] (mu.east);
\begin{pgfonlayer}{background}
\fill[catc] ($(bl) + (-0.5,0)$) rectangle ($(tr) + (0.5,0)$);
\fill[catc] (bl) to[out=90, in=180] (dl.west) -- (dl.east) to[out=0, in=-90] (b) to[out=90, in=180] (dr.west) -- (dr.east) to[out=0, in=-90] 
 (d) -- (tr) -- ++(0.5, 0) -- ($(br) + (0.5,0)$) -- cycle;
\end{pgfonlayer}
\end{tikzpicture}
\end{gathered}
=
\begin{gathered}
\begin{tikzpicture}[scale=0.5]
\path coordinate[dot, label=45:$\mu$] (mu)
 +(-1,-1) coordinate[label=below:$T$] (bm)
 +(1,-1) coordinate[label=below:$T$] (br)
 ++(0,1) coordinate (a) ++(-1,1) coordinate[dot, label=below:$\delta$] (delta)
 +(-1,1) coordinate (b)
 +(1,1) coordinate (c)
 ++(-1,-1) coordinate (d) ++(0,-2) coordinate[label=below:$T'$] (bl)
 (b) ++(0,2) coordinate[label=above:$T$] (tl)
 (c) ++(0,2) coordinate[label=above:$T'$] (tr);
\draw (bl) -- (d) to[out=90, in=180] (delta.west) -- (delta.east) to[out=0, in=-90] (c) -- (tr)
 (mu.north) -- (a) to[out=90, in=0] (delta.east) -- (delta.west) to[out=180, in=-90] (b) -- (tl)
 (bm) to[out=90, in=180] (mu.west) -- (mu.east) to[out=0, in=90] (br);
\begin{pgfonlayer}{background}
\fill[catc] ($(bl) + (-0.5,0)$) rectangle ($(tr) + (0.5,0)$);
\fill[catc] (bl) -- (d) to[out=90, in=180] (delta.west) -- (delta.east) to[out=0, in=-90] (c) -- (tr)
 -- ++(1.5,0) -- ($(br) + (0.5,0)$) -- cycle;
\end{pgfonlayer}
\end{tikzpicture}
\end{gathered}
\end{equation*}
These diagrams describe exactly the same relationship, but the visual intuition for the nature of the
axiom is completely lost.
\end{remark}

\begin{proposition}
Given monads $(\mathcal{C}, T, \eta, \mu)$ and $(\mathcal{C}, T', \eta', \mu')$ and a distributive
law as in definition \ref{def:distrlaw}, then $T' \circ T$ gives a monad on $\mathcal{C}$ with
unit and multiplication:
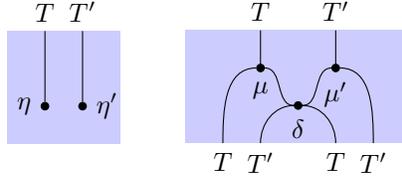
\begin{equation*}
\begin{gathered}
\begin{tikzpicture}[scale=0.5]
\path coordinate[dot, label=left:$\eta$] (eta) +(0,2) coordinate[label=above:$T$] (tl)
 ++(1,0) coordinate[dot, label=right:$\eta'$] (eta') +(0,2) coordinate[label=above:$T'$] (tr);
\draw (eta) -- (tl)
 (eta') -- (tr);
\begin{pgfonlayer}{background}
\fill[catc] ($(tl) + (-1,0)$) rectangle ($(eta') + (1,-1)$);
\fill[white] ($(tl) + (-1,-3)$) rectangle ($(eta') + (1,-2)$);
\end{pgfonlayer}
\end{tikzpicture}
\end{gathered}
\qquad
\begin{gathered}
\begin{tikzpicture}[scale=0.5]
\path coordinate[dot, label=below:$\delta$] (delta)
 +(-1,-1) coordinate[label=below:$T'$] (bl)
 +(1,-1) coordinate[label=below:$T$] (br)
 +(-1,1) coordinate[dot, label=below:$\mu$] (mu)
 +(1,1) coordinate[dot, label=below:$\mu'$] (mu')
 (mu) +(0,1) coordinate[label=above:$T$] (tl) +(-1,-2) coordinate[label=below:$T$] (bll)
 (mu') +(0,1) coordinate[label=above:$T'$] (tr) +(1,-2) coordinate[label=below:$T'$] (brr);
\draw (mu) -- (tl)
 (mu') -- (tr)
 (bl) to[out=90, in=180] (delta.west) -- (delta.east) to[out=0, in=90] (br)
 (mu) to[out=0, in=180] (delta.west) -- (delta.east) to[out=0, in=180] (mu')
 (bll) to[out=90, in=180] (mu)
 (brr) to[out=90, in=0] (mu');
\begin{pgfonlayer}{background}
\fill[catc] ($(bll) + (-1,0)$) rectangle ($(brr) + (1,3)$);
\end{pgfonlayer}
\end{tikzpicture}
\end{gathered}
\end{equation*}
\end{proposition}
\begin{proof}
For the first of the unit axioms, we have the following equalities:
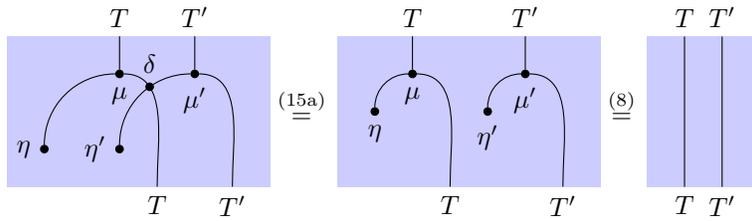
\begin{equation*}
\begin{gathered}
\begin{tikzpicture}[scale=0.5]
\path coordinate[dot, label=below:$\mu'$] (mu') +(0,1) coordinate[label=above:$T'$] (tr)
 +(1,-3) coordinate[label=below:$T'$] (br)
 +(-2,-2) coordinate[dot, label=left:$\eta'$] (eta')
 (mu') ++(-2,0) coordinate[dot, label=below:$\mu$] (mu) ++(1,-3) coordinate[label=below:$T$] (bl)
 (mu) ++(-2,-2) coordinate[dot, label=left:$\eta$] (eta)
 (mu) ++(0,1) coordinate[label=above:$T$] (tl);
\draw (mu') -- (tr)
 (mu') to[out=0, in=90] (br)
 (eta) to[out=90, in=180] (mu)
 (mu) -- (tl);
\draw[name path=vert] (bl) to[out=90, in=0] (mu);
\draw[name path=curv] (eta') to[out=90, in=180] (mu');
\path[name intersections={of=vert and curv}] coordinate[dot, label=above:$\delta$] (delta) at (intersection-1);
\begin{pgfonlayer}{background}
\fill[catc] ($(tl) + (-3,0)$) rectangle ($(br) + (1,0)$);
\end{pgfonlayer}
\end{tikzpicture}
\end{gathered}
\stackrel{\eqref{eq:distretal}}{=}
\begin{gathered}
\begin{tikzpicture}[scale=0.5]
\path coordinate[dot, label=below:$\mu'$] (mu') +(0,1) coordinate[label=above:$T'$] (tr)
 +(1,-3) coordinate[label=below:$T'$] (br) +(-1,-1) coordinate[dot, label=below:$\eta'$] (eta')
 (mu') ++(-3,0) coordinate[dot, label=below:$\mu$] (mu) +(0,1) coordinate[label=above:$T$] (tl)
 +(1,-3) coordinate[label=below:$T$] (bl) +(-1,-1) coordinate[dot, label=below:$\eta$] (eta);
\draw (mu) -- (tl)
 (mu') -- (tr)
 (eta) to[out=90, in=180] (mu.west) -- (mu.east) to[out=0, in=90] (bl)
 (eta') to[out=90, in=180] (mu'.west) -- (mu'.east) to[out=0, in=90] (br);
\begin{pgfonlayer}{background}
\fill[catc] ($(tl) + (-2,0)$) rectangle ($(br) + (1,0)$);
\end{pgfonlayer}
\end{tikzpicture}
\end{gathered}
\stackrel{\eqref{eq:monad-unit}}{=}
\begin{gathered}
\begin{tikzpicture}[scale=0.5]
\path coordinate[label=below:$T$] (bl) +(0,4) coordinate[label=above:$T$] (tl)
 ++(1,0) coordinate[label=below:$T'$] (br) ++(0,4) coordinate[label=above:$T'$] (tr);
\draw (bl) -- (tl)
 (br) -- (tr);
\begin{pgfonlayer}{background}
\fill[catc] ($(bl) + (-1,0)$) rectangle ($(tr) + (1,0)$);
\end{pgfonlayer}
\end{tikzpicture}
\end{gathered}
\end{equation*}
The second unit axiom follows dually using axiom \eqref{eq:distretar}. 
The proof of associativity is a more interesting exercise in manipulating
string diagrams. At each stage we attempt to draw our string diagrams
so as to emphasise the forthcoming proof step, exploiting the topological
flexibility of the notation.
\begin{eqproof*}
\begin{tikzpicture}[scale=0.5]
\path coordinate[dot, label=above:$\delta$] (delta1)
 +(-2,1) coordinate[dot, label=below:$\mu$] (mu1)
 +(2,1) coordinate[dot, label=below:$\mu'$] (mu1')
 (delta1) +(-1,-1) coordinate[label=below:$T'$] (b) +(1,-1) coordinate[label=below:$T$] (c)
 (mu1) +(-1,-2) coordinate[label=below:$T$] (a)
 (mu1') +(1,-2) coordinate[label=below:$T'$] (d)
 (d) ++(1,0) coordinate[label=below:$T$] (e)
 (mu1) ++(1,1) coordinate[dot, label=below:$\mu$] (mu2) ++(0,2) coordinate[label=above:$T$] (a')
 (mu1') ++(1,2) coordinate[dot, label=below:$\mu'$] (mu2') ++(0,1) coordinate[label=above:$T'$] (b')
 (mu2') ++(2,-4) coordinate[label=below:$T'$] (f);
\draw (b) to[out=90, in=180] (delta1.west) -- (delta1.east) to[out=0, in=90] (c)
 (a) 
 to[out=90, in=180] (mu1.west) -- (mu1.east) 
 to[out=0, in=180] (delta1.west) -- (delta1.east) 
 to[out=0, in=180] (mu1'.west) -- (mu1'.east) 
 to[out=0, in=90] (d)
 (mu1) to[out=90, in=180] (mu2)
 (mu2) -- (a')
 (mu2') -- (b')
 (mu2') to[out=0, in=90] (f);
\draw[name path=curv] (e) to[out=90, in=0] (mu2);
\draw[name path=vert] (mu1') to[out=90, in=180] (mu2');
\path[name intersections={of=vert and curv}] coordinate[dot, label=135:$\delta$] (delta2) at (intersection-1);
\begin{pgfonlayer}{background}
\fill[catc] ($(a) + (-1,0)$) rectangle ($(b') + (3,0)$);
\end{pgfonlayer}
\end{tikzpicture}
\explain{distributive law axiom \eqref{eq:distrmul}}
\begin{tikzpicture}[scale=0.5]
\path coordinate[dot, label=above:$\delta$] (delta1)
 +(-2,1) coordinate[dot, label=below:$\mu$] (mu1)
 +(2,3) coordinate[dot, label=below:$\mu'$] (mu1')
 (delta1) +(-1,-1) coordinate[label=below:$T'$] (b) +(1,-1) coordinate[label=below:$T$] (c)
 (mu1) +(-1,-2) coordinate[label=below:$T$] (a)
 (mu1') +(1,-4) coordinate[label=below:$T'$] (d)
 (d) ++(1,0) coordinate[label=below:$T$] (e)
 (mu1) ++(1,1) coordinate[dot, label=below:$\mu$] (mu2) ++(0,3) coordinate[label=above:$T$] (a')
 (mu1') ++(1,1) coordinate[dot, label=below:$\mu'$] (mu2') ++(0,1) coordinate[label=above:$T'$] (b')
 (mu2') ++(2,-5) coordinate[label=below:$T'$] (f);
\draw (b) to[out=90, in=180] (delta1.west) -- (delta1.east) to[out=0, in=90] (c)
 (a) 
 to[out=90, in=180] (mu1.west) -- (mu1.east) 
 to[out=0, in=180] (delta1.west)
 (mu1) to[out=90, in=180] (mu2);
\draw[name path=vert2] (d) to[out=90, in=0] (mu1');
\draw[name path=vert1] (delta1) to[out=0, in=180] (mu1');
\draw (mu2) -- (a');
\draw (mu2') -- (b')
 (mu2') to[out=0, in=90] (f);
\draw[name path=curv] (e) to[out=90, in=0] (mu2);
\draw (mu1') to[out=90, in=180] (mu2');
\path[name intersections={of=vert1 and curv}] coordinate[dot, label=135:$\delta$] (delta2) at (intersection-1);
\path[name intersections={of=vert2 and curv}] coordinate[dot, label=left:$\delta$] (delta3) at (intersection-1);
\begin{pgfonlayer}{background}
\fill[catc] ($(a) + (-1,0)$) rectangle ($(b') + (3,0)$);
\end{pgfonlayer}
\end{tikzpicture}
\explain{monad associativity axiom \eqref{eq:monad-associativity}}
\begin{tikzpicture}[scale=0.5]
\path coordinate[dot, label=135:$\mu$] (mu1) ++(1,1) coordinate[dot, label=below:$\mu$] (mu2) ++(0,1) coordinate[label=above:$T$] (a')
 (mu1) ++(5,0) coordinate[dot, label=45:$\mu'$] (mu1') ++(-1,1) coordinate[dot, label=45:$\mu'$] (mu2') ++(0,1) coordinate[label=above:$T'$] (b')
 (mu1) ++(-1,-2) coordinate[label=below:$T$] (a)
 (mu1) ++(0,-2) coordinate[label=below:$T'$] (b)
 (mu1') ++(1,-2) coordinate[label=below:$T'$] (f)
 (mu1') ++(0,-2) coordinate[label=below:$T$] (e)
 (b) ++(1,0) coordinate[label=below:$T$] (c)
 (e) ++(-1,0) coordinate[label=below:$T'$] (d);
\draw (a) to[out=90, in=180] (mu1.west) -- (mu1.north) to[out=90, in=180] (mu2.west) -- (mu2.north) -- (a');
\draw[name path=ru] (b) to[out=90, in=180] (mu2');
\draw[name path=rd] (d) to[out=90, in=180] (mu1');
\draw (f) to[out=90, in=0] (mu1'.east) -- (mu1'.north) to[out=90, in=0] (mu2'.east) -- (mu2'.north) -- (b');
\draw[name path=lu] (e) to[out=90, in=0] (mu2);
\draw[name path=ld] (c) to[out=90, in=0] (mu1);
\path[name intersections={of=lu and ru}] coordinate[dot, label=above:$\delta$] (delta1) at (intersection-1);
\path[name intersections={of=lu and rd}] coordinate[dot, label=below:$\delta$] (delta2) at (intersection-1);
\path[name intersections={of=ru and ld}] coordinate[dot, label=below:$\delta$] (delta3) at (intersection-1);
\begin{pgfonlayer}{background}
\fill[catc] ($(a) + (-1,0)$) rectangle ($(b') + (3,0)$);
\end{pgfonlayer}
\end{tikzpicture}
\explain{monad associativity axiom \eqref{eq:monad-associativity}}
\begin{tikzpicture}[xscale=-0.5, yscale=0.5]
\path coordinate[dot, label=above:$\delta$] (delta1)
 +(-2,1) coordinate[dot, label=below:$\mu'$] (mu1)
 +(2,3) coordinate[dot, label=below:$\mu$] (mu1')
 (delta1) +(-1,-1) coordinate[label=below:$T$] (b) +(1,-1) coordinate[label=below:$T'$] (c)
 (mu1) +(-1,-2) coordinate[label=below:$T'$] (a)
 (mu1') +(1,-4) coordinate[label=below:$T$] (d)
 (d) ++(1,0) coordinate[label=below:$T'$] (e)
 (mu1) ++(1,1) coordinate[dot, label=below:$\mu'$] (mu2) ++(0,3) coordinate[label=above:$T'$] (a')
 (mu1') ++(1,1) coordinate[dot, label=below:$\mu$] (mu2') ++(0,1) coordinate[label=above:$T$] (b')
 (mu2') ++(2,-5) coordinate[label=below:$T$] (f);
\draw (b) to[out=90, in=180] (delta1.west) -- (delta1.east) to[out=0, in=90] (c)
 (a) 
 to[out=90, in=180] (mu1.west) -- (mu1.east) 
 to[out=0, in=180] (delta1.west)
 (mu1) to[out=90, in=180] (mu2);
\draw[name path=vert2] (d) to[out=90, in=0] (mu1');
\draw[name path=vert1] (delta1) to[out=0, in=180] (mu1');
\draw (mu2) -- (a');
\draw (mu2') -- (b')
 (mu2') to[out=0, in=90] (f);
\draw[name path=curv] (e) to[out=90, in=0] (mu2);
\draw (mu1') to[out=90, in=180] (mu2');
\path[name intersections={of=vert1 and curv}] coordinate[dot, label=45:$\delta$] (delta2) at (intersection-1);
\path[name intersections={of=vert2 and curv}] coordinate[dot, label=left:$\delta$] (delta3) at (intersection-1);
\begin{pgfonlayer}{background}
\fill[catc] ($(a) + (-1,0)$) rectangle ($(b') + (3,0)$);
\end{pgfonlayer}
\end{tikzpicture}
\explain{distributive law axiom \eqref{eq:distrmul}}
\begin{tikzpicture}[xscale=-0.5, yscale=0.5]
\path coordinate[dot, label=above:$\delta$] (delta1)
 +(-2,1) coordinate[dot, label=below:$\mu'$] (mu1)
 +(2,1) coordinate[dot, label=below:$\mu$] (mu1')
 (delta1) +(-1,-1) coordinate[label=below:$T$] (b) +(1,-1) coordinate[label=below:$T'$] (c)
 (mu1) +(-1,-2) coordinate[label=below:$T'$] (a)
 (mu1') +(1,-2) coordinate[label=below:$T$] (d)
 (d) ++(1,0) coordinate[label=below:$T'$] (e)
 (mu1) ++(1,1) coordinate[dot, label=below:$\mu'$] (mu2) ++(0,2) coordinate[label=above:$T'$] (a')
 (mu1') ++(1,2) coordinate[dot, label=below:$\mu$] (mu2') ++(0,1) coordinate[label=above:$T$] (b')
 (mu2') ++(2,-4) coordinate[label=below:$T$] (f);
\draw (b) to[out=90, in=180] (delta1.west) -- (delta1.east) to[out=0, in=90] (c)
 (a) 
 to[out=90, in=180] (mu1.west) -- (mu1.east) 
 to[out=0, in=180] (delta1.west) -- (delta1.east) 
 to[out=0, in=180] (mu1'.west) -- (mu1'.east) 
 to[out=0, in=90] (d)
 (mu1) to[out=90, in=180] (mu2)
 (mu2) -- (a')
 (mu2') -- (b')
 (mu2') to[out=0, in=90] (f);
\draw[name path=curv] (e) to[out=90, in=0] (mu2);
\draw[name path=vert] (mu1') to[out=90, in=180] (mu2');
\path[name intersections={of=vert and curv}] coordinate[dot, label=45:$\delta$] (delta2) at (intersection-1);
\begin{pgfonlayer}{background}
\fill[catc] ($(a) + (-1,0)$) rectangle ($(b') + (3,0)$);
\end{pgfonlayer}
\end{tikzpicture}
\end{eqproof*}
We note how the proof makes essential use of the symmetry of the monad and distributive law multiplication
axioms. 
\end{proof}
\begin{proposition}
Given monads $(\mathcal{C}, T, \eta, \mu)$ and $(\mathcal{C}, T', \eta', \mu')$ and a distributive
law as in definition \ref{def:distrlaw}, then $T'$ induces a monad on \eilmo{T}.
\end{proposition}
\begin{proof}
We first note that a distributive law is a \refcmonad morphism, and so $T'$ lifts to an endofunctor
on \eilmo{T} by proposition \ref{prop:cmonadlift}. We take $\eta'$ and $\mu'$ as the unit and
multiplication.
That $\eta$ gives a natural transformation in \eilmo{T} follow from:
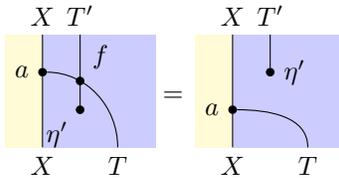
\begin{equation*}
\begin{gathered}
\begin{tikzpicture}[scale=0.5]
\path coordinate[dot, label=left:$a$] (a) +(0,1) coordinate[label=above:$X$] (tl) +(1,1) coordinate[label=above:$T'$] (tr)
 ++(0,-2) coordinate[label=below:$X$] (bl) ++(2,0) coordinate[label=below:$T$] (br)
 ++(-1,1) coordinate[dot, label=-135:$\eta'$] (eta');
\draw (bl) -- (a) -- (tl);
\draw[name path=curv] (br) to[out=90, in=0] (a);
\draw[name path=vert] (eta') -- (tr);
\path[name intersections={of=vert and curv}] coordinate[dot, label=45:$f$] (f) at (intersection-1);
\begin{pgfonlayer}{background}
\fill[catterm] ($(tl) + (-1,0)$) rectangle (bl);
\fill[catc] (tl) rectangle ($(br) + (1,0)$);
\end{pgfonlayer}
\end{tikzpicture}
\end{gathered}
=
\begin{gathered}
\begin{tikzpicture}[scale=0.5]
\path coordinate[dot, label=left:$a$] (a) +(0,-1) coordinate[label=below:$X$] (bl)
 +(2,-1) coordinate[label=below:$T$] (br)
 ++(0,2) coordinate[label=above:$X$] (tl)
 ++(1,0) coordinate[label=above:$T'$] (tr)
 ++(0,-1) coordinate[dot, label=right:$\eta'$] (eta');
\draw (eta') -- (tr)
 (bl) -- (a) -- (tl)
 (br) to[out=90, in=0] (a);
\begin{pgfonlayer}{background}
\fill[catterm] ($(tl) + (-1,0)$) rectangle (bl);
\fill[catc] (tl) rectangle ($(br) + (1,0)$);
\end{pgfonlayer}
\end{tikzpicture}
\end{gathered}
\end{equation*}
Similarly that $\mu'$ is a natural transformation in \eilmo{T} follows from:
\begin{equation*}
\begin{gathered}
\begin{tikzpicture}[scale=0.5]
\path coordinate[dot, label=left:$a$] (a) +(0,1) coordinate[label=above:$X$] (tl)
 +(2,1) coordinate[label=above:$T'$] (tr)
 +(4,-1.5) coordinate (b)
 ++(0,-3) coordinate[label=below:$X$] (bl)
 ++(2,1) coordinate[dot, label=45:$\mu'$] (mu')
 +(-1,-1) coordinate[label=below:$T'$] (bm)
 ++(1,-1) coordinate[label=below:$T'$] (br)
 (br) ++(1,0) coordinate[label=below:$T$] (brr);
\draw (bm) to[out=90, in=180] (mu'.west) -- (mu'.east) to[out=0, in=90] (br)
 (bl) -- (a) -- (tl);
\draw[name path=vert] (mu') -- (tr);
\draw[name path=curv] (a) to[out=0, in=90] (b) -- (brr);
\path[name intersections={of=vert and curv}] coordinate[dot, label=-135:$f$] (f) at (intersection-1);
\begin{pgfonlayer}{background}
\fill[catterm] ($(tl) + (-1,0)$) rectangle (bl);
\fill[catc] (tl) rectangle ($(brr) + (1,0)$);
\end{pgfonlayer}
\end{tikzpicture}
\end{gathered}
=
\begin{gathered}
\begin{tikzpicture}[scale=0.5]
\path coordinate[dot, label=left:$a$] (a)
 +(0,2) coordinate[label=above:$X$] (tl)
 +(2,2) coordinate[label=above:$T'$] (tr)
 +(4,-1.5) coordinate (b)
 ++(0,-2) coordinate[label=below:$X$] (bl)
 ++(1,0) coordinate[label=below:$T'$] (bm) +(1,3) coordinate[dot, label=45:$\mu'$] (mu')
 ++(2,0) coordinate[label=below:$T'$] (br)
 ++(1,0) coordinate[label=below:$T$] (brr);
\draw[name path=vert1] (bm) to[out=90, in=180] (mu');
\draw[name path=vert2] (br) to[out=90, in=0] (mu');
\draw[name path=curv] (brr) -- (b) to[out=90, in=0] (a);
\draw (mu') -- (tr);
\path[name intersections={of=vert1 and curv}] coordinate[dot, label=135:$f$] (f1) at (intersection-1);
\path[name intersections={of=vert2 and curv}] coordinate[dot, label=-135:$f$] (f2) at (intersection-1);
\begin{pgfonlayer}{background}
\fill[catterm] ($(tl) + (-1,0)$) rectangle (bl);
\fill[catc] (tl) rectangle ($(brr) + (1,0)$);
\end{pgfonlayer}
\end{tikzpicture}
\end{gathered}
\end{equation*}
That $\eta'$ and $\mu'$ satisfying the monad axioms in then obvious as morphisms in \eilmo{T} compose
exactly as in $\mathcal{C}$.
\end{proof}

\section{Representable Functors}
\label{sec:repr}
We now investigate the important concept of representable functors. Many aspects of category theory
can be phrased as requiring that a certain functor is representable, and examined from the string diagram
perspective, representability provides a standard framework for generating calculation rules in a
uniform manner. These calculation rules can then be specialized to provide graphical rules for limits,
colimits, adjunctions, and left and right Kan extensions with little effort, unifying how these concepts
can be approached graphically. Our use of representability bears some resemblance to the use of
initiality to provide calculation rules in \citep{Fokkinga1992b}.

\subsection{Covariant Representable Functors}
\label{sec:rep}
\begin{definition}[Covariant Representable Functor]
A covariant functor $G : \mathcal{C} \rightarrow \refcset$ is said to be \define{representable} if there
is an object $S$ of $\mathcal{C}$ such that there is a natural isomorphism:
\begin{equation}
\label{eq:repr}
G \cong \mathcal{C}(S,-)
\end{equation}
\end{definition}
We will identify the elements of a set with morphisms from the one element set, 
as discussed in section \ref{sec:elements}.
We will use a box notation to represent the mappings of the natural isomorphism,
similar to the box notation for monoidal functors introduced in \citep{CockettSeely1999, BluteCockettSeely2002}
and described in detail in \citep{Mellies2006}.
For an object $X$ in $\mathcal{C}$ the mappings from left to right and right to left in equation 
\eqref{eq:repr} can be drawn graphically as:
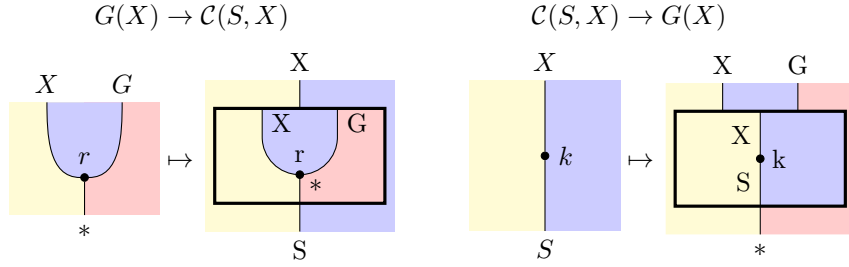
\begin{trivlist}
\item
\begin{minipage}{0.495\textwidth}
\begin{align*}
G(X) &\rightarrow \mathcal{C}(S,X)\\
\begin{gathered}
\diagrepto{r}{*}{X}{G}{catterm}{catc}{catset}{1}{2}
\end{gathered}
&\mapsto
\begin{gathered}
\begin{tikzpicture}[auto, x=0.5cm, y=0.5cm]
\diagbb{2.5cm}{2cm};
\outerreptoex{cbb}{X}{G}{r}{S};
\end{tikzpicture}
\end{gathered}
\end{align*}
\end{minipage}
\begin{minipage}{0.495\textwidth}
\begin{align*}
\mathcal{C}(S,X) &\rightarrow G(X)\\
\begin{gathered}
\twocelldiag{k}{S}{X}{catterm}{catc}{2}{2}{2}
\end{gathered}
&\mapsto
\begin{gathered}
\begin{tikzpicture}[auto, x=0.5cm, y=0.5cm]
\diagbb{2.5cm}{2cm};
\outerrepfromex{cbb}{S}{X}{k}{G};
\end{tikzpicture}
\end{gathered}
\end{align*}
\end{minipage}
\end{trivlist}
As these operations correspond to isomorphisms, the following two conditions hold:
\begin{subequations}
\begin{equation}
\label{eq:reprisoa}
\begin{gathered}
\begin{tikzpicture}[auto, x=0.5cm, y=0.5cm]
\diagbb{3.5cm}{3cm};
\path (cbb.north)
      -- ++(-1, 0) coordinate[label=above:$X$] (a)
      -- ++(0,-1.5) coordinate (a')
      (cbb.north)
      -- ++(1, 0) coordinate[label=above:$G$] (b)
      -- ++(0,-1.5) coordinate (b')
      (cbb.south) coordinate[label=below:$*$] (bot)
      -- ++(0,1.5) coordinate (c);
\fill[catterm] (cbb.south) -- (cbb.center) -- (a') -- (a) 
 -- (cbb.north west) -- (cbb.south west) -- cycle;
\fill[catset] (cbb.south) -- (cbb.center) -- (b') -- (b)
 -- (cbb.north east) -- (cbb.south east) -- cycle;
\fill[catc] (a) rectangle (b');
\draw (a') -- (a)
      (b') -- (b)
      (cbb.south) -- (c);
\path (cbb.south west) -- ++(0.5,1) coordinate (bl)
      (cbb.north east) -- ++(-0.5,-1) coordinate (tr);
\node[rectangle, fit=(bl)(tr)] (subdiag) {};
\reptoex{subdiag}{X}{G}{r}{S};
\draw[very thick] (subdiag.south west) rectangle (subdiag.north east);
\end{tikzpicture}
\end{gathered}
=
\begin{gathered}
\diagrepto{r}{*}{X}{G}{catterm}{catc}{catset}{1}{2}
\end{gathered}
\end{equation}
\begin{equation}
\label{eq:reprisob}
\begin{gathered}
\begin{tikzpicture}[auto, x=0.5cm, y=0.5cm]
\diagbb{3.5cm}{3cm};
\fill[catterm] (cbb.south west) rectangle (cbb.north);
\fill[catc] (cbb.south east) rectangle (cbb.north);
\coordinate[label=below:$S$] (b) at (cbb.south);
\coordinate[label=above:$X$] (t) at (cbb.north);
\draw (b) -- (t);
\path (cbb.south west) -- ++(0.5,1) coordinate (bl)
      (cbb.north east) -- ++(-0.5,-1) coordinate (tr);
\node[rectangle, fit=(bl)(tr)] (subdiag) {};
\repfromex{subdiag}{S}{X}{k}{G};
\draw[very thick] (subdiag.south west) rectangle (subdiag.north east);
\end{tikzpicture}
\end{gathered}
=
\begin{gathered}
\twocelldiag{k}{S}{X}{catterm}{catc}{2}{2}{2}
\end{gathered}
\end{equation}
\end{subequations}
The important aspect of representability is that the isomorphism is natural.
This can be seen as requiring that the following two calculation rules hold,
allowing us to ``push and pop'' morphisms in and out of the box notation.
\begin{subequations}
\begin{equation}
\label{eq:thetaone}
\begin{gathered}
\begin{tikzpicture}[auto, x=0.5cm, y=0.5cm]
\diagbb{2.5cm}{2.5cm};
\fill[catterm] (cbb.south west) rectangle (cbb.north);
\fill[catc] (cbb.south east) rectangle (cbb.north);
\path (cbb.south west) -- ++(0.5,1) coordinate (bl)
      (cbb.north east) -- ++(-0.5,-2) coordinate (tr)
      (cbb.north) coordinate[label=above:$Y$] (top) -- ++(0,-1) coordinate[dot, label=right:$h$] (h)
      (cbb.south) coordinate[label=below:$S$] (bot);
\draw (top) -- (h) -- ++(0,-1.5)
      (bot) -- ++(0,1.5);
\node[rectangle, fit=(bl)(tr)] (subdiag) {};
\repto{subdiag}{X}{G}{r}{$*$};
\draw[very thick] (subdiag.south west) rectangle (subdiag.north east);
\end{tikzpicture}
\end{gathered}
=
\begin{gathered}
\begin{tikzpicture}[auto, x=0.5cm, y=0.5cm]
\diagbb{2.5cm}{2.5cm};
\path (cbb.south west) -- ++(0.5,1) coordinate (bl)
      (cbb.north east) -- ++(-0.5,-1) coordinate (tr);
\fill[catterm] (cbb.south west) rectangle (cbb.north);
\fill[catc] (cbb.south east) rectangle (cbb.north);
\draw (cbb.south) coordinate[label=below:$S$] (bot) -- ++(0,1.5)
      (cbb.north) coordinate[label=above:$Y$] (top) -- ++(0, -1.5);
\node[rectangle, fit=(bl)(tr)] (subdiag) {};
\reptocomp{subdiag}{X}{Y}{G}{r}{h}{$*$};
\draw[very thick] (subdiag.south west) rectangle (subdiag.north east);
\end{tikzpicture}
\end{gathered}
\end{equation}
\begin{equation}
\label{eq:thetatwo}
\begin{gathered}
\begin{tikzpicture}[auto, x=0.5cm, y=0.5cm]
\diagbb{2.5cm}{2cm};
\path (cbb.south west) -- ++(0.5,1) coordinate (bl)
      (cbb.north east) -- ++(-0.5,-1.5) coordinate (tr)
      (cbb.north) -- ++(-1,0) coordinate[label=above:$Y$] (a')
      (cbb.north) -- ++(1,0) coordinate[label=above:$G$] (b');
\path (a') -- ++(0,-1.5) coordinate (a)
      (b') -- ++(0,-1.5) coordinate (b);
\coordinate[dot, label=right:$h$] (h) at ($(a)!0.5!(a')$);
\begin{pgfonlayer}{background}
\fill[catterm] (cbb.south) -- (cbb.center) -- (a)
  -- (a') -- (cbb.north west) -- (cbb.south west) -- cycle;
\fill[catset] (cbb.south) -- (cbb.center) -- (b)
  -- (b') -- (cbb.north east) -- (cbb.south east) -- cycle;
\fill[catc] (a) rectangle (b');
\end{pgfonlayer}
\draw (a) -- (h) -- (a')
      (b) -- (b')
      (cbb.south) coordinate[label=below:$*$] (bot) -- ++(0,1.5);
\node[rectangle, fit=(bl)(tr)] (subdiag) {};
\repfrom{subdiag}{S}{X}{k};
\draw[very thick] (subdiag.south west) rectangle (subdiag.north east);
\end{tikzpicture}
\end{gathered}
=
\begin{gathered}
\begin{tikzpicture}[auto, x=0.5cm, y=0.5cm]
\diagbb{2.5cm}{2cm};
\path (cbb.north)
      -- +(-1,0) coordinate[label=above:$Y$] (a')
      -- +(1,0) coordinate[label=above:$G$] (b');
\path (a') -- ++(0,-2) coordinate (a)
      (b') -- ++(0,-2) coordinate (b);
\path (cbb.south west) -- ++(0.5,1) coordinate (bl)
      (cbb.north east) -- ++(-0.5,-1) coordinate (tr);
\fill[catterm] (cbb.south) -- (cbb.center) -- (a) -- (a') 
 -- (cbb.north west) -- (cbb.south west) -- cycle;
\fill[catset] (cbb.south) -- (cbb.center) -- (b) -- (b') 
 -- (cbb.north east) -- (cbb.south east) -- cycle;
\fill[catc] (a) rectangle (b');
\draw (a) -- (a')
      (b) -- (b')
      (cbb.south) coordinate[label=below:$*$] -- ++(0,1.5);
\node[rectangle, fit=(bl)(tr)] (subdiag) {};
\repfromcomp{subdiag}{S}{X}{Y}{k}{h};
\draw[very thick] (subdiag.south west) rectangle (subdiag.north east);
\end{tikzpicture}
\end{gathered}
\end{equation}
\end{subequations}
\newcommand{\lmorph}[1]{
\path (#1.center) -- ++(-0.5, 0) coordinate (f)
      -- ++(1,-1) coordinate (a);
\path let \p1 = (#1.north) in
      let \p2 = (#1.south) in
      let \p3 = (f) in
      let \p4 = (a) in
      coordinate (t) at (\x3, \y1)
      coordinate (b) at (\x3, \y2)
      coordinate (a') at (\x4, \y2);
\fill[catterm] (b) -- (t) -- (#1.north west) -- (#1.south west) -- cycle;
\fill[catd] (f) to[out=0, in=90] (a) -- (a') -- (#1.south east) -- (#1.north east) -- (t) -- cycle;
\fill[catc] (f) to[out=0, in=90] (a) -- (a') -- (b) -- cycle;

\draw (b) to node{$X$} (f) to node{$Y$} (t)
      (f) to[out=0, in=90] (a) to node{$F$} (a');
\strnat{f};
\strlabl{f}{$f$}
}

\newcommand{\lmorphex}[1]{
\path (#1.center) -- ++(-0.5, 0) coordinate (f)
      -- ++(1,-1) coordinate (a);
\path let \p1 = (#1.north) in
      let \p2 = (#1.south) in
      let \p3 = (f) in
      let \p4 = (a) in
      coordinate (t) at (\x3, \y1)
      coordinate (b) at (\x3, \y2)
      coordinate (a') at (\x4, \y2);
\coordinate (h) at ($(f)!0.5!(t)$);
\fill[catterm] (b) -- (t) -- (#1.north west) -- (#1.south west) -- cycle;
\fill[catd] (f) to[out=0, in=90] (a) -- (a') -- (#1.south east) -- (#1.north east) -- (t) -- cycle;
\fill[catc] (f) to[out=0, in=90] (a) -- (a') -- (b) -- cycle;

\draw (b) to node{$X$} (f) to node{$Y$} (h) to node{$Z$} (t)
      (f) to[out=0, in=90] (a) to node{$F$} (a');
\strnat{f};
\strlabl{f}{$f$};
\strnat{h};
\strlabr{h}{$h$}
}

\newcommand{\rmorph}[1]{
\path (#1.center) -- ++(-0.5, 0) coordinate (f)
      -- ++(1,1) coordinate (a);
\path let \p1 = (#1.north) in
      let \p2 = (#1.south) in
      let \p3 = (f) in
      let \p4 = (a) in
      coordinate (t) at (\x3, \y1)
      coordinate (b) at (\x3, \y2)
      coordinate (a') at (\x4, \y1);
\fill[catterm] (b) -- (t) -- (#1.north west) -- (#1.south west) -- cycle;
\fill[catc] (f) to[out=0, in=270] (a) -- (a') -- (#1.north east) -- (#1.south east) -- (b) -- cycle;
\fill[catd] (f) to[out=0, in=270] (a) -- (a') -- (t) -- cycle;

\draw (b) to node{$X$} (f) to node{$Y$} (t)
      (f) to[out=0, in=270] (a) to node[swap]{$G$} (a');
\strnat{f};
\strlabl{f}{$g$}
}
The next lemma shows that to prove representability it is sufficient to show a family
of bijections satisfy just one of the two equations above.
\begin{lemma}
\label{lem:pushpop}
For a given family of isomorphisms of hom sets:
\begin{equation*}
\theta_X : F(X) \cong \mathcal{C}(S,X)
\end{equation*}
The conditions in equations \eqref{eq:thetaone} and \eqref{eq:thetatwo} are equivalent.
\end{lemma}
\begin{proof}
It is easy to show that if every component of a natural transformation has an inverse,
then the inverses form a natural transformation. The claim then follows immediately.
\end{proof}
\begin{remark}
Given a representable functor:
\begin{equation*}
G \cong \mathcal{C}(S,-)
\end{equation*}
Generally we will know more about the structure of the left hand side than just
that it is a functor with codomain \refcset. It will be common for $G$ to involve
a hom bifunctor $\mathcal{C}(-,-)$ in some way, and this structure can be used to
put the above equations into more convenient form than reasoning in terms of elements
of an abstract set. This will be seen in the examples in section \ref{sec:reprexamples}.
\end{remark}

\subsection{Contravariant Representable Functors}
\begin{definition}[Contravariant Representable Functor]
A contravariant functor $F : \mathcal{C}^{op} \rightarrow \refcset$ is said to be representable
if there is an object $R$ of $\mathcal{C}$ such that there is a natural isomorphism:
\begin{equation}
\label{eq:reprcontra}
F \cong \mathcal{C}(-,R)
\end{equation}
\end{definition}
It is obvious from equation \eqref{eq:reprcontra} we also have:
\begin{equation*}
F \cong \mathcal{C}^{op}(R,-)
\end{equation*}
This is now expressed as in equation \eqref{eq:repr} for the covariant case, and so we can use
identical definitions to those in section \ref{sec:rep}, but replacing $\mathcal{C}$ with $\mathcal{C}^{op}$
everywhere.
Alternatively, we can use both $\mathcal{C}$ and $\mathcal{C}^{op}$ in our diagrams.
The left to right and right to left mappings of equation \eqref{eq:repr} then become:
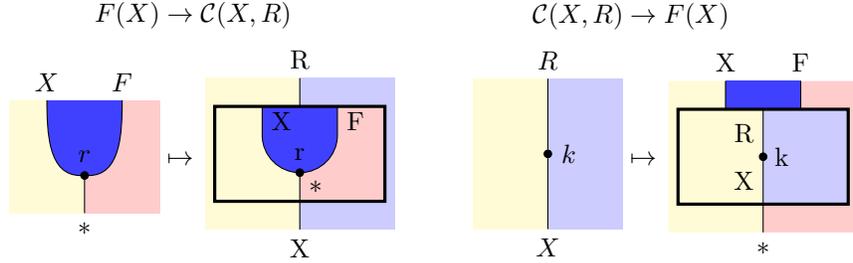
\begin{trivlist}
\item
\begin{minipage}{0.495\textwidth}
\begin{align*}
F(X) &\rightarrow \mathcal{C}(X, R)\\
\begin{gathered}
\diagrepto{r}{*}{X}{F}{catterm}{catcop}{catset}{1}{2}
\end{gathered}
&\mapsto
\begin{gathered}
\begin{tikzpicture}[auto, x=0.5cm, y=0.5cm]
\diagbb{2.5cm}{2cm};
\repcontratoex{cbb}{X}{F}{r}{R};
\end{tikzpicture}
\end{gathered}
\end{align*}
\end{minipage}
\newcommand{\repcontrafromex}[5]{
 \path (#1.north)
       -- ++(-1,0) coordinate[label=above:#2] (#1-a)
       -- ++(0,-1.5) coordinate (#1-a')
       (#1.north)
       -- ++(1,0) coordinate[label=above:#5] (#1-b)
       -- ++(0,-1.5) coordinate (#1-b');
 \fill[catcop] (#1-a) rectangle (#1-b');
 \fill[catterm] (#1.south west) -- (#1.south) -- (#1.center) -- (#1-a') -- (#1-a) -- (#1.north west) -- cycle;
 \fill[catset] (#1.south east) -- (#1.south) -- (#1.center) -- (#1-b') -- (#1-b) -- (#1.north east) -- cycle;
\coordinate[label=below:$*$] (bot) at (#1.south);
 \draw (#1.south) -- ++(0,1.5)
       (#1-b) -- (#1-b')
       (#1-a) -- (#1-a');
 \path (#1.south west) -- ++(0.5,1) coordinate (#1-bl)
       (#1.north east) -- ++(-0.5,-1) coordinate (#1-tr);
 \node[rectangle, fit=(#1-bl)(#1-tr)] (#1-subdiag) {};
 \repfrom{#1-subdiag}{#2}{#3}{#4};
 \draw[very thick] (#1-subdiag.south west) rectangle (#1-subdiag.north east)
}
\begin{minipage}{0.495\textwidth}
\begin{align*}
\mathcal{C}(X, R) &\rightarrow F(X)\\
\begin{gathered}
\twocelldiag{k}{X}{R}{catterm}{catc}{2}{2}{2}
\end{gathered}
&\mapsto
\begin{gathered}
\begin{tikzpicture}[auto, x=0.5cm, y=0.5cm]
\diagbb{2.5cm}{2cm};
\repcontrafromex{cbb}{X}{R}{k}{F};
\end{tikzpicture}
\end{gathered}
\end{align*}
\end{minipage}
\end{trivlist}
The calculation rules given in equations \eqref{eq:thetaone} and \eqref{eq:thetatwo} can also be adapted in a similar manner.

\subsection{Universality from Representables}
\label{sec:universality}
Much of category theory is concerned with the important notion of a universal property. We now
show how the representability of a functor leads to a particular universal property. When more
structure is known about the representable functor, we can recover well known universal properties
such as those of adjunctions, limits, colimits and Kan extensions, as will be discussed in later
sections and examples.

\begin{definition}[Counit]
For a representable functor:
\begin{equation*}
G \cong \mathcal{C}(R,-)
\end{equation*}
the \define{counit} is the element of $G(R)$ that is the image of $1_R$ under the natural isomorphism.
\end{definition}

\begin{proposition}
For a representable functor:
\begin{equation*}
G \cong \mathcal{C}(R,-)
\end{equation*}
and $X$ an object of $\mathcal{C}$, for every element $r \in G(X)$, there exists a unique $f$ such 
that $r$ can be written in the form:
\begin{equation*}
r = G(f)(c)
\end{equation*}
where $c$ is the counit of the representable functor.
\end{proposition}
\begin{proof}
For an arbitrary $r \in F(X)$ we have the following sequence of equalities:
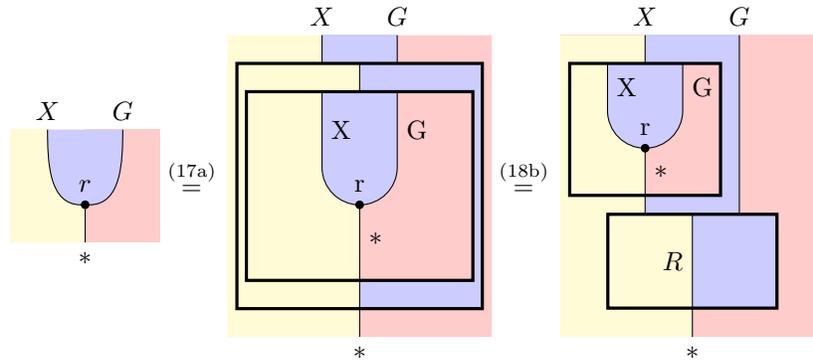
\begin{equation*}
\begin{gathered}
\diagrepto{r}{*}{X}{G}{catterm}{catc}{catset}{1}{2}
\end{gathered}
\stackrel{\eqref{eq:reprisoa}}{=}
\begin{gathered}
\begin{tikzpicture}[auto, x=0.5cm, y=0.5cm]
\diagbb{3.5cm}{4cm};
\path (cbb.north)
      -- ++(-1, 0) coordinate[label=above:$X$] (a)
      -- ++(0,-1.5) coordinate (a')
      (cbb.north)
      -- ++(1, 0) coordinate[label=above:$G$] (b)
      -- ++(0,-1.5) coordinate (b')
      (cbb.south) coordinate[label=below:$*$] (bot)
      -- ++(0,1.5) coordinate (c);
\fill[catterm] (cbb.south) -- (cbb.center) -- (a') -- (a) 
 -- (cbb.north west) -- (cbb.south west) -- cycle;
\fill[catset] (cbb.south) -- (cbb.center) -- (b') -- (b)
 -- (cbb.north east) -- (cbb.south east) -- cycle;
\fill[catc] (a) rectangle (b');
\draw (a') -- (a)
      (b') -- (b)
      (cbb.south) -- (c);
\path (cbb.south west) -- ++(0.5,1) coordinate (bl)
      (cbb.north east) -- ++(-0.5,-1) coordinate (tr);
\node[rectangle, fit=(bl)(tr)] (subdiag) {};
\reptoex{subdiag}{X}{G}{r}{S};
\draw[very thick] (subdiag.south west) rectangle (subdiag.north east);
\end{tikzpicture}
\end{gathered}
\stackrel{\eqref{eq:thetatwo}}{=}
\begin{gathered}
\begin{tikzpicture}[auto, x=0.5cm, y=0.5cm]
\diagbb{3.5cm}{4cm};
\path (cbb.north west) ++(0.5,-1.0) coordinate (tl)
      (cbb.center) ++ (0.5,0) coordinate (br);
\path (cbb.center) ++(-2,-1) coordinate (tl')
      (cbb.center) ++(2,-3) coordinate (br');
\node[rectangle, fit=(tl)(br)] (subdiag) {};
\node[rectangle, fit=(tl')(br')] (subdiag') {};
\path (subdiag.south) coordinate (a)
      (subdiag.north) coordinate (b)
      (subdiag'.south) coordinate (c)
      (subdiag'.north east) ++(-1,0) coordinate (d);
\path let \p1 = (subdiag'.north) in
      let \p2 = (a) in
      let \p3 = (cbb.north) in
      let \p4 = (b) in
      let \p5 = (cbb.south) in
      let \p6 = (c) in
      let \p7 = (d) in
 coordinate (a') at (\x2, \y1)
 coordinate[label=above:$X$] (b') at (\x4, \y3)
 coordinate[label=below:$*$] (c') at (\x6, \y5)
 coordinate[label=above:$G$] (d') at (\x7, \y3);
\begin{pgfonlayer}{background}
\fill[catset] (cbb.north west) rectangle (cbb.south east);
\fill[catc] (a') rectangle (d');
\fill[catterm] (b') -- (a') -- (c) -- (c') -- (cbb.south west) -- (cbb.north west) -- cycle;
\end{pgfonlayer}
\draw (a') -- (a)
      (b) -- (b')
      (c') -- (c)
      (d) -- (d');
\fill[catc] (subdiag'.south west) rectangle (subdiag'.north east);
\fill[catterm] (subdiag'.south west) rectangle (subdiag'.north);
\draw (subdiag'.south) to node{$R$} (subdiag'.north);
\repto{subdiag}{X}{G}{r}{$*$};
\draw[very thick] (subdiag.south west) rectangle (subdiag.north east)
                  (subdiag'.south west) rectangle (subdiag'.north east);
\end{tikzpicture}
\end{gathered}
\end{equation*}
Now assume $r = G(f)(c)$ and $r = G(g)(c)$, then:
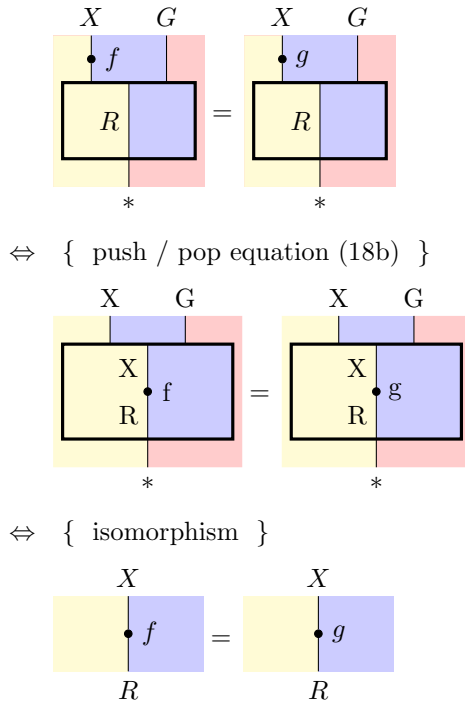
\begin{eqproof*}
\begin{gathered}
\repuniva{f}
\end{gathered}
=
\begin{gathered}
\repuniva{g}
\end{gathered}
\explain[\Leftrightarrow]{ push / pop equation \eqref{eq:thetatwo} }
\begin{gathered}
\begin{tikzpicture}[auto, x=0.5cm, y=0.5cm]
\diagbb{2.5cm}{2cm};
\outerrepfromex{cbb}{R}{X}{f}{G};
\end{tikzpicture}
\end{gathered}
=
\begin{gathered}
\begin{tikzpicture}[auto, x=0.5cm, y=0.5cm]
\diagbb{2.5cm}{2cm};
\outerrepfromex{cbb}{R}{X}{g}{G};
\end{tikzpicture}
\end{gathered}
\explain[\Leftrightarrow]{ isomorphism }
\begin{gathered}
\twocelldiag{f}{R}{X}{catterm}{catc}{2}{2}{1}
\end{gathered}
=
\begin{gathered}
\twocelldiag{g}{R}{X}{catterm}{catc}{2}{2}{1}
\end{gathered}
\end{eqproof*}
\end{proof}
\begin{remark}
A dual universality result can be given for contravariant representable functors,
the details are left to the interested reader.
\end{remark}

\subsection{Examples of Representable Functors and their Calculation Rules}
\label{sec:reprexamples}
In this section we will consider two important applications of the representability based approach of section
\ref{sec:repr}, adjunctions and Kan extensions. Another important source of examples are limits and colimits,
these will be discussed in detail later in section \ref{sec:limits}.

\begin{example}[Adjunctions]
For $F \dashv G$, for each $X$ and $Y$ we have mappings:
\begin{subequations}
\begin{equation}
\label{eq:adjFtoG}
\begin{gathered}
\gmorph{f}{F}{X}{Y}{catterm}{catc}{catd}{1.5}{1}
\end{gathered}
\mapsto
\begin{gathered}
\begin{tikzpicture}[auto, x=0.5cm, y=0.5cm]
\diagbb{2.5cm}{2.5cm};
\path (cbb.center)
      -- +(-1.5, -1.5) coordinate (bl)
      -- +(1.5,1.5) coordinate (tr)
      -- +(0,-1.5) coordinate (a)
      -- +(-1,1.5) coordinate (b)
      -- +(1, 1.5) coordinate (c);
\path let \p1 = (cbb.south) in
      let \p2 = (cbb.north) in
      let \p3 = (a) in
      let \p4 = (b) in
      let \p5 = (c) in
 coordinate[label=below:$X$] (a') at (\x3, \y1)
 coordinate[label=above:$Y$] (b') at (\x4, \y2)
 coordinate[label=above:$G$] (c') at (\x5, \y2);
\draw (a') -- (a)
      (b) -- (b')
      (c) -- (c');
\begin{pgfonlayer}{background}
\fill[catd] (b) -- (b') -- (c') -- (c) -- cycle;
\fill[catc] (a') -- (a) -- (c) -- (c') 
 -- (cbb.north east) -- (cbb.south east) -- cycle;
\fill[catterm] (a') -- (a) -- (b) -- (b')
 -- (cbb.north west) -- (cbb.south west) -- cycle;
\end{pgfonlayer}
\node[rectangle, fit=(bl)(tr)] (subdiag) {};
\lmorph{subdiag};
\draw[very thick] (subdiag.south west) rectangle (subdiag.north east);
\end{tikzpicture}
\end{gathered}
:=
\begin{gathered}
\bendfmorph{f}{1}
\end{gathered}
\end{equation}
\begin{equation}
\label{eq:adjGtoF}
\begin{gathered}
\fmorph{g}{G}{X}{Y}{catterm}{catc}{catd}{1.5}{1}
\end{gathered}
\mapsto
\begin{gathered}
\begin{tikzpicture}[auto, x=0.5cm, y=0.5cm]
\diagbb{2.5cm}{2.5cm};
\path (cbb.center)
      -- +(-1.5, -1.5) coordinate (bl)
      -- +(1.5,1.5) coordinate (tr)
      -- +(-1,-1.5) coordinate (a)
      -- +(1,-1.5) coordinate (b)
      -- +(0, 1.5) coordinate (c);
\path let \p1 = (cbb.north) in
      let \p2 = (cbb.south) in
      let \p3 = (a) in
      let \p4 = (b) in
      let \p5 = (c) in
 coordinate[label=below:$X$] (a') at (\x3, \y2)
 coordinate[label=below:$F$] (b') at (\x4, \y2)
 coordinate[label=above:$Y$] (c') at (\x5, \y1);
\begin{pgfonlayer}{background}
\fill[catc] (a') -- (a) -- (b) -- (b') -- cycle;
\fill[catd] (b') -- (b) -- (c) -- (c')
 -- (cbb.north east) -- (cbb.south east) -- cycle;
\fill[catterm] (a') -- (a) -- (c) -- (c') 
 -- (cbb.north west) -- (cbb.south west) -- cycle;
\end{pgfonlayer}
\draw (a') -- (a)
      (b') -- (b)
      (c) -- (c');
\node[rectangle, fit=(bl)(tr)] (subdiag) {};
\rmorph{subdiag};
\draw[very thick] (subdiag.south west) rectangle (subdiag.north east);
\end{tikzpicture}
\end{gathered}
:=
\begin{gathered}
\bendgmorph{g}{1}
\end{gathered}
\end{equation}
\end{subequations}
It is immediate from the adjunction axioms that these maps are mutually inverse. 
Naturality of the first mapping follows easily from the following equalities:
\begin{equation*}
\begin{gathered}
\begin{tikzpicture}[auto, x=0.5cm, y=0.5cm]
\diagbb{2.5cm}{3cm};
\path (cbb.center)
      -- +(-1.5, -1.5) coordinate (bl)
      -- +(1.5,1.5) coordinate (tr)
      -- +(0,-1.5) coordinate (a)
      -- +(-1,1.5) coordinate (b)
      -- +(1, 1.5) coordinate (c);
\path let \p1 = (cbb.south) in
      let \p2 = (cbb.north) in
      let \p3 = (a) in
      let \p4 = (b) in
      let \p5 = (c) in
 coordinate[label=below:$X$] (a') at (\x3, \y1)
 coordinate[label=above:$Z$] (b') at (\x4, \y2)
 coordinate[label=above:$G$] (c') at (\x5, \y2);
\coordinate[dot, label=right:$h$] (h) at ($(b)!0.5!(b')$);
\begin{pgfonlayer}{background}
\fill[catd] (b) -- (b') -- (c') -- (c) -- cycle;
\fill[catc] (a') -- (a) -- (c) -- (c') 
 -- (cbb.north east) -- (cbb.south east) -- cycle;
\fill[catterm] (a') -- (a) -- (b) -- (b')
 -- (cbb.north west) -- (cbb.south west) -- cycle;
\end{pgfonlayer}
\draw (a') -- (a)
      (b) -- (h) -- (b')
      (c) -- (c');
\node[rectangle, fit=(bl)(tr)] (subdiag) {};
\lmorph{subdiag};
\draw[very thick] (subdiag.south west) rectangle (subdiag.north east);
\end{tikzpicture}
\end{gathered}
\stackrel{\eqref{eq:adjFtoG}}{=}
\begin{gathered}
\begin{tikzpicture}[scale=0.5]
\path coordinate[dot, label=left:$f$] (f)
 +(0,-3) coordinate[label=below:$X$] (b)
 +(1,-1) coordinate[dot, label=below:$\eta$] (eta)
 ++(0,2) coordinate[dot, label=left:$h$] (h)
 ++(0,1) coordinate[label=above:$Z$] (tl)
 ++(3,0) coordinate[label=above:$G$] (tr);
\draw (b) -- (f) -- (h) -- (tl)
 (f) to[out=-90, in=180] (eta) to[out=0, in=-90] (tr);
\begin{pgfonlayer}{background}
\fill[catterm] ($(b) + (-1,0)$) rectangle (tl);
\fill[catc] (b) rectangle ($(tr) + (1,0)$);
\fill[catd] (f) to[out=-90, in=180] (eta) to[out=0, in=-90] (tr) -- (tl) -- cycle;
\end{pgfonlayer}
\end{tikzpicture}
\end{gathered}
\stackrel{\eqref{eq:adjFtoG}}{=}
\begin{gathered}
\begin{tikzpicture}[auto, x=0.5cm, y=0.5cm]
\diagbb{2.5cm}{3cm};
\path (cbb.center)
      -- +(-1.5, -2) coordinate (bl)
      -- +(1.5,2) coordinate (tr)
      -- +(0,-1.5) coordinate (a)
      -- +(-1,1.5) coordinate (b)
      -- +(1, 1.5) coordinate (c);
\path let \p1 = (cbb.south) in
      let \p2 = (cbb.north) in
      let \p3 = (a) in
      let \p4 = (b) in
      let \p5 = (c) in
 coordinate[label=below:$X$] (a') at (\x3, \y1)
 coordinate[label=above:$Z$] (b') at (\x4, \y2)
 coordinate[label=above:$G$] (c') at (\x5, \y2);
\begin{pgfonlayer}{background}
\fill[catd] (b) -- (b') -- (c') -- (c) -- cycle;
\fill[catc] (a') -- (a) -- (c) -- (c') 
 -- (cbb.north east) -- (cbb.south east) -- cycle;
\fill[catterm] (a') -- (a) -- (b) -- (b')
 -- (cbb.north west) -- (cbb.south west) -- cycle;
\end{pgfonlayer}
\draw (a') -- (a)
      (b) -- (b')
      (c) -- (c');
\node[rectangle, fit=(bl)(tr)] (subdiag) {};
\lmorphex{subdiag};
\draw[very thick] (subdiag.south west) rectangle (subdiag.north east);
\end{tikzpicture}
\end{gathered}
\end{equation*}
Naturality of the second map then follows from lemma \ref{lem:pushpop}.
We remark that the usual bijection, natural in $X$ and $Y$ 
induced by an adjunction:
\begin{center}
\begin{prooftree}
\AxiomC{$F(X) \rightarrow Y$}
\doubleLine
\UnaryInfC{$X \rightarrow G(Y)$}
\end{prooftree}
\end{center}
can also be seen as an immediate corollary of lemma \ref{lem:mvnat}.
\end{example}

\begin{example}[Kan Extensions]
A novel graphical notation for Kan extensions was introduced in \citep{Hinze2012}.
The paper then develops many proofs using both a traditional symbolic style and
a string diagram based graphical approach, illustrating the compactness and efficiency
of the latter.

We will examine Kan extensions
as an example application of the representability based results we have developed
in earlier sections.
Given functors $J : \mathcal{B} \rightarrow \mathcal{A}$ and $F : \mathcal{B} \rightarrow \mathcal{C}$,
$F$ is said to have a left Kan extension along $J$ if the functor
$[\mathcal{B}, \mathcal{C}](F, (-) \circ J)$ is representable.
Then we have:
\begin{equation*}
[\mathcal{B}, \mathcal{C}](F, (-) \circ J) \cong [\mathcal{A}, \mathcal{C}](\lkan{F}{J}, -)
\end{equation*}
Dually, $F$ has a right Kan extension along $J$ if we have the following natural isomorphism:
\begin{equation*}
[\mathcal{B}, \mathcal{C}]((-) \circ J, F) \cong [\mathcal{A}, \mathcal{C}](-, \rkan{F}{J})
\end{equation*}
If we denote the counit as $c$, in the case of left Kan extensions, 
the universality results of section \ref{sec:universality} specialize to the following axiom, giving the usual
unique factorization property of left Kan extensions:
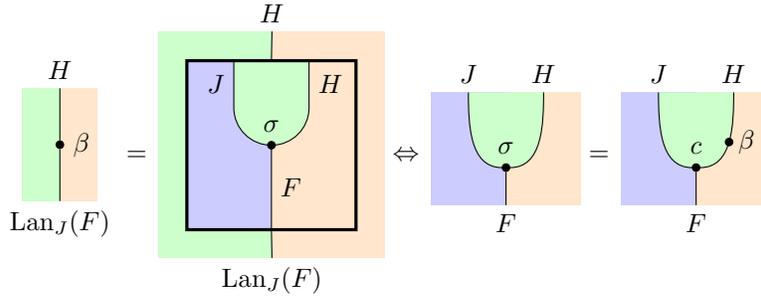
\newcommand{\lkanmor}[5]{
\path (#1) coordinate (#1-sigma)
      +(-1,1) coordinate (#1-a)
      +(1,1) coordinate (#1-b);
\path let \p1 = (#1.north) in
      let \p2 = (#1.south) in
      let \p3 = (#1-sigma) in
      let \p4 = (#1-a) in
      let \p5 = (#1-b) in
 coordinate (#1-d) at (\x3, \y2)
 coordinate (#1-a') at (\x4, \y1)
 coordinate (#1-b') at (\x5, \y1);
\fill[catc] (#1.south west) rectangle (#1.north east);
\fill[catd] (#1-a') -- (#1-a) to[out=270, in=180] (#1-sigma) to[out=0, in=270] (#1-b) -- (#1-b') -- cycle;
\fill[cate] (#1-d) -- (#1-sigma) to[out=0, in=270] (#1-b) -- (#1-b') -- (#1.north east) -- (#1.south east) -- cycle;
\draw (#1-a') to node[swap]{#3} (#1-a) to[out=270, in=180] (#1-sigma) to[out=0, in=270] (#1-b) to node[swap]{#5} (#1-b')
      (#1-sigma) to node{#4} (#1-d);
\strnat{#1-sigma};
\strlabu{#1-sigma}{#2}
}
\begin{equation*}
\begin{gathered}
\twocelldiag{\beta}{\lkan{F}{J}}{H}{catd}{cate}{1}{1}{1.5}
\end{gathered}
=
\begin{gathered}
\begin{tikzpicture}[auto,x=0.5cm,y=0.5cm]
\diagbb{3cm}{3cm};
\path (cbb.south west) 
      ++(1,1) coordinate (bl)
      ++(4,4.0) coordinate (tr)
      (cbb.north west) 
      ++(2,0) coordinate (a') 
      ++(0,-1) coordinate (a)
      (cbb.north east)
      ++(-4,0) coordinate (b')
      ++(0,-1) coordinate (b);
\node[rectangle, fit=(bl)(tr)] (subdiag) {};
\fillall{cate};
\fillwest{catd}{cbb.south}{cbb.north};
\draw (cbb.south) coordinate[label=below:\lkan{F}{J}] (bot) -- (subdiag.south) -- (subdiag.north) -- (cbb.north) coordinate[label=above:$H$] (top);
\lkanmor{subdiag}{$\sigma$}{$J$}{$F$}{$H$};
\draw[very thick] (subdiag.south west) rectangle (subdiag.north east);
\end{tikzpicture}
\end{gathered}
\Leftrightarrow
\begin{gathered}
\diagrepto{\sigma}{F}{J}{H}{catc}{catd}{cate}{1}{2}
\end{gathered}
=
\begin{gathered}
\begin{tikzpicture}[scale=0.5]
\path coordinate[dot, label=above:$c$] (c) 
 +(-1,2) coordinate[label=above:$J$] (tl)
 +(0,-1) coordinate[label=below:$F$] (b)
 +(1,2) coordinate[label=above:$H$] (tr);
\draw (b) -- (c);
\draw (tl) to[out=-90, in=180] (c.west) -- (c.east) to[out=0, in=-90] coordinate[dot, midway, label=right:$\beta$] (beta) (tr);
\begin{pgfonlayer}{background}
\fill[catc] ($(tl) + (-1,0)$) rectangle (b);
\fill[cate] ($(tr) + (1,0)$) rectangle (b);
\fill[catd] (tl) to[out=-90, in=180] (c.west) -- (c.east) to[out=0, in=-90] (tr) -- cycle;
\end{pgfonlayer}
\end{tikzpicture}
\end{gathered}
\end{equation*}
We recover the calculation laws of \citep{Hinze2012} as follows:
\begin{itemize}
 \item The computation law follows immediately from the universal property above
 \item The reflection law follows directly from the definition of the counit
 \item The fusion law is an instance of the ``push / pop'' identity in equation \eqref{eq:thetaone} 
\end{itemize}
\end{example}

\section{Limits and colimits}
\label{sec:limits}
Limits and colimits are key notions in category theory. They can be approached from the perspective
of representability introduced in section \ref{sec:repr}. We will start with the general setting,
and then provide specialized results and notation for some common cases.

\subsection{Arbitrary Limits and Colimits}
\label{sec:arblimits}
We first define a standard functor used when reasoning about limits and colimits.
\begin{definition}
Let $\mathcal{C}$, $\mathcal{D}$ be categories, then we define the functor 
$\Delta : \mathcal{C} \rightarrow [\mathcal{D}, \mathcal{C}]$ as taking objects to
the corresponding constant functor, and morphism $f$ to the natural transformation with
all components equal to $f$.
\end{definition}
Now we define limits and colimits in terms of representability of appropriate functors.
\begin{definition}
Let $\mathcal{C}$, $\mathcal{D}$ be categories, and the functor $D:\mathcal{D} \rightarrow \mathcal{C}$ a
diagram in $\mathcal{C}$. The limit of $D$ exists if the functor $[\mathcal{D}, \mathcal{C}](\Delta(-),D)$ is
representable, i.e.
\begin{equation*}
[\mathcal{D}, \mathcal{C}](\Delta(-),D) \cong \mathcal{C}(-,\limit{D})
\end{equation*}
The limiting cone then corresponds to the identity morphism on \limit{D}.

Dually, the colimit of $D$ exists if the functor $[\mathcal{D}, \mathcal{C}](D,\Delta(-))$ is representable, i.e.
\begin{equation*}
[\mathcal{D}, \mathcal{C}](D, \Delta(-)) \cong \mathcal{C}(\colimit{D},-)
\end{equation*}
\end{definition}
As the existence of limits and colimits corresponds to representability of certain functors, we can use the calculation
rules in section \ref{sec:repr}. We can use our additional knowledge of the structure of the functor on the left hand side to
phrase the calculation laws in a more convenient form. For limits we have:
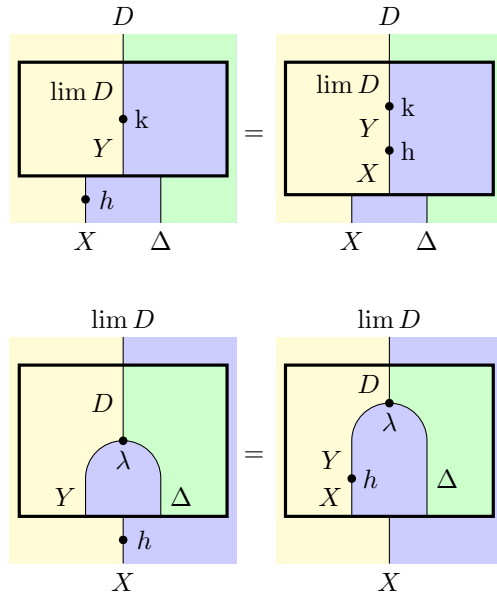
\begin{equation*}
\begin{gathered}
\begin{tikzpicture}[auto, x=0.5cm, y=0.5cm]
\diagbb{3cm}{2.5cm};
\path (cbb.north west) -- ++(0.5,-1) coordinate (tl)
      (cbb.south east) -- ++(-0.5,1.5) coordinate (br);
\node[rectangle, fit=(br)(tl)] (subdiag) {};
\path (subdiag.north) 
      coordinate (a)
      (subdiag.south)
      +(-1,0) coordinate (b)
      +(1,0) coordinate (c);
\path let \p1 = (cbb.north) in
      let \p2 = (cbb.south) in
      let \p3 = (a) in
      let \p4 = (b) in
      let \p5 = (c) in
 coordinate[label=above:$D$] (a') at (\x3,\y1)
 coordinate[label=below:$X$] (b') at (\x4,\y2)
 coordinate[label=below:$\Delta$] (c') at (\x5,\y2)
 coordinate[dot, label=right:$h$] (h) at ($(b)!0.5!(b')$);
\draw (a) -- (a')
      (b') -- (h) -- (b)
      (c') -- (c);
\repfrom{subdiag}{$Y$}{\limit{D}}{k};
\draw[very thick] (subdiag.south west) rectangle (subdiag.north east);
\begin{pgfonlayer}{background}
\fill[catc] (cbb.south west) rectangle (cbb.north east);
\fill[catterm] (b') -- (b) -- (a) -- (a') -- (cbb.north west) -- (cbb.south west) -- cycle;
\fill[catd] (c') -- (c) -- (a) -- (a') -- (cbb.north east) -- (cbb.south east) -- cycle;
\end{pgfonlayer}
\end{tikzpicture}
\end{gathered}
=
\begin{gathered}
\begin{tikzpicture}[auto, x=0.5cm, y=0.5cm]
\diagbb{3cm}{2.5cm};
\path (cbb.north west) -- ++(0.5,-1) coordinate (tl)
      (cbb.south east) -- ++(-0.5,1) coordinate (br);
\node[rectangle, fit=(br)(tl)] (subdiag) {};
\path (subdiag.north) 
      coordinate (a)
      (subdiag.south)
      +(-1,0) coordinate (b)
      +(1,0) coordinate (c);
\path let \p1 = (cbb.north) in
      let \p2 = (cbb.south) in
      let \p3 = (a) in
      let \p4 = (b) in
      let \p5 = (c) in
 coordinate[label=above:$D$] (a') at (\x3,\y1)
 coordinate[label=below:$X$] (b') at (\x4,\y2)
 coordinate[label=below:$\Delta$] (c') at (\x5,\y2);
\draw (a) -- (a')
      (b') -- (b)
      (c') -- (c);
\repfromcomp{subdiag}{$X$}{$Y$}{\limit{D}}{h}{k};
\draw[very thick] (subdiag.south west) rectangle (subdiag.north east);
\begin{pgfonlayer}{background}
\fill[catc] (cbb.south west) rectangle (cbb.north east);
\fill[catterm] (b') -- (b) -- (a) -- (a') -- (cbb.north west) -- (cbb.south west) -- cycle;
\fill[catd] (c') -- (c) -- (a) -- (a') -- (cbb.north east) -- (cbb.south east) -- cycle;
\end{pgfonlayer}
\end{tikzpicture}
\end{gathered}
\end{equation*}
\newcommand{\reptoi}[5]{
\path (#1.center) coordinate (#1-r)
      +(-1,-1) coordinate (#1-a)
      +(1,-1) coordinate (#1-b)
      (#1.north) coordinate (#1-t);
\path let \p1 = (subdiag.north) in
      let \p2 = (subdiag.south) in
      let \p3 = (#1-a) in
      let \p4 = (#1-b) in
 coordinate (#1-a') at (\x3, \y2)
 coordinate (#1-b') at (\x4, \y2);
\fill[catterm] (#1.south west) rectangle (#1.north east);
\fill[catc] (#1-a') -- (#1-a) to[out=90, in=180] (#1-r) to[out=0, in=90] (#1-b) -- (#1-b') -- cycle;
\fill[catd] (#1-b') -- (#1-b) to[out=90, in=0] (#1-r) -- (#1-t) -- (#1.north east) -- (#1.south east) -- cycle;
\draw (#1-a') to node{#2} (#1-a) to[out=90, in=180] (#1-r) to[out=0, in=90] (#1-b) to node{#3} (#1-b')
      (#1-r) to node{#5} (#1-t);
\strnat{#1-r};
\strlabd{#1-r}{#4};
}
\newcommand{\reptoicomp}[7]{
\path (#1.center) ++(0,1) coordinate (#1-r)
      +(-1,-1) coordinate (#1-a)
      +(1,-1) coordinate (#1-b)
      (#1.north) coordinate (#1-t);
\path let \p1 = (subdiag.north) in
      let \p2 = (subdiag.south) in
      let \p3 = (#1-a) in
      let \p4 = (#1-b) in
 coordinate (#1-a') at (\x3, \y2)
 coordinate (#1-b') at (\x4, \y2)
 coordinate (#1-h) at ($(#1-a')!0.5!(#1-a)$);
\fill[catterm] (#1.south west) rectangle (#1.north east);
\fill[catc] (#1-a') -- (#1-a) to[out=90, in=180] (#1-r) to[out=0, in=90] (#1-b) -- (#1-b') -- cycle;
\fill[catd] (#1-b') -- (#1-b) to[out=90, in=0] (#1-r) -- (#1-t) -- (#1.north east) -- (#1.south east) -- cycle;
\draw (#1-a') to node{#2} (#1-h) to node{#3} (#1-a) to[out=90,in=180] (#1-r) to[out=0,in=90] (#1-b) to node{#4} (#1-b')
      (#1-r) to node{#7} (#1-t);
\strnat{#1-r};
\strlabd{#1-r}{#5};
\strnat{#1-h};
\strlabr{#1-h}{#6};
}
\begin{equation*}
\begin{gathered}
\begin{tikzpicture}[auto, x=0.5cm, y=0.5cm]
\diagbb{3cm}{3cm};
\path (cbb.north west) -- ++(0.5,-1) coordinate (tl)
      (cbb.south east) -- ++(-0.5,1.5) coordinate (br);
\node[rectangle, fit=(br)(tl)] (subdiag) {};
\path (subdiag.north) coordinate (a)
      (cbb.north) coordinate[label=above:\limit{D}] (a')
      (subdiag.south) coordinate (b)
      (cbb.south) coordinate[label=below:$X$] (b')
      coordinate[dot, label=right:$h$] (h) at ($(b)!0.5!(b')$);
\draw (a) -- (a')
      (b') -- (h) -- (b);
\reptoi{subdiag}{$Y$}{$\Delta$}{$\lambda$}{$D$};
\draw[very thick] (subdiag.south west) rectangle (subdiag.north east);
\begin{pgfonlayer}{background}
\fill[catterm] (cbb.south west) rectangle (cbb.north);
\fill[catc] (cbb.south east) rectangle (cbb.north);
\end{pgfonlayer}
\end{tikzpicture}
\end{gathered}
=
\begin{gathered}
\begin{tikzpicture}[auto, x=0.5cm, y=0.5cm]
\diagbb{3cm}{3cm};
\path (cbb.north west) -- ++(0.5,-1) coordinate (tl)
      (cbb.south east) -- ++(-0.5,1.5) coordinate (br);
\node[rectangle, fit=(br)(tl)] (subdiag) {};
\path (subdiag.north) coordinate (a)
      (cbb.north) coordinate[label=above:\limit{D}] (a')
      (subdiag.south) coordinate (b)
      (cbb.south) coordinate[label=below:$X$] (b')
      coordinate (h) at ($(b)!0.5!(b')$);
\draw (a) -- (a')
      (b') -- (b);
\reptoicomp{subdiag}{$X$}{$Y$}{$\Delta$}{$\lambda$}{$h$}{$D$};
\draw[very thick] (subdiag.south west) rectangle (subdiag.north east);
\begin{pgfonlayer}{background}
\fill[catterm] (cbb.south west) rectangle (cbb.north);
\fill[catc] (cbb.south east) rectangle (cbb.north);
\end{pgfonlayer}
\end{tikzpicture}
\end{gathered}
\end{equation*}
For colimits the following equations hold:
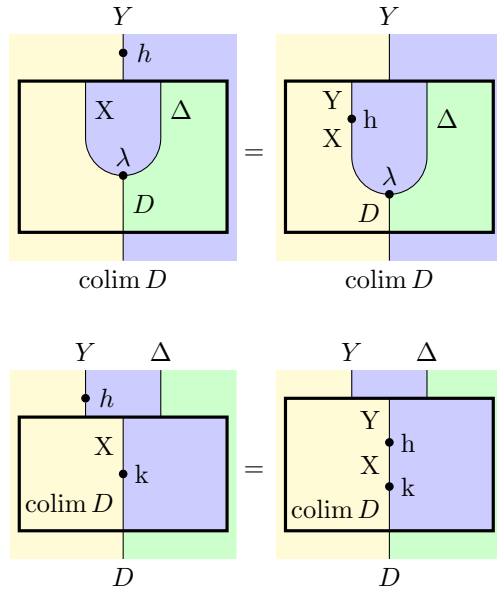
\begin{equation*}
\begin{gathered}
\begin{tikzpicture}[auto, x=0.5cm, y=0.5cm]
\diagbb{3cm}{3cm};
\path (cbb.south west) -- ++(0.5,1) coordinate (bl)
      (cbb.north east) -- ++(-0.5,-1.5) coordinate (tr)
      (cbb.north) -- ++(0,-0.5) coordinate[dot, label=right:$h$] (h);
\draw (cbb.north) coordinate[label=above:$Y$] -- (h) -- ++(0,-1.5)
      (cbb.south) coordinate[label=below:\colimit{D}] -- ++(0,1.5);
\node[rectangle, fit=(bl)(tr)] (subdiag) {};
\begin{scope}[catset/.style={color=green!20}]
\repto{subdiag}{X}{$\Delta$}{$\lambda$}{$D$};
\end{scope}
\draw[very thick] (subdiag.south west) rectangle (subdiag.north east);
\begin{pgfonlayer}{background}
\fill[catterm] (cbb.south west) rectangle (cbb.north);
\fill[catc] (cbb.south east) rectangle (cbb.north);
\end{pgfonlayer}
\end{tikzpicture}
\end{gathered}
=
\begin{gathered}
\begin{tikzpicture}[auto, x=0.5cm, y=0.5cm]
\diagbb{3cm}{3cm};
\path (cbb.south west) -- ++(0.5,1) coordinate (bl)
      (cbb.north east) -- ++(-0.5,-1.5) coordinate (tr);
\draw (cbb.south) coordinate[label=below:\colimit{D}] -- ++(0,1.5)
      (cbb.north) coordinate[label=above:$Y$] -- ++(0, -1.5);
\node[rectangle, fit=(bl)(tr)] (subdiag) {};
\begin{scope}[catset/.style={color=green!20}]
\reptocomp{subdiag}{X}{Y}{$\Delta$}{$\lambda$}{h}{$D$};
\end{scope}
\draw[very thick] (subdiag.south west) rectangle (subdiag.north east);
\begin{pgfonlayer}{background}
\fill[catterm] (cbb.south west) rectangle (cbb.north);
\fill[catc] (cbb.south east) rectangle (cbb.north);
\end{pgfonlayer}
\end{tikzpicture}
\end{gathered}
\end{equation*}
\begin{equation*}
\begin{gathered}
\begin{tikzpicture}[auto, x=0.5cm, y=0.5cm]
\diagbb{3cm}{2.5cm};
\path (cbb.south west) -- ++(0.5,1) coordinate (bl)
      (cbb.north east) -- ++(-0.5,-1.5) coordinate (tr)
      (cbb.north) -- ++(-1,0) coordinate[label=above:$Y$] (a')
      (cbb.north) -- ++(1,0) coordinate[label=above:$\Delta$] (b');
\path (a') -- ++(0,-1.5) coordinate (a)
      (b') -- ++(0,-1.5) coordinate (b);
\coordinate[dot, label=right:$h$] (h) at ($(a)!0.5!(a')$);
\draw (a) -- (h) -- (a')
      (b) -- (b')
      (cbb.south) coordinate[label=below:$D$] -- ++(0,1.5);
\node[rectangle, fit=(bl)(tr)] (subdiag) {};
\repfrom{subdiag}{\colimit{D}}{X}{k};
\draw[very thick] (subdiag.south west) rectangle (subdiag.north east);
\begin{pgfonlayer}{background}
\fill[catterm] (cbb.south) -- (cbb.center) -- (a)
  -- (a') -- (cbb.north west) -- (cbb.south west) -- cycle;
\fill[catd] (cbb.south) -- (cbb.center) -- (b)
  -- (b') -- (cbb.north east) -- (cbb.south east) -- cycle;
\fill[catc] (a) rectangle (b');
\end{pgfonlayer}
\end{tikzpicture}
\end{gathered}
=
\begin{gathered}
\begin{tikzpicture}[auto, x=0.5cm, y=0.5cm]
\diagbb{3cm}{2.5cm};
\path (cbb.north)
      -- +(-1,0) coordinate[label=above:$Y$] (a')
      -- +(1,0) coordinate[label=above:$\Delta$] (b');
\path (a') -- ++(0,-2) coordinate (a)
      (b') -- ++(0,-2) coordinate (b);
\path (cbb.south west) -- ++(0.5,1) coordinate (bl)
      (cbb.north east) -- ++(-0.5,-1) coordinate (tr);
\draw (a) -- (a')
      (b) -- (b')
      (cbb.south) coordinate[label=below:$D$] -- ++(0,1.5);
\node[rectangle, fit=(bl)(tr)] (subdiag) {};
\repfromcomp{subdiag}{\colimit{D}}{X}{Y}{k}{h};
\draw[very thick] (subdiag.south west) rectangle (subdiag.north east);
\begin{pgfonlayer}{background}
\fill[catterm] (cbb.south) -- (cbb.center) -- (a) -- (a') 
 -- (cbb.north west) -- (cbb.south west) -- cycle;
\fill[catd] (cbb.south) -- (cbb.center) -- (b) -- (b') 
 -- (cbb.north east) -- (cbb.south east) -- cycle;
\fill[catc] (a) rectangle (b');
\end{pgfonlayer}
\end{tikzpicture}
\end{gathered}
\end{equation*}
We will specialize the universality result described in section \ref{sec:universality} when
we discuss some specific limits and colimits in later sections.

\subsection{Specific Limits and Colimits}
We now examine a few specific limits and colimits. In these concrete cases we can
specialize the results in section \ref{sec:arblimits} and provide more convenient
notation and calculation rules for some common cases.

\subsubsection*{Initial Objects}
The initial object will be denoted $0$ and
we will use $!: 0 \rightarrow X$ to denote the unique morphism from the initial object to an object $X$.
The universal property for initial objects can be expressed by the following relationship, for each object $X$ in $\mathcal{C}$:
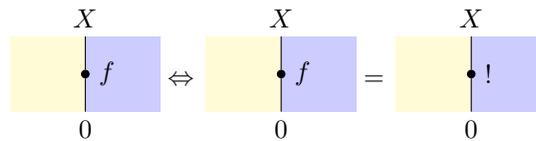
\begin{equation*}
\begin{gathered}
\twocelldiag{f}{0}{X}{catterm}{catc}{2}{2}{1}
\end{gathered}
\Leftrightarrow
\begin{gathered}
\twocelldiag{f}{0}{X}{catterm}{catc}{2}{2}{1}
\end{gathered}
=
\begin{gathered}
\twocelldiag{!}{0}{X}{catterm}{catc}{2}{2}{1}
\end{gathered}
\end{equation*}
We will now further specialize this rule for initial algebras, and as an example calculation apply the rule to provide
a graphical proof of Lambek's lemma. 
Initial algebras, and their dual notion terminal coalgebras, are of interested in modelling of datatypes
\citep{GoguenThatcherWagnerWright1975, GoguenThatcherWagnerWright1977, Hagino1987, Hagino1987b, Malcolm1990} and \citep{Malcolm1990b}.
\begin{example}[Initial Algebras and Lambek's Lemma]
\label{ex:initalg}
For endofunctor $T : \mathcal{C} \rightarrow \mathcal{C}$, if $T$ has an initial algebra $(\initalgbase{T}, \initalg{T})$ then
we can rewrite the initiality condition in a more useful form, with the left hand equation in $\mathcal{C}$
and the right hand equality in \algcat{T}. We also adopt some standard notation in this context and denote the
unique morphism from the initial algebra to an algebra $a$ as \fold{T}{a}.
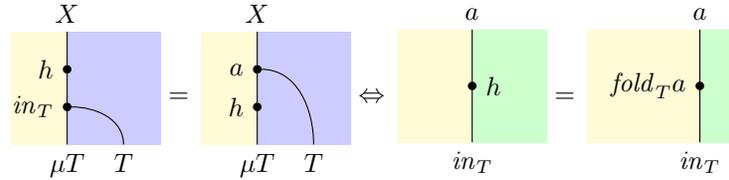
\begin{equation*}
\begin{gathered}
\begin{tikzpicture}[scale=0.5]
\path coordinate[dot, label=left:\initalg{T}]  (init)
 +(0,-1) coordinate[label=below:\initalgbase{T}] (bl)
 +(1.5,-1) coordinate[label=below:$T$] (br)
 ++(0,1) coordinate[dot, label=left:$h$] (h)
 ++(0,1) coordinate[label=above:$X$] (t);
\draw (bl) -- (init) -- (h) -- (t)
 (br) to[out=90, in=0] (init);
\begin{pgfonlayer}{background}
\fill[catterm] ($(t) + (-1.5,0)$) rectangle (bl);
\fill[catc] (t) rectangle ($(br) + (1,0)$);
\end{pgfonlayer}
\end{tikzpicture}
\end{gathered}
=
\begin{gathered}
\begin{tikzpicture}[scale=0.5]
\path coordinate[dot, label=left:$h$] (h)
 +(0,-1) coordinate[label=below:\initalgbase{T}] (bl)
 +(1.5,-1) coordinate[label=below:$T$] (br)
 ++(0,1) coordinate[dot, label=left:$a$] (a)
 ++(0,1) coordinate[label=above:$X$] (t);
\draw (bl) -- (h) -- (a) -- (t)
 (br) to[out=90, in=0] (a);
\begin{pgfonlayer}{background}
\fill[catterm] ($(t) + (-1.5,0)$) rectangle (bl);
\fill[catc] (t) rectangle ($(br) + (1,0)$);
\end{pgfonlayer}
\end{tikzpicture}
\end{gathered}
\Leftrightarrow 
\begin{gathered}
\twocelldiag{h}{\initalg{T}}{a}{catterm}{catd}{2}{2}{1.5}
\end{gathered}
=
\begin{gathered}
\begin{tikzpicture}[scale=0.5]
\path coordinate[dot, label=left:$\fold{T}a$] (fold)
 +(0,-1.5) coordinate[label=below:\initalg{T}] (b)
 +(0,1.5) coordinate[label=above:$a$] (t);
\draw (b) -- (fold) -- (t);
\begin{pgfonlayer}{background}
\fill[catterm] ($(b) + (-3.0,0)$) rectangle (t);
\fill[catd] (t) rectangle ($(b) + (1,0)$);
\end{pgfonlayer}
\end{tikzpicture}
\end{gathered}
\end{equation*}
It is immediately obvious that $\fold{T}{\initalg{T}} = 1_{\initalgbase{T}}$.
We can then prove Lambek's lemma, that the initial algebra is an isomorphism.
We have:
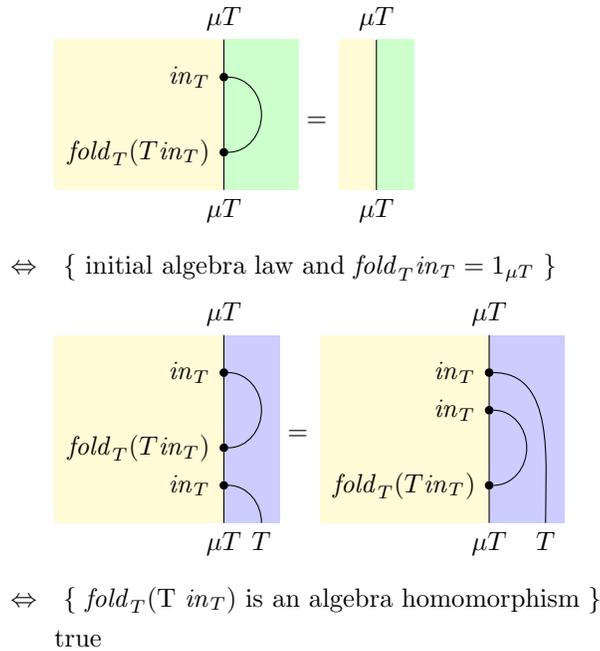
\begin{eqproof*}
\begin{gathered}
\begin{tikzpicture}[scale=0.5]
\path coordinate[dot, label=left:\initalg{T}] (init)
 +(0,1) coordinate[label=above:\initalgbase{T}] (t)
 +(1,-1) coordinate (a)
 ++(0,-2) coordinate[dot, label=left:$\fold{T}(T \initalg{T})$] (fold)
 ++(0,-1) coordinate[label=below:\initalgbase{T}] (b);
\draw (b) -- (fold) -- (init) -- (t);
\draw (fold) to[out=0, in=-90] (a) to[out=90, in=0] (init);
\begin{pgfonlayer}{background}
\fill[catterm] ($(t) + (-4.5,0)$) rectangle (b);
\fill[catd] (t) rectangle ($(b) + (2,0)$);
\end{pgfonlayer}
\end{tikzpicture}
\end{gathered}
= 
\begin{gathered}
\onecelldiag{\initalgbase{T}}{catterm}{catd}{4}{1}
\end{gathered}
\explain[\Leftrightarrow]{initial algebra law and $\fold{T}{\initalg{T}} = 1_{\initalgbase{T}}$}
\begin{gathered}
\begin{tikzpicture}[scale=0.5]
\path coordinate[dot, label=left:\initalg{T}] (init1)
 +(0,-1) coordinate[label=below:\initalgbase{T}] (bl)
 +(1,-1) coordinate[label=below:$T$] (br)
 ++(0,1) coordinate[dot, label=left:$\fold{T}(T \initalg{T})$] (fold)
 ++(0,2) coordinate[dot, label=left:\initalg{T}] (init2)
 +(1,-1) coordinate (a)
 ++(0,1) coordinate[label=above:\initalgbase{T}] (t);
 \draw (bl) -- (init1) -- (fold) -- (init2) -- (t)
 (br) to[out=90, in=0] (init1)
 (fold) to[out=0, in=-90] (a) to[out=90, in=0] (init2);
\begin{pgfonlayer}{background}
\fill[catterm] ($(t) + (-4.5,0)$) rectangle (bl);
\fill[catc] (t) rectangle ($(br) + (0.5,0)$);
\end{pgfonlayer}
\end{tikzpicture}
\end{gathered}
=
\begin{gathered}
\begin{tikzpicture}[scale=0.5]
\path coordinate[dot, label=left:$\fold{T}(T \initalg{T})$] (fold)
 +(0,-1) coordinate[label=below:\initalgbase{T}] (bl)
 +(1.5,-1) coordinate[label=below:$T$] (br)
 ++(0,2) coordinate[dot, label=left:\initalg{T}] (init1)
 +(1,-1) coordinate (a)
 ++(0,1) coordinate[dot, label=left:\initalg{T}] (init2)
 ++(0,1) coordinate[label=above:\initalgbase{T}] (t);
\draw (bl) -- (fold) -- (init1) -- (init2) -- (t)
 (br) to[out=90, in=0] (init2)
 (fold) to[out=0, in=-90] (a) to[out=90, in=0] (init1);
\begin{pgfonlayer}{background}
\fill[catterm] ($(t) + (-4.5, 0)$) rectangle (bl);
\fill[catc] (t) rectangle ($(br) + (0.5,0)$);
\end{pgfonlayer}
\end{tikzpicture}
\end{gathered}
\explain[\Leftrightarrow]{\fold{T}{$(T \initalg{T})$} is an algebra homomorphism}
\true
\end{eqproof*}
For the other direction we have:
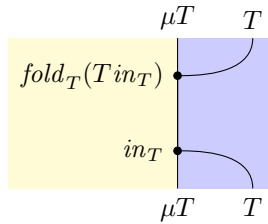
\begin{eqproof*}
\begin{tikzpicture}[scale=0.5]
\path coordinate[dot, label=left:\initalg{T}] (init)
 +(0,-1) coordinate[label=below:\initalgbase{T}] (bl)
 +(2,-1) coordinate[label=below:$T$] (br)
 ++(0,2) coordinate[dot, label=left:$\fold{T}(T \initalg{T})$] (fold)
 +(0,1) coordinate[label=above:\initalgbase{T}] (tl)
 +(2,1) coordinate[label=above:$T$] (tr);
\draw (bl) -- (fold) -- (init) -- (tl)
 (br) to[out=90, in=0] (init)
 (fold) to[out=0, in=-90] (tr);
\begin{pgfonlayer}{background}
\fill[catterm] ($(tl) + (-4.5, 0)$) rectangle (bl);
\fill[catc] (tl) rectangle ($(br) + (0.5, 0)$);
\end{pgfonlayer}
\end{tikzpicture}
\explain{\fold{T}{$(T \initalg{T})$} is an algebra homomorphism}
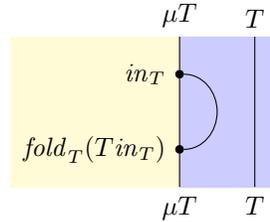
\begin{tikzpicture}[scale=0.5]
\path coordinate[dot, label=left:$\fold{T}(T \initalg{T})$] (fold)
 +(0,-1) coordinate[label=below:\initalgbase{T}] (bl)
 +(2,-1) coordinate[label=below:$T$] (br)
 ++(0,2) coordinate[dot, label=left:\initalg{T}] (init)
 +(1,-1) coordinate (a)
 +(0,1) coordinate[label=above:\initalgbase{T}] (tl)
 +(2,1) coordinate[label=above:$T$] (tr);
\draw (bl) -- (fold) -- (init) -- (tl)
 (br) -- (tr)
 (fold) to[out=0, in=-90] (a) to[out=90, in=0] (init);
\begin{pgfonlayer}{background}
\fill[catterm] ($(tl) + (-4.5, 0)$) rectangle (bl);
\fill[catc] (tl) rectangle ($(br) + (0.5, 0)$);
\end{pgfonlayer}
\end{tikzpicture}
\explain{previous part}
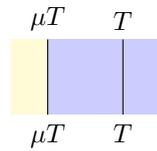
\begin{tikzpicture}[auto, x=0.5cm, y=0.5cm]
\path coordinate[label=above:\initalgbase{T}] (tl)
 +(2,0) coordinate[label=above:$T$] (tr)
 ++(0,-2) coordinate[label=below:\initalgbase{T}] (bl)
 ++(2,0) coordinate[label=below:$T$] (br);
\draw (bl) -- (tl)
      (br) -- (tr);
\begin{pgfonlayer}{background}
\fill[catterm] ($(tl) + (-1,0)$) rectangle (bl);
\fill[catc] (tl) rectangle ($(br) + (1,0)$);
\end{pgfonlayer}
\end{tikzpicture}
\end{eqproof*}
\end{example}

\subsubsection*{Binary Products}
\label{sec:binprod}
The binary product of objects $Y$ and $Z$ will be written $Y \times Z$, in the usual way usual. 
The representability condition gives a bijective correspondence between morphisms $X \rightarrow Y \times Z$ and 
pairs of morphisms $X \rightarrow Y$ and $X \rightarrow Z$.
To aid calculations we will introduce three different maps:
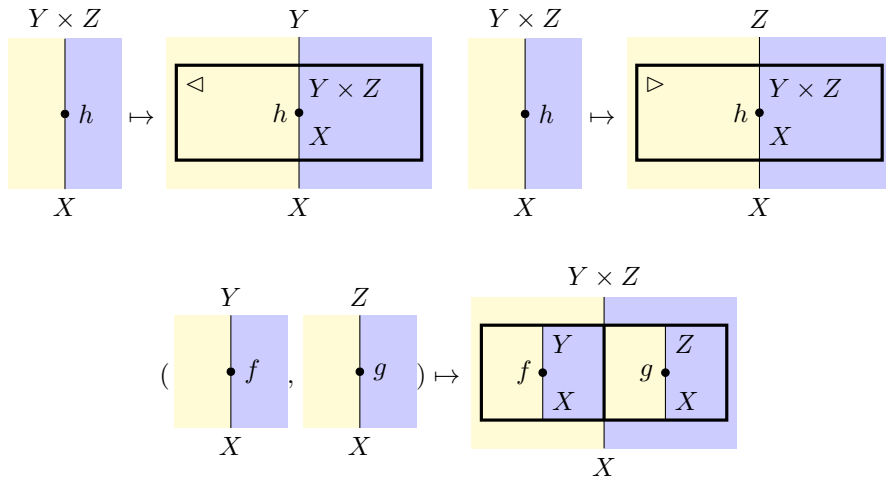
\begin{trivlist}
\item
\begin{minipage}{0.495\textwidth}
\begin{equation*}
\begin{gathered}
\twocelldiag{h}{X}{Y \times Z}{catterm}{catc}{1.5}{1.5}{2}
\end{gathered}
\mapsto
\begin{gathered}
\begin{tikzpicture}[auto, x=0.5cm, y=0.5cm]
\diagbb{3.5cm}{2cm};
\path (cbb.south west) ++(0.5,1) coordinate (bl)
      (cbb.north east) ++(-0.5,-1) coordinate (tr);
\node[rectangle, fit=(bl)(tr)] (subdiag) {};
\begin{scope}
\path[clip] (subdiag.south west) rectangle (subdiag.north east);
\singlemor{subdiag}{$h$}{$X$}{$Y \times Z$}{catterm}{catc};
\path (subdiag.north west) ++(0.5,-0.5) node {$\lhd$};
\end{scope}
\draw[very thick] (subdiag.south west) rectangle (subdiag.north east);
\draw (cbb.south) coordinate[label=below:$X$] -- (subdiag.south)
      (subdiag.north)  -- (cbb.north) coordinate[label=above:$Y$];
\begin{pgfonlayer}{background}
\fill[catterm] (cbb.north west) rectangle (cbb.south);
\fill[catc] (cbb.north east) rectangle (cbb.south);
\end{pgfonlayer}
\end{tikzpicture}
\end{gathered}
\end{equation*}
\end{minipage}
\begin{minipage}{0.495\textwidth}
\begin{equation*}
\begin{gathered}
\twocelldiag{h}{X}{Y \times Z}{catterm}{catc}{1.5}{1.5}{2}
\end{gathered}
\mapsto
\begin{gathered}
\begin{tikzpicture}[auto, x=0.5cm, y=0.5cm]
\diagbb{3.5cm}{2cm};
\path (cbb.south west) ++(0.5,1) coordinate (bl)
      (cbb.north east) ++(-0.5,-1) coordinate (tr);
\node[rectangle, fit=(bl)(tr)] (subdiag) {};
\begin{scope}
\path[clip] (subdiag.south west) rectangle (subdiag.north east);
\singlemor{subdiag}{$h$}{$X$}{$Y \times Z$}{catterm}{catc};
\path (subdiag.north west) ++(0.5,-0.5) node {$\rhd$};
\end{scope}
\draw[very thick] (subdiag.south west) rectangle (subdiag.north east);
\draw (cbb.south) coordinate[label=below:$X$] -- (subdiag.south)
      (subdiag.north)  -- (cbb.north) coordinate[label=above:$Z$];
\begin{pgfonlayer}{background}
\fill[catterm] (cbb.north west) rectangle (cbb.south);
\fill[catc] (cbb.north east) rectangle (cbb.south);
\end{pgfonlayer}
\end{tikzpicture}
\end{gathered}
\end{equation*}
\end{minipage}
\end{trivlist}
\begin{equation*}
(
\begin{gathered}
\twocelldiag{f}{X}{Y}{catterm}{catc}{1.5}{1.5}{1.5}
\end{gathered}
,
\begin{gathered}
\twocelldiag{g}{X}{Z}{catterm}{catc}{1.5}{1.5}{1.5}
\end{gathered}
)
\mapsto
\begin{gathered}
\begin{tikzpicture}[auto, x=0.5cm, y=0.5cm]
\diagbb{3.5cm}{2cm};
\path (cbb.south west) ++(0.5,1) coordinate (bl)
      (cbb.north east) ++(-0.5,-1) coordinate (tr);
\node[rectangle, fit=(bl)(tr)] (subdiag) {};
\begin{scope}
\path[clip] (subdiag.south west) rectangle (subdiag.north);
\node[rectangle, fit=(subdiag.south west)(subdiag.north)] (left) {};
\singlemor{left}{$f$}{$X$}{$Y$}{catterm}{catc};
\end{scope}
\begin{scope}
\path[clip] (subdiag.south) rectangle (subdiag.north east);
\node[rectangle, fit=(subdiag.south)(subdiag.north east)] (right) {};
\singlemor{right}{$g$}{$X$}{$Z$}{catterm}{catc};
\end{scope}
\draw[very thick] (subdiag.south west) rectangle (subdiag.north east);
\draw[very thick] (subdiag.south) -- (subdiag.north);
\draw (subdiag.north) -- (cbb.north) coordinate[label=above:$Y \times Z$];
\draw (cbb.south) coordinate[label=below:$X$] -- (subdiag.south);
\begin{pgfonlayer}{background}
\fill[catterm] (cbb.north west) rectangle (cbb.south);
\fill[catc] (cbb.north east) rectangle (cbb.south);
\end{pgfonlayer}
\end{tikzpicture}
\end{gathered}
\end{equation*}
We then define the usual notation for the two components of the unit:
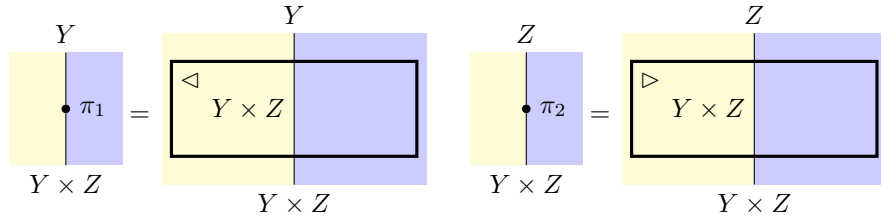
\begin{trivlist}
\item
\begin{minipage}{0.495\textwidth}
\begin{equation*}
\begin{gathered}
\twocelldiag{\pi_1}{Y \times Z}{Y}{catterm}{catc}{1.5}{1.5}{1.5}
\end{gathered}
=
\begin{gathered}
\begin{tikzpicture}[auto, x=0.5cm, y=0.5cm]
\diagbb{3.5cm}{2cm};
\path (cbb.south west) ++(0.5,1) coordinate (bl)
      (cbb.north east) ++(-0.5,-1) coordinate (tr);
\node[rectangle, fit=(bl)(tr)] (subdiag) {};
\begin{scope}
\path[clip] (subdiag.south west) rectangle (subdiag.north east);
\zeromor{subdiag}{$Y \times Z$}{catterm}{catc};
\path (subdiag.north west) ++(0.5,-0.5) node {$\lhd$};
\end{scope}
\draw[very thick] (subdiag.south west) rectangle (subdiag.north east);
\draw (cbb.south) coordinate[label=below:$Y \times Z$] -- (subdiag.south)
      (subdiag.north) -- (cbb.north) coordinate[label=above:$Y$] ;
\begin{pgfonlayer}{background}
\fill[catterm] (cbb.north west) rectangle (cbb.south);
\fill[catc] (cbb.north east) rectangle (cbb.south);
\end{pgfonlayer}
\end{tikzpicture}
\end{gathered}
\end{equation*}
\end{minipage}
\begin{minipage}{0.495\textwidth}
\begin{equation*}
\begin{gathered}
\twocelldiag{\pi_2}{Y \times Z}{Z}{catterm}{catc}{1.5}{1.5}{1.5}
\end{gathered}
=
\begin{gathered}
\begin{tikzpicture}[auto, x=0.5cm, y=0.5cm]
\diagbb{3.5cm}{2cm};
\path (cbb.south west) ++(0.5,1) coordinate (bl)
      (cbb.north east) ++(-0.5,-1) coordinate (tr);
\node[rectangle, fit=(bl)(tr)] (subdiag) {};
\begin{scope}
\path[clip] (subdiag.south west) rectangle (subdiag.north east);
\zeromor{subdiag}{$Y \times Z$}{catterm}{catc};
\path (subdiag.north west) ++(0.5,-0.5) node {$\rhd$};
\end{scope}
\draw[very thick] (subdiag.south west) rectangle (subdiag.north east);
\draw (cbb.south) coordinate[label=below:$Y \times Z$] -- (subdiag.south)
      (subdiag.north) -- (cbb.north) coordinate[label=above:$Z$] ;
\begin{pgfonlayer}{background}
\fill[catterm] (cbb.north west) rectangle (cbb.south);
\fill[catc] (cbb.north east) rectangle (cbb.south);
\end{pgfonlayer}
\end{tikzpicture}
\end{gathered}
\end{equation*}
\end{minipage}
\end{trivlist}
Using the material in section \ref{sec:universality} we can then write the universal property of products as:
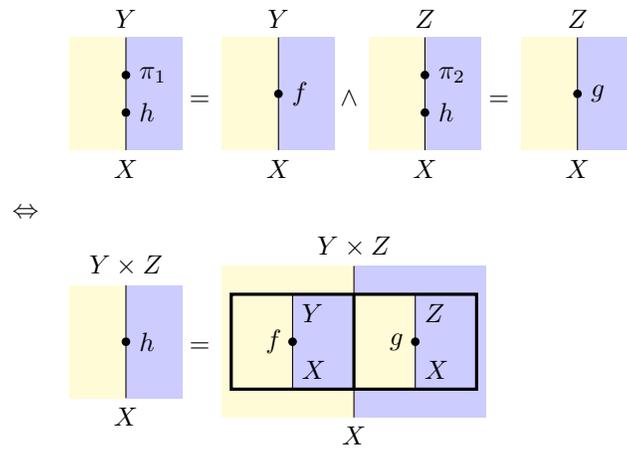
\begin{equation*}
\begin{array}{ll}
 &
\begin{gathered}
\twocellpairdiag{h}{\pi_1}{X}{Y}{catterm}{catc}{1.5}{1.5}{1}
\end{gathered}
=
\begin{gathered}
\twocelldiag{f}{X}{Y}{catterm}{catc}{1.5}{1.5}{1.5}
\end{gathered}
\wedge
\begin{gathered}
\twocellpairdiag{h}{\pi_2}{X}{Z}{catterm}{catc}{1.5}{1.5}{1}
\end{gathered}
=
\begin{gathered}
\twocelldiag{g}{X}{Z}{catterm}{catc}{1.5}{1.5}{1.5}
\end{gathered}
\\
\Leftrightarrow & \\
 &
\begin{gathered}
\twocelldiag{h}{X}{Y \times Z}{catterm}{catc}{1.5}{1.5}{1.5}
\end{gathered}
=
\begin{gathered}
\begin{tikzpicture}[auto, x=0.5cm, y=0.5cm]
\diagbb{3.5cm}{2cm};
\path (cbb.south west) ++(0.5,1) coordinate (bl)
      (cbb.north east) ++(-0.5,-1) coordinate (tr);
\node[rectangle, fit=(bl)(tr)] (subdiag) {};
\begin{scope}
\path[clip] (subdiag.south west) rectangle (subdiag.north);
\node[rectangle, fit=(subdiag.south west)(subdiag.north)] (left) {};
\singlemor{left}{$f$}{$X$}{$Y$}{catterm}{catc};
\end{scope}
\begin{scope}
\path[clip] (subdiag.south) rectangle (subdiag.north east);
\node[rectangle, fit=(subdiag.south)(subdiag.north east)] (right) {};
\singlemor{right}{$g$}{$X$}{$Z$}{catterm}{catc};
\end{scope}
\draw[very thick] (subdiag.south west) rectangle (subdiag.north east);
\draw[very thick] (subdiag.south) -- (subdiag.north);
\draw (subdiag.north) -- (cbb.north) coordinate[label=above:$Y \times Z$];
\draw (cbb.south) coordinate[label=below:$X$] -- (subdiag.south);
\begin{pgfonlayer}{background}
\fill[catterm] (cbb.north west) rectangle (cbb.south);
\fill[catc] (cbb.north east) rectangle (cbb.south);
\end{pgfonlayer}
\end{tikzpicture}
\end{gathered}
\end{array}
\end{equation*}
The ``push / pop'' equations \ref{eq:thetaone} and \ref{eq:thetatwo} then lead to the following three equalities:
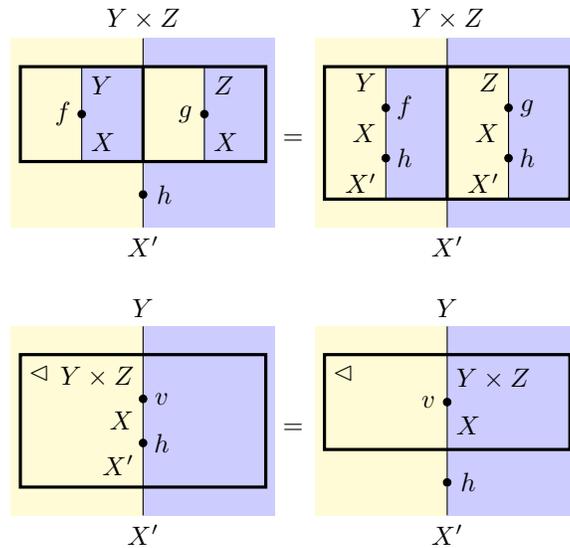
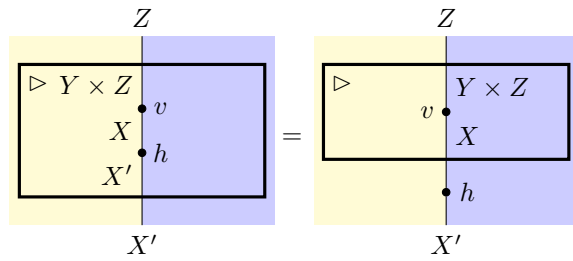
\begin{equation*}
\begin{gathered}
\begin{tikzpicture}[auto, x=0.5cm, y=0.5cm]
\diagbb{3.5cm}{2.5cm};
\path (cbb.south west) ++(0.5,2) coordinate (bl)
      (cbb.north east) ++(-0.5,-1) coordinate (tr);
\node[rectangle, fit=(bl)(tr)] (subdiag) {};
\begin{scope}
\path[clip] (subdiag.south west) rectangle (subdiag.north);
\node[rectangle, fit=(subdiag.south west)(subdiag.north)] (left) {};
\singlemor{left}{$f$}{$X$}{$Y$}{catterm}{catc};
\end{scope}
\begin{scope}
\path[clip] (subdiag.south) rectangle (subdiag.north east);
\node[rectangle, fit=(subdiag.south)(subdiag.north east)] (right) {};
\singlemor{right}{$g$}{$X$}{$Z$}{catterm}{catc};
\end{scope}
\draw[very thick] (subdiag.south west) rectangle (subdiag.north east);
\draw[very thick] (subdiag.south) -- (subdiag.north);
\draw (subdiag.north) -- (cbb.north) coordinate[label=above:$Y \times Z$];
\coordinate[dot, label=right:$h$] (h) at ($(cbb.south)!0.5!(subdiag.south)$);
\draw (cbb.south) coordinate[label=below:$X'$] -- (h) -- (subdiag.south);
\begin{pgfonlayer}{background}
\fill[catterm] (cbb.north west) rectangle (cbb.south);
\fill[catc] (cbb.north east) rectangle (cbb.south);
\end{pgfonlayer}
\end{tikzpicture}
\end{gathered}
=
\begin{gathered}
\begin{tikzpicture}[auto, x=0.5cm, y=0.5cm]
\diagbb{3.5cm}{2.5cm};
\path (cbb.south west) ++(0.5,1) coordinate (bl)
      (cbb.north east) ++(-0.5,-1) coordinate (tr);
\node[rectangle, fit=(bl)(tr)] (subdiag) {};
\begin{scope}
\path[clip] (subdiag.south west) rectangle (subdiag.north);
\node[rectangle, fit=(subdiag.south west)(subdiag.north)] (left) {};
\doublemor{left}{$h$}{$f$}{$X'$}{$X$}{$Y$}{catterm}{catc};
\end{scope}
\begin{scope}
\path[clip] (subdiag.south) rectangle (subdiag.north east);
\node[rectangle, fit=(subdiag.south)(subdiag.north east)] (right) {};
\doublemor{right}{$h$}{$g$}{$X'$}{$X$}{$Z$}{catterm}{catc};
\end{scope}
\draw[very thick] (subdiag.south west) rectangle (subdiag.north east);
\draw[very thick] (subdiag.south) -- (subdiag.north);
\draw (cbb.south) coordinate[label=below:$X'$] -- (subdiag.south);
\coordinate (h) at ($(subdiag.north)!0.5!(cbb.north)$);
\draw (subdiag.north) -- (cbb.north) coordinate[label=above:$Y \times Z$];
\begin{pgfonlayer}{background}
\fill[catterm] (cbb.north west) rectangle (cbb.south);
\fill[catc] (cbb.north east) rectangle (cbb.south);
\end{pgfonlayer}
\end{tikzpicture}
\end{gathered}
\end{equation*}
\begin{equation*}
\begin{gathered}
\begin{tikzpicture}[auto, x=0.5cm, y=0.5cm]
\diagbb{3.5cm}{2.5cm};
\path (cbb.south west) ++(0.5,1) coordinate (bl)
      (cbb.north east) ++(-0.5,-1) coordinate (tr);
\node[rectangle, fit=(bl)(tr)] (subdiag) {};
\begin{scope}
\path[clip] (subdiag.south west) rectangle (subdiag.north east);
\doublemor{subdiag}{$h$}{$v$}{$X'$}{$X$}{$Y \times Z$}{catterm}{catc};
\path (subdiag.north west) ++(0.5,-0.5) node {$\lhd$};
\end{scope}
\draw[very thick] (subdiag.south west) rectangle (subdiag.north east);
\draw (cbb.south) coordinate[label=below:$X'$] -- (subdiag.south)
      (subdiag.north) -- (cbb.north) coordinate[label=above:$Y$];
\begin{pgfonlayer}{background}
\fill[catterm] (cbb.north west) rectangle (cbb.south);
\fill[catc] (cbb.north east) rectangle (cbb.south);
\end{pgfonlayer}
\end{tikzpicture}
\end{gathered}
=
\begin{gathered}
\begin{tikzpicture}[auto, x=0.5cm, y=0.5cm]
\diagbb{3.5cm}{2.5cm};
\path (cbb.south west) ++(0.5,2) coordinate (bl)
      (cbb.north east) ++(-0.5,-1) coordinate (tr);
\node[rectangle, fit=(bl)(tr)] (subdiag) {};
\begin{scope}
\path[clip] (subdiag.south west) rectangle (subdiag.north east);
\singlemor{subdiag}{$v$}{$X$}{$Y \times Z$}{catterm}{catc};
\path (subdiag.north west) ++(0.5,-0.5) node {$\lhd$};
\end{scope}
\draw[very thick] (subdiag.south west) rectangle (subdiag.north east);
\coordinate[dot, label=right:$h$] (h) at ($(cbb.south)!0.5!(subdiag.south)$);
\draw (cbb.south) coordinate[label=below:$X'$] -- (h) -- (subdiag.south)
      (subdiag.north) -- (cbb.north) coordinate[label=above:$Y$];
\begin{pgfonlayer}{background}
\fill[catterm] (cbb.north west) rectangle (cbb.south);
\fill[catc] (cbb.north east) rectangle (cbb.south);
\end{pgfonlayer}
\end{tikzpicture}
\end{gathered}
\end{equation*}
\begin{equation*}
\begin{gathered}
\begin{tikzpicture}[auto, x=0.5cm, y=0.5cm]
\diagbb{3.5cm}{2.5cm};
\path (cbb.south west) ++(0.5,1) coordinate (bl)
      (cbb.north east) ++(-0.5,-1) coordinate (tr);
\node[rectangle, fit=(bl)(tr)] (subdiag) {};
\begin{scope}
\path[clip] (subdiag.south west) rectangle (subdiag.north east);
\doublemor{subdiag}{$h$}{$v$}{$X'$}{$X$}{$Y \times Z$}{catterm}{catc};
\path (subdiag.north west) ++(0.5,-0.5) node {$\rhd$};
\end{scope}
\draw[very thick] (subdiag.south west) rectangle (subdiag.north east);
\draw (cbb.south) coordinate[label=below:$X'$] -- (subdiag.south)
      (subdiag.north) -- (cbb.north) coordinate[label=above:$Z$];
\begin{pgfonlayer}{background}
\fill[catterm] (cbb.north west) rectangle (cbb.south);
\fill[catc] (cbb.north east) rectangle (cbb.south);
\end{pgfonlayer}
\end{tikzpicture}
\end{gathered}
=
\begin{gathered}
\begin{tikzpicture}[auto, x=0.5cm, y=0.5cm]
\diagbb{3.5cm}{2.5cm};
\path (cbb.south west) ++(0.5,2) coordinate (bl)
      (cbb.north east) ++(-0.5,-1) coordinate (tr);
\node[rectangle, fit=(bl)(tr)] (subdiag) {};
\begin{scope}
\path[clip] (subdiag.south west) rectangle (subdiag.north east);
\singlemor{subdiag}{$v$}{$X$}{$Y \times Z$}{catterm}{catc};
\path (subdiag.north west) ++(0.5,-0.5) node {$\rhd$};
\end{scope}
\draw[very thick] (subdiag.south west) rectangle (subdiag.north east);
\coordinate[dot, label=right:$h$] (h) at ($(cbb.south)!0.5!(subdiag.south)$);
\draw (cbb.south) coordinate[label=below:$X'$] -- (h) -- (subdiag.south)
      (subdiag.north) -- (cbb.north) coordinate[label=above:$Z$];      
\begin{pgfonlayer}{background}
\fill[catterm] (cbb.north west) rectangle (cbb.south);
\fill[catc] (cbb.north east) rectangle (cbb.south);
\end{pgfonlayer}
\end{tikzpicture}
\end{gathered}
\end{equation*}
That these maps witness a bijection leads to the following three of equalities:
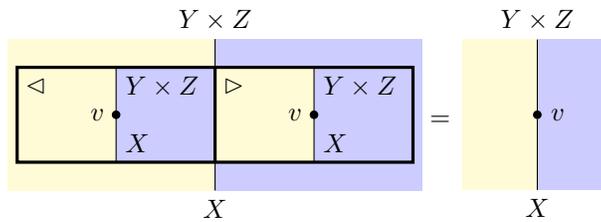
\begin{equation*}
\begin{gathered}
\begin{tikzpicture}[auto, x=0.5cm, y=0.5cm]
\diagbb{5.5cm}{2cm};
\path (cbb.south west) ++(0.5,1) coordinate (bl)
      (cbb.north east) ++(-0.5,-1) coordinate (tr);
\node[rectangle, fit=(bl)(tr)] (subdiag) {};
\begin{scope}
\path[clip] (subdiag.south west) rectangle (subdiag.north);
\node[rectangle, fit=(subdiag.south west)(subdiag.north)] (left) {};
\singlemor{left}{$v$}{$X$}{$Y \times Z$}{catterm}{catc};
\end{scope}
\begin{scope}
\path[clip] (subdiag.south) rectangle (subdiag.north east);
\node[rectangle, fit=(subdiag.south)(subdiag.north east)] (right) {};
\singlemor{right}{$v$}{$X$}{$Y \times Z$}{catterm}{catc};
\end{scope}
\draw[very thick] (subdiag.south west) rectangle (subdiag.north east);
\draw[very thick] (subdiag.south) -- (subdiag.north);
\draw (cbb.south) coordinate[label=below:$X$] -- (subdiag.south);
\draw (subdiag.north) -- (cbb.north) coordinate[label=above:$Y \times Z$];
\path (left.north west) ++(0.75,-0.75) node {$\lhd$}
      (right.north west) ++(0.75,-0.75) node {$\rhd$};
\begin{pgfonlayer}{background}
\fill[catterm] (cbb.north west) rectangle (cbb.south);
\fill[catc] (cbb.north east) rectangle (cbb.south);
\end{pgfonlayer}
\end{tikzpicture}
\end{gathered}
=
\begin{gathered}
\twocelldiag{v}{X}{Y \times Z}{catterm}{catc}{2}{2}{2}
\end{gathered}
\end{equation*}
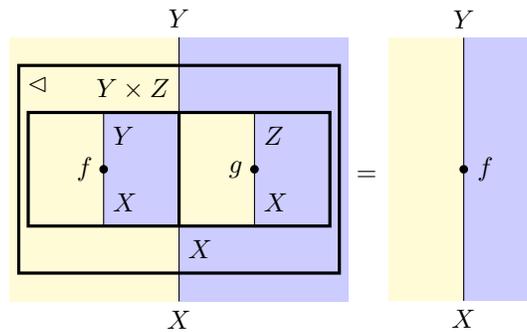
\begin{equation*}
\begin{gathered}
\begin{tikzpicture}[auto, x=0.5cm, y=0.5cm]
\diagbb{4.5cm}{3.5cm};
\path (cbb.south west) ++(0.5,1) coordinate (bl')
      (cbb.north east) ++(-0.5,-1) coordinate (tr');
\node[rectangle, fit=(bl')(tr')] (subdiag') {};
\path (subdiag'.south west) ++(0.5,1.5) coordinate (bl)
      (subdiag'.north east) ++(-0.5,-1.5) coordinate (tr);
\node[rectangle, fit=(bl)(tr)] (subdiag) {};
\fill[catc] (subdiag'.south west) rectangle (subdiag'.north east);
\fill[catterm] (subdiag'.south west) rectangle (subdiag'.north);
\begin{scope}
\path[clip] (subdiag.south west) rectangle (subdiag.north);
\node[rectangle, fit=(subdiag.south west)(subdiag.north)] (left) {};
\singlemor{left}{$f$}{$X$}{$Y$}{catterm}{catc};
\end{scope}
\begin{scope}
\path[clip] (subdiag.south) rectangle (subdiag.north east);
\node[rectangle, fit=(subdiag.south)(subdiag.north east)] (right) {};
\singlemor{right}{$g$}{$X$}{$Z$}{catterm}{catc};
\end{scope}
\draw[very thick] (subdiag.south west) rectangle (subdiag.north east);
\draw[very thick] (subdiag.south) -- (subdiag.north);
\draw (subdiag'.south) to node[swap]{$X$} (subdiag.south);
\draw (subdiag.north) to node{$Y \times Z$} (subdiag'.north);
\draw[very thick] (subdiag'.south west) rectangle (subdiag'.north east);
\draw (cbb.south) coordinate[label=below:$X$] -- (subdiag'.south)
      (subdiag'.north) -- (cbb.north) coordinate[label=above:$Y$];
\path (subdiag'.north west) ++(0.5,-0.5) node{$\lhd$};
\begin{pgfonlayer}{background}
\fill[catterm] (cbb.north west) rectangle (cbb.south);
\fill[catc] (cbb.north east) rectangle (cbb.south);
\end{pgfonlayer}
\end{tikzpicture}
\end{gathered}
=
\begin{gathered}
\twocelldiag{f}{X}{Y}{catterm}{catc}{2}{2}{3.5}
\end{gathered}
\end{equation*}
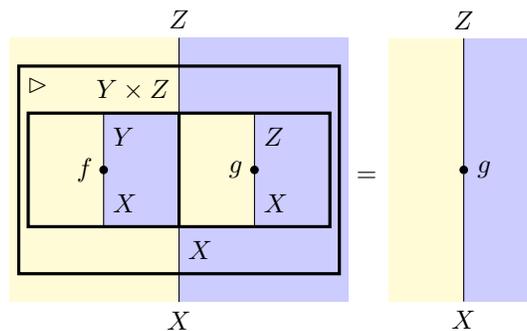
\begin{equation*}
\begin{gathered}
\begin{tikzpicture}[auto, x=0.5cm, y=0.5cm]
\diagbb{4.5cm}{3.5cm};
\path (cbb.south west) ++(0.5,1) coordinate (bl')
      (cbb.north east) ++(-0.5,-1) coordinate (tr');
\node[rectangle, fit=(bl')(tr')] (subdiag') {};
\path (subdiag'.south west) ++(0.5,1.5) coordinate (bl)
      (subdiag'.north east) ++(-0.5,-1.5) coordinate (tr);
\node[rectangle, fit=(bl)(tr)] (subdiag) {};
\fill[catc] (subdiag'.south west) rectangle (subdiag'.north east);
\fill[catterm] (subdiag'.south west) rectangle (subdiag'.north);
\begin{scope}
\path[clip] (subdiag.south west) rectangle (subdiag.north);
\node[rectangle, fit=(subdiag.south west)(subdiag.north)] (left) {};
\singlemor{left}{$f$}{$X$}{$Y$}{catterm}{catc};
\end{scope}
\begin{scope}
\path[clip] (subdiag.south) rectangle (subdiag.north east);
\node[rectangle, fit=(subdiag.south)(subdiag.north east)] (right) {};
\singlemor{right}{$g$}{$X$}{$Z$}{catterm}{catc};
\end{scope}
\draw[very thick] (subdiag.south west) rectangle (subdiag.north east);
\draw[very thick] (subdiag.south) -- (subdiag.north);
\draw (subdiag'.south) to node[swap]{$X$} (subdiag.south);
\draw (subdiag.north) to node{$Y \times Z$} (subdiag'.north);
\draw[very thick] (subdiag'.south west) rectangle (subdiag'.north east);
\draw (cbb.south) coordinate[label=below:$X$] -- (subdiag'.south)
      (subdiag'.north) -- (cbb.north) coordinate[label=above:$Z$];
\path (subdiag'.north west) ++(0.5,-0.5) node{$\rhd$};
\begin{pgfonlayer}{background}
\fill[catterm] (cbb.north west) rectangle (cbb.south);
\fill[catc] (cbb.north east) rectangle (cbb.south);
\end{pgfonlayer}
\end{tikzpicture}
\end{gathered}
=
\begin{gathered}
\twocelldiag{g}{X}{Z}{catterm}{catc}{2}{2}{3.5}
\end{gathered}
\end{equation*}
Readers familiar with more standard introductions to category theory will hopefully recognise
many standard exercises in the properties of binary products are given in graphical form
by the various equations above.

\subsubsection*{Terminal Objects}
The terminal object will be denoted $1$,
and $! : X \rightarrow 1$ will denote the unique morphism from an object $X$ to
the terminal object. The universal property of the terminal object can then
be written, for each object $X$ in $\mathcal{C}$:
\begin{equation*}
\begin{gathered}
\twocelldiag{f}{X}{1}{catterm}{catc}{2}{2}{1}
\end{gathered}
\Leftrightarrow
\begin{gathered}
\twocelldiag{f}{X}{1}{catterm}{catc}{2}{2}{1}
\end{gathered}
=
\begin{gathered}
\twocelldiag{!}{X}{1}{catterm}{catc}{2}{2}{1}
\end{gathered}
\end{equation*}

\begin{example}[Terminal Coalgebras and the Fusion Law]
\label{ex:terminalcoalg}
Coalgebras are the dual notion to algebras, a standard reference is \citep{Rutten2000}. 
For an endofunctor $T : \mathcal{C} \rightarrow \mathcal{C}$
a coalgebra is a pair consisting of an object $X$ and a morphism $X \rightarrow T(X)$.
A coalgebra morphism $(X, a) \rightarrow (Y, b)$ is a $\mathcal{C}$ morphism $h : X \rightarrow Y$ 
satisfying:
\begin{equation*}
\begin{gathered}
\coalgmordom{a}{h}{X}{Y}{T}{catc}{1}{1}
\end{gathered}
=
\begin{gathered}
\coalgmorcodom{b}{h}{X}{Y}{T}{catc}{1}{1}
\end{gathered}
\end{equation*}
Dually to the situation with algebras in example \ref{ex:initalg}, 
if an endofunctor $T$ has a terminal coalgebra $(\termcoalgbase{T}, \termcoalg{T})$, 
we can rewrite the terminal object condition for coalgebras in a more convenient form.
We will adopt the standard notation of \unfold{T}{a} for the unique morphism from a
coalgebra $a$ to the terminal coalgebra.
\begin{equation*}
\begin{gathered}
\coalgmorcodom{\termcoalg{T}}{h}{X}{\termcoalgbase{T}}{T}{catc}{2}{1}
\end{gathered}
=
\begin{gathered}
\coalgmordom{a}{h}{X}{\termcoalgbase{T}}{T}{catc}{1}{1}
\end{gathered}
\Leftrightarrow 
\begin{gathered}
\twocelldiag{h}{a}{\termcoalg{T}}{catterm}{catd}{1}{1}{1.5}
\end{gathered}
=
\begin{gathered}
\twocelldiag{\unfold{T}a}{a}{\termcoalg{T}}{catterm}{catd}{1}{3.5}{1.5}
\end{gathered}
\end{equation*}
We will now prove the important (strong) fusion law for terminal coalgebras. We have
the following chain of equivalences:
\begin{eqproof*}
\begin{gathered}
\begin{tikzpicture}[scale=0.5]
\path coordinate[dot, label=left:$h$] (h) +(0,-1) coordinate[label=below:$g$] (b)
 ++(0,1) coordinate[dot, label=left:$\unfold{T}f$] (unfold)
 ++(0,1) coordinate[label=above:\termcoalg{T}] (t);
\draw (b) -- (h) -- (unfold) -- (t);
\begin{pgfonlayer}{background}
\fill[catterm] ($(t) + (-4,0)$) rectangle (b);
\fill[catd] (t) rectangle ($(b) + (1,0)$);
\end{pgfonlayer}
\end{tikzpicture}
\end{gathered}
=
\begin{gathered}
\twocelldiag{\unfold{T}g}{g}{\termcoalg{T}}{catterm}{catd}{1}{4}{1.5}
\end{gathered}
\explain[\Leftrightarrow]{universal property of terminal coalgebra}
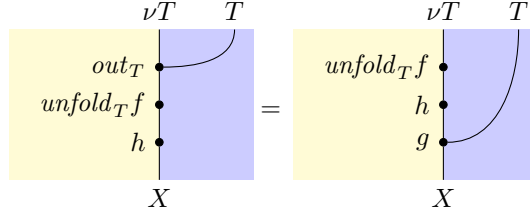
\begin{gathered}
\begin{tikzpicture}[scale=0.5]
\path coordinate[dot, label=left:$h$] (h)
 +(0,-1) coordinate[label=below:$X$] (b)
 ++(0,1) coordinate[dot, label=left:$\unfold{T}f$] (unfold)
 ++(0,1) coordinate[dot, label=left:\termcoalg{T}] (term)
 ++(0,1) coordinate[label=above:\termcoalgbase{T}] (tl)
 ++(2,0) coordinate[label=above:$T$] (tr);
\draw (b) -- (h) -- (unfold) -- (term) -- (tl)
 (term) to[out=0, in=-90] (tr);
\begin{pgfonlayer}{background}
\fill[catterm] ($(b) + (-4,0)$) rectangle (tl);
\fill[catc] (b) rectangle ($(tr) + (0.5,0)$);
\end{pgfonlayer}
\end{tikzpicture}
\end{gathered}
=
\begin{gathered}
\begin{tikzpicture}[scale=0.5]
\path coordinate[dot, label=left:$g$] (g)
 +(0,-1) coordinate[label=below:$X$] (b)
 ++(0,1) coordinate[dot, label=left:$h$] (h)
 ++(0,1) coordinate[dot, label=left:$\unfold{T}f$] (unfold)
 ++(0,1) coordinate[label=above:\termcoalgbase{T}] (tl)
 ++(2,0) coordinate[label=above:$T$] (tr);
\draw (b) -- (g) -- (h) -- (unfold) -- (tl)
 (g) to[out=0, in=-90] (tr);
\begin{pgfonlayer}{background}
\fill[catterm] ($(b) + (-4,0)$) rectangle (tl);
\fill[catc] (b) rectangle ($(tr) + (0.5,0)$);
\end{pgfonlayer}
\end{tikzpicture}
\end{gathered}
\explain[\Leftrightarrow]{\unfold{T}{f} is a coalgebra homomorphism}
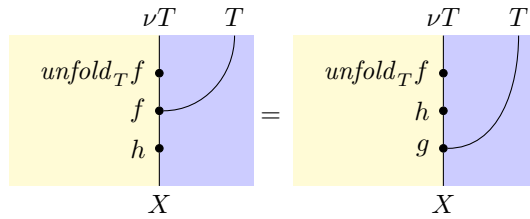
\begin{gathered}
\begin{tikzpicture}[scale=0.5]
\path coordinate[dot, label=left:$h$] (h)
 +(0,-1) coordinate[label=below:$X$] (b)
 ++(0,1) coordinate[dot, label=left:$f$] (f)
 ++(0,1) coordinate[dot, label=left:$\unfold{T}f$] (unfold)
 ++(0,1) coordinate[label=above:\termcoalgbase{T}] (tl)
 ++(2,0) coordinate[label=above:$T$] (tr);
\draw (b) -- (h) -- (f) -- (unfold) -- (tl)
 (f) to[out=0, in=-90] (tr);
\begin{pgfonlayer}{background}
\fill[catterm] ($(b) + (-4,0)$) rectangle (tl);
\fill[catc] (b) rectangle ($(tr) + (0.5,0)$);
\end{pgfonlayer}
\end{tikzpicture}
\end{gathered}
=
\begin{gathered}
\begin{tikzpicture}[scale=0.5]
\path coordinate[dot, label=left:$g$] (g)
 +(0,-1) coordinate[label=below:$X$] (b)
 ++(0,1) coordinate[dot, label=left:$h$] (h)
 ++(0,1) coordinate[dot, label=left:$\unfold{T}f$] (unfold)
 ++(0,1) coordinate[label=above:\termcoalgbase{T}] (tl)
 ++(2,0) coordinate[label=above:$T$] (tr);
\draw (b) -- (g) -- (h) -- (unfold) -- (tl)
 (g) to[out=0, in=-90] (tr);
\begin{pgfonlayer}{background}
\fill[catterm] ($(b) + (-4,0)$) rectangle (tl);
\fill[catc] (b) rectangle ($(tr) + (0.5,0)$);
\end{pgfonlayer}
\end{tikzpicture}
\end{gathered}
\end{eqproof*}
The intuitive explanation of the fusion law is that if the final equality is satisfied, then we can ``fuse''
the morphism $h$ with $\unfold{T}{f}$ to give a single unfold morphism $\unfold{T}{g}$. 
As the precondition for fusion is rather cumbersome, it is often easier to apply
the weak fusion law which follows as an immediate corollary via Leibniz:
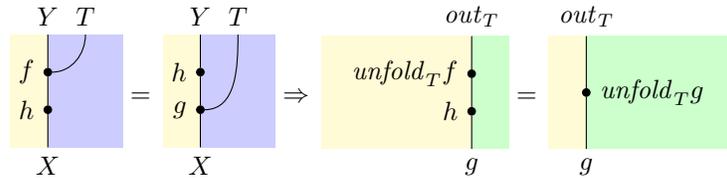
\begin{equation*}
\begin{gathered}
\coalgmorcodom{f}{h}{X}{Y}{T}{catc}{1}{1}
\end{gathered}
=
\begin{gathered}
\coalgmordom{g}{h}{X}{Y}{T}{catc}{1}{1}
\end{gathered}
\Rightarrow 
\begin{gathered}
\begin{tikzpicture}[scale=0.5]
\path coordinate[dot, label=left:$h$] (h) +(0,-1) coordinate[label=below:$g$] (b)
 ++(0,1) coordinate[dot, label=left:$\unfold{T}f$] (unfold)
 ++(0,1) coordinate[label=above:\termcoalg{T}] (t);
\draw (b) -- (h) -- (unfold) -- (t);
\begin{pgfonlayer}{background}
\fill[catterm] ($(t) + (-4,0)$) rectangle (b);
\fill[catd] (t) rectangle ($(b) + (1,0)$);
\end{pgfonlayer}
\end{tikzpicture}
\end{gathered}
=
\begin{gathered}
\twocelldiag{\unfold{T}g}{g}{\termcoalg{T}}{catterm}{catd}{1}{4}{1.5}
\end{gathered}
\end{equation*}
\end{example}

\subsubsection*{Binary Coproducts}
\label{sec:bincoprod}
Binary coproducts follow the dual pattern to the description of binary products in
section \ref{sec:binprod}. We will write the binary coproduct of $X$ and $Y$ as
$X + Y$ in the usual way.
The representability condition gives a bijective correspondence between morphisms $X  + Y \rightarrow Z$ and 
pairs of morphisms $X \rightarrow Z$ and $Y \rightarrow Z$.
As with binary products, to aid calculations we will introduce three different maps:
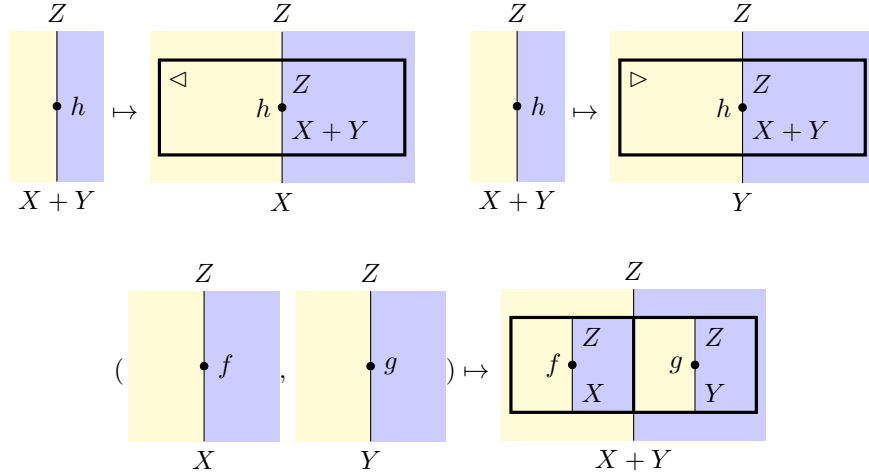
\begin{trivlist}
\item
\begin{minipage}{0.495\textwidth}
\begin{equation*}
\begin{gathered}
\twocelldiag{h}{X + Y}{Z}{catterm}{catc}{1.25}{1.25}{2}
\end{gathered}
\mapsto
\begin{gathered}
\begin{tikzpicture}[auto, x=0.5cm, y=0.5cm]
\diagbb{3.5cm}{2cm};
\path (cbb.south west) ++(0.5,1) coordinate (bl)
      (cbb.north east) ++(-0.5,-1) coordinate (tr);
\node[rectangle, fit=(bl)(tr)] (subdiag) {};
\begin{scope}
\path[clip] (subdiag.south west) rectangle (subdiag.north east);
\singlemor{subdiag}{$h$}{$X + Y$}{$Z$}{catterm}{catc};
\path (subdiag.north west) ++(0.5,-0.5) node {$\lhd$};
\end{scope}
\draw[very thick] (subdiag.south west) rectangle (subdiag.north east);
\draw (cbb.south) coordinate[label=below:$X$] -- (subdiag.south)
      (subdiag.north) -- (cbb.north) coordinate[label=above:$Z$];
\begin{pgfonlayer}{background}
\fill[catterm] (cbb.north west) rectangle (cbb.south);
\fill[catc] (cbb.north east) rectangle (cbb.south);
\end{pgfonlayer}
\end{tikzpicture}
\end{gathered}
\end{equation*}
\end{minipage}
\begin{minipage}{0.495\textwidth}
\begin{equation*}
\begin{gathered}
\twocelldiag{h}{X + Y}{Z}{catterm}{catc}{1.25}{1.25}{2}
\end{gathered}
\mapsto
\begin{gathered}
\begin{tikzpicture}[auto, x=0.5cm, y=0.5cm]
\diagbb{3.5cm}{2cm};
\path (cbb.south west) ++(0.5,1) coordinate (bl)
      (cbb.north east) ++(-0.5,-1) coordinate (tr);
\node[rectangle, fit=(bl)(tr)] (subdiag) {};
\begin{scope}
\path[clip] (subdiag.south west) rectangle (subdiag.north east);
\singlemor{subdiag}{$h$}{$X + Y$}{$Z$}{catterm}{catc};
\path (subdiag.north west) ++(0.5,-0.5) node {$\rhd$};
\end{scope}
\draw[very thick] (subdiag.south west) rectangle (subdiag.north east);
\coordinate (h) at ($(cbb.south)!0.5!(subdiag.south)$);
\draw (cbb.south) coordinate[label=below:$Y$] -- (subdiag.south)
      (subdiag.north) -- (cbb.north) coordinate[label=above:$Z$];      
\begin{pgfonlayer}{background}
\fill[catterm] (cbb.north west) rectangle (cbb.south);
\fill[catc] (cbb.north east) rectangle (cbb.south);
\end{pgfonlayer}
\end{tikzpicture}
\end{gathered}
\end{equation*}
\end{minipage}
\end{trivlist}
\begin{equation*}
(
\begin{gathered}
\twocelldiag{f}{X}{Z}{catterm}{catc}{2}{2}{2}
\end{gathered}
,
\begin{gathered}
\twocelldiag{g}{Y}{Z}{catterm}{catc}{2}{2}{2}
\end{gathered}
)
\mapsto
\begin{gathered}
\begin{tikzpicture}[auto, x=0.5cm, y=0.5cm]
\diagbb{3.5cm}{2cm};
\path (cbb.south west) ++(0.5,1) coordinate (bl)
      (cbb.north east) ++(-0.5,-1) coordinate (tr);
\node[rectangle, fit=(bl)(tr)] (subdiag) {};
\begin{scope}
\path[clip] (subdiag.south west) rectangle (subdiag.north);
\node[rectangle, fit=(subdiag.south west)(subdiag.north)] (left) {};
\singlemor{left}{$f$}{$X$}{$Z$}{catterm}{catc};
\end{scope}
\begin{scope}
\path[clip] (subdiag.south) rectangle (subdiag.north east);
\node[rectangle, fit=(subdiag.south)(subdiag.north east)] (right) {};
\singlemor{right}{$g$}{$Y$}{$Z$}{catterm}{catc};
\end{scope}
\draw[very thick] (subdiag.south west) rectangle (subdiag.north east);
\draw[very thick] (subdiag.south) -- (subdiag.north);
\draw (subdiag.north) -- (cbb.north) coordinate[label=above:$Z$];
\draw (cbb.south) coordinate[label=below:$X + Y$] -- (subdiag.south);
\begin{pgfonlayer}{background}
\fill[catterm] (cbb.north west) rectangle (cbb.south);
\fill[catc] (cbb.north east) rectangle (cbb.south);
\end{pgfonlayer}
\end{tikzpicture}
\end{gathered}
\end{equation*}
We then define the usual notation for the two components of the counit:
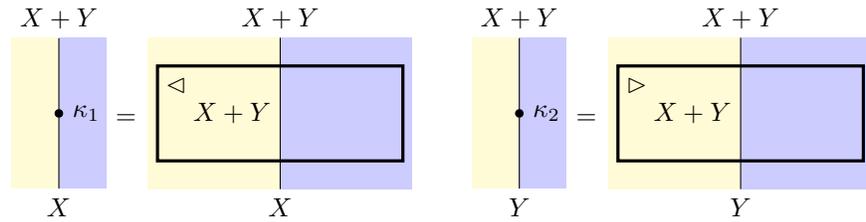
\begin{trivlist}
\item
\begin{minipage}{0.495\textwidth}
\begin{equation*}
\begin{gathered}
\twocelldiag{\kappa_1}{X}{X + Y}{catterm}{catc}{1.25}{1.25}{2}
\end{gathered}
=
\begin{gathered}
\begin{tikzpicture}[auto, x=0.5cm, y=0.5cm]
\diagbb{3.5cm}{2cm};
\path (cbb.south west) ++(0.5,1) coordinate (bl)
      (cbb.north east) ++(-0.5,-1) coordinate (tr);
\node[rectangle, fit=(bl)(tr)] (subdiag) {};
\begin{scope}
\path[clip] (subdiag.south west) rectangle (subdiag.north east);
\zeromor{subdiag}{$X + Y$}{catterm}{catc};
\path (subdiag.north west) ++(0.5,-0.5) node {$\lhd$};
\end{scope}
\draw[very thick] (subdiag.south west) rectangle (subdiag.north east);
\draw (cbb.south) coordinate[label=below:$X$] -- (subdiag.south)
      (subdiag.north) -- (cbb.north) coordinate[label=above:$X + Y$];
\begin{pgfonlayer}{background}
\fill[catterm] (cbb.north west) rectangle (cbb.south);
\fill[catc] (cbb.north east) rectangle (cbb.south);
\end{pgfonlayer}
\end{tikzpicture}
\end{gathered}
\end{equation*}
\end{minipage}
\begin{minipage}{0.495\textwidth}
\begin{equation*}
\begin{gathered}
\twocelldiag{\kappa_2}{Y}{X + Y}{catterm}{catc}{1.25}{1.25}{2}
\end{gathered}
=
\begin{gathered}
\begin{tikzpicture}[auto, x=0.5cm, y=0.5cm]
\diagbb{3.5cm}{2cm};
\path (cbb.south west) ++(0.5,1) coordinate (bl)
      (cbb.north east) ++(-0.5,-1) coordinate (tr);
\node[rectangle, fit=(bl)(tr)] (subdiag) {};
\begin{scope}
\path[clip] (subdiag.south west) rectangle (subdiag.north east);
\zeromor{subdiag}{$X + Y$}{catterm}{catc};
\path (subdiag.north west) ++(0.5,-0.5) node {$\rhd$};
\end{scope}
\draw[very thick] (subdiag.south west) rectangle (subdiag.north east);
\draw (cbb.south) coordinate[label=below:$Y$] -- (subdiag.south)
      (subdiag.north) -- (cbb.north) coordinate[label=above:$X + Y$];
\begin{pgfonlayer}{background}
\fill[catterm] (cbb.north west) rectangle (cbb.south);
\fill[catc] (cbb.north east) rectangle (cbb.south);
\end{pgfonlayer}
\end{tikzpicture}
\end{gathered}
\end{equation*}
\end{minipage}
\end{trivlist}
Using the material in section \ref{sec:universality} we can then write the universal property of coproducts as:
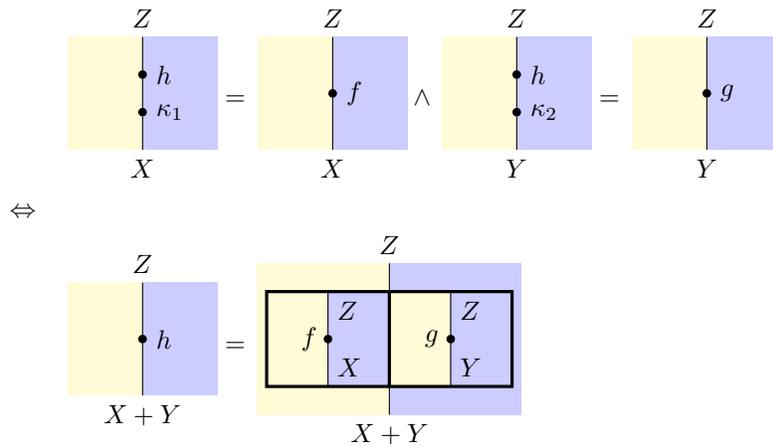
\begin{equation*}
\begin{array}{ll}
 &
\begin{gathered}
\twocellpairdiag{\kappa_1}{h}{X}{Z}{catterm}{catc}{2}{2}{1}
\end{gathered}
=
\begin{gathered}
\twocelldiag{f}{X}{Z}{catterm}{catc}{2}{2}{1.5}
\end{gathered}
\wedge
\begin{gathered}
\twocellpairdiag{\kappa_2}{h}{Y}{Z}{catterm}{catc}{2}{2}{1}
\end{gathered}
=
\begin{gathered}
\twocelldiag{g}{Y}{Z}{catterm}{catc}{2}{2}{1.5}
\end{gathered}
\\
\Leftrightarrow & \\
 &
\begin{gathered}
\twocelldiag{h}{X + Y}{Z}{catterm}{catc}{2}{2}{1.5}
\end{gathered}
=
\begin{gathered}
\begin{tikzpicture}[auto, x=0.5cm, y=0.5cm]
\diagbb{3.5cm}{2cm};
\path (cbb.south west) ++(0.5,1) coordinate (bl)
      (cbb.north east) ++(-0.5,-1) coordinate (tr);
\node[rectangle, fit=(bl)(tr)] (subdiag) {};
\begin{scope}
\path[clip] (subdiag.south west) rectangle (subdiag.north);
\node[rectangle, fit=(subdiag.south west)(subdiag.north)] (left) {};
\singlemor{left}{$f$}{$X$}{$Z$}{catterm}{catc};
\end{scope}
\begin{scope}
\path[clip] (subdiag.south) rectangle (subdiag.north east);
\node[rectangle, fit=(subdiag.south)(subdiag.north east)] (right) {};
\singlemor{right}{$g$}{$Y$}{$Z$}{catterm}{catc};
\end{scope}
\draw[very thick] (subdiag.south west) rectangle (subdiag.north east);
\draw[very thick] (subdiag.south) -- (subdiag.north);
\draw (subdiag.north) -- (cbb.north) coordinate[label=above:$Z$];
\draw (cbb.south) coordinate[label=below:$X + Y$] -- (subdiag.south) ;
\begin{pgfonlayer}{background}
\fill[catterm] (cbb.north west) rectangle (cbb.south);
\fill[catc] (cbb.north east) rectangle (cbb.south);
\end{pgfonlayer}
\end{tikzpicture}
\end{gathered}
\end{array}
\end{equation*}
The ``push / pop'' equations \ref{eq:thetaone} and \ref{eq:thetatwo} lead to the following three equalities:
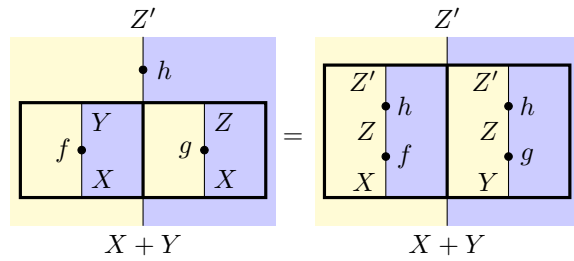
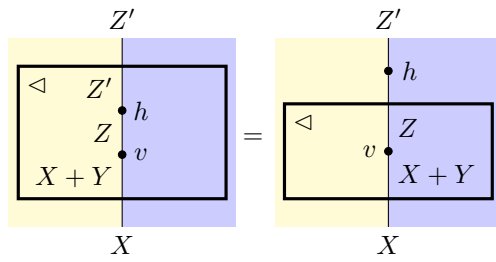
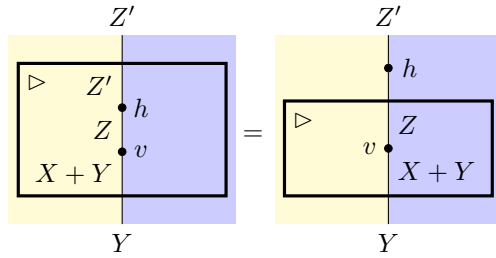
\begin{equation*}
\begin{gathered}
\begin{tikzpicture}[auto, x=0.5cm, y=0.5cm]
\diagbb{3.5cm}{2.5cm};
\path (cbb.south west) ++(0.5,1) coordinate (bl)
      (cbb.north east) ++(-0.5,-2) coordinate (tr);
\node[rectangle, fit=(bl)(tr)] (subdiag) {};
\begin{scope}
\path[clip] (subdiag.south west) rectangle (subdiag.north);
\node[rectangle, fit=(subdiag.south west)(subdiag.north)] (left) {};
\singlemor{left}{$f$}{$X$}{$Y$}{catterm}{catc};
\end{scope}
\begin{scope}
\path[clip] (subdiag.south) rectangle (subdiag.north east);
\node[rectangle, fit=(subdiag.south)(subdiag.north east)] (right) {};
\singlemor{right}{$g$}{$X$}{$Z$}{catterm}{catc};
\end{scope}
\draw[very thick] (subdiag.south west) rectangle (subdiag.north east);
\draw[very thick] (subdiag.south) -- (subdiag.north);
\draw (cbb.south) coordinate[label=below:$X + Y$] -- (subdiag.south);
\coordinate[dot, label=right:$h$] (h) at ($(subdiag.north)!0.5!(cbb.north)$);
\draw (subdiag.north) -- (h) -- (cbb.north) coordinate[label=above:$Z'$];
\begin{pgfonlayer}{background}
\fill[catterm] (cbb.north west) rectangle (cbb.south);
\fill[catc] (cbb.north east) rectangle (cbb.south);
\end{pgfonlayer}
\end{tikzpicture}
\end{gathered}
=
\begin{gathered}
\begin{tikzpicture}[auto, x=0.5cm, y=0.5cm]
\diagbb{3.5cm}{2.5cm};
\path (cbb.south west) ++(0.5,1) coordinate (bl)
      (cbb.north east) ++(-0.5,-1) coordinate (tr);
\node[rectangle, fit=(bl)(tr)] (subdiag) {};
\begin{scope}
\path[clip] (subdiag.south west) rectangle (subdiag.north);
\node[rectangle, fit=(subdiag.south west)(subdiag.north)] (left) {};
\doublemor{left}{$f$}{$h$}{$X$}{$Z$}{$Z'$}{catterm}{catc};
\end{scope}
\begin{scope}
\path[clip] (subdiag.south) rectangle (subdiag.north east);
\node[rectangle, fit=(subdiag.south)(subdiag.north east)] (right) {};
\doublemor{right}{$g$}{$h$}{$Y$}{$Z$}{$Z'$}{catterm}{catc};
\end{scope}
\draw[very thick] (subdiag.south west) rectangle (subdiag.north east);
\draw[very thick] (subdiag.south) -- (subdiag.north);
\draw (cbb.south) coordinate[label=below:$X + Y$] -- (subdiag.south);
\coordinate (h) at ($(subdiag.north)!0.5!(cbb.north)$);
\draw (subdiag.north) -- (cbb.north) coordinate[label=above:$Z'$];
\begin{pgfonlayer}{background}
\fill[catterm] (cbb.north west) rectangle (cbb.south);
\fill[catc] (cbb.north east) rectangle (cbb.south);
\end{pgfonlayer}
\end{tikzpicture}
\end{gathered}
\end{equation*}
\begin{equation*}
\begin{gathered}
\begin{tikzpicture}[auto, x=0.5cm, y=0.5cm]
\diagbb{3cm}{2.5cm};
\path (cbb.south west) ++(0.5,1) coordinate (bl)
      (cbb.north east) ++(-0.5,-1) coordinate (tr);
\node[rectangle, fit=(bl)(tr)] (subdiag) {};
\begin{scope}
\path[clip] (subdiag.south west) rectangle (subdiag.north east);
\doublemor{subdiag}{$v$}{$h$}{$X + Y$}{$Z$}{$Z'$}{catterm}{catc};
\path (subdiag.north west) ++(0.5,-0.5) node {$\lhd$};
\end{scope}
\draw[very thick] (subdiag.south west) rectangle (subdiag.north east);
\draw (cbb.south) coordinate[label=below:$X$] -- (subdiag.south)
      (subdiag.north) -- (cbb.north) coordinate[label=above:$Z'$];
\begin{pgfonlayer}{background}
\fill[catterm] (cbb.north west) rectangle (cbb.south);
\fill[catc] (cbb.north east) rectangle (cbb.south);
\end{pgfonlayer}
\end{tikzpicture}
\end{gathered}
=
\begin{gathered}
\begin{tikzpicture}[auto, x=0.5cm, y=0.5cm]
\diagbb{3cm}{2.5cm};
\path (cbb.south west) ++(0.5,1) coordinate (bl)
      (cbb.north east) ++(-0.5,-2) coordinate (tr);
\node[rectangle, fit=(bl)(tr)] (subdiag) {};
\begin{scope}
\path[clip] (subdiag.south west) rectangle (subdiag.north east);
\singlemor{subdiag}{$v$}{$X + Y$}{$Z$}{catterm}{catc};
\path (subdiag.north west) ++(0.5,-0.5) node {$\lhd$};
\end{scope}
\draw[very thick] (subdiag.south west) rectangle (subdiag.north east);
\coordinate[dot, label=right:$h$] (h) at ($(subdiag.north)!0.5!(cbb.north)$);
\draw (cbb.south) coordinate[label=below:$X$] -- (subdiag.south)
      (subdiag.north) -- (h) -- (cbb.north) coordinate[label=above:$Z'$];
\begin{pgfonlayer}{background}
\fill[catterm] (cbb.north west) rectangle (cbb.south);
\fill[catc] (cbb.north east) rectangle (cbb.south);
\end{pgfonlayer}
\end{tikzpicture}
\end{gathered}
\end{equation*}
\begin{equation*}
\begin{gathered}
\begin{tikzpicture}[auto, x=0.5cm, y=0.5cm]
\diagbb{3cm}{2.5cm};
\path (cbb.south west) ++(0.5,1) coordinate (bl)
      (cbb.north east) ++(-0.5,-1) coordinate (tr);
\node[rectangle, fit=(bl)(tr)] (subdiag) {};
\begin{scope}
\path[clip] (subdiag.south west) rectangle (subdiag.north east);
\doublemor{subdiag}{$v$}{$h$}{$X + Y$}{$Z$}{$Z'$}{catterm}{catc};
\path (subdiag.north west) ++(0.5,-0.5) node {$\rhd$};
\end{scope}
\draw[very thick] (subdiag.south west) rectangle (subdiag.north east);
\draw (cbb.south) coordinate[label=below:$Y$] -- (subdiag.south)
      (subdiag.north) -- (cbb.north) coordinate[label=above:$Z'$];
\begin{pgfonlayer}{background}
\fill[catterm] (cbb.north west) rectangle (cbb.south);
\fill[catc] (cbb.north east) rectangle (cbb.south);
\end{pgfonlayer}
\end{tikzpicture}
\end{gathered}
=
\begin{gathered}
\begin{tikzpicture}[auto, x=0.5cm, y=0.5cm]
\diagbb{3cm}{2.5cm};
\path (cbb.south west) ++(0.5,1) coordinate (bl)
      (cbb.north east) ++(-0.5,-2) coordinate (tr);
\node[rectangle, fit=(bl)(tr)] (subdiag) {};
\begin{scope}
\path[clip] (subdiag.south west) rectangle (subdiag.north east);
\singlemor{subdiag}{$v$}{$X + Y$}{$Z$}{catterm}{catc};
\path (subdiag.north west) ++(0.5,-0.5) node {$\rhd$};
\end{scope}
\draw[very thick] (subdiag.south west) rectangle (subdiag.north east);
\coordinate[dot, label=right:$h$] (h) at ($(subdiag.north)!0.5!(cbb.north)$);
\draw (cbb.south) coordinate[label=below:$Y$] -- (subdiag.south)
      (subdiag.north) -- (h) -- (cbb.north) coordinate[label=above:$Z'$];
\begin{pgfonlayer}{background}
\fill[catterm] (cbb.north west) rectangle (cbb.south);
\fill[catc] (cbb.north east) rectangle (cbb.south);
\end{pgfonlayer}
\end{tikzpicture}
\end{gathered}
\end{equation*}
That these maps witness a bijection leads to the following three identities:
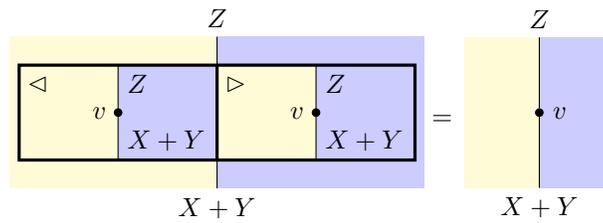
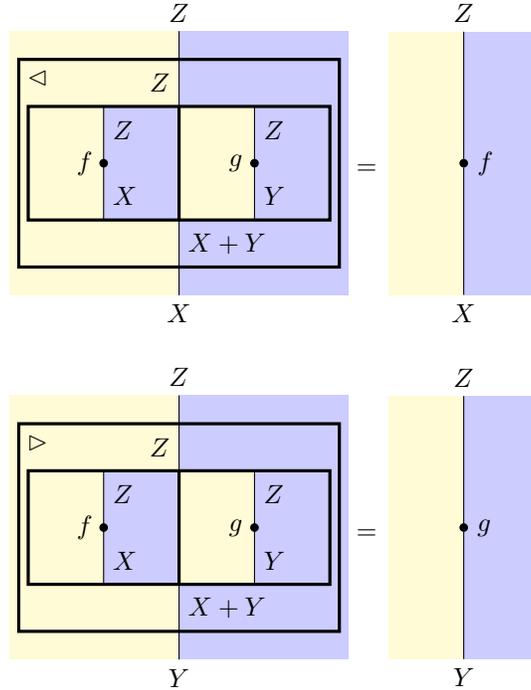
\begin{equation*}
\begin{gathered}
\begin{tikzpicture}[auto, x=0.5cm, y=0.5cm]
\diagbb{5.5cm}{2cm};
\path (cbb.south west) ++(0.5,1) coordinate (bl)
      (cbb.north east) ++(-0.5,-1) coordinate (tr);
\node[rectangle, fit=(bl)(tr)] (subdiag) {};
\begin{scope}
\path[clip] (subdiag.south west) rectangle (subdiag.north);
\node[rectangle, fit=(subdiag.south west)(subdiag.north)] (left) {};
\singlemor{left}{$v$}{$X + Y$}{$Z$}{catterm}{catc};
\end{scope}
\begin{scope}
\path[clip] (subdiag.south) rectangle (subdiag.north east);
\node[rectangle, fit=(subdiag.south)(subdiag.north east)] (right) {};
\singlemor{right}{$v$}{$X + Y$}{$Z$}{catterm}{catc};
\end{scope}
\draw[very thick] (subdiag.south west) rectangle (subdiag.north east);
\draw[very thick] (subdiag.south) -- (subdiag.north);
\draw (cbb.south) coordinate[label=below:$X + Y$] -- (subdiag.south);
\draw (subdiag.north) -- (cbb.north) coordinate[label=above:$Z$];
\path (left.north west) ++(0.75,-0.75) node {$\lhd$}
      (right.north west) ++(0.75,-0.75) node {$\rhd$};
\begin{pgfonlayer}{background}
\fill[catterm] (cbb.north west) rectangle (cbb.south);
\fill[catc] (cbb.north east) rectangle (cbb.south);
\end{pgfonlayer}
\end{tikzpicture}
\end{gathered}
=
\begin{gathered}
\twocelldiag{v}{X + Y}{Z}{catterm}{catc}{2}{2}{2}
\end{gathered}
\end{equation*}
\begin{equation*}
\begin{gathered}
\begin{tikzpicture}[auto, x=0.5cm, y=0.5cm]
\diagbb{4.5cm}{3.5cm};
\path (cbb.south west) ++(0.5,1) coordinate (bl')
      (cbb.north east) ++(-0.5,-1) coordinate (tr');
\node[rectangle, fit=(bl')(tr')] (subdiag') {};
\path (subdiag'.south west) ++(0.5,1.5) coordinate (bl)
      (subdiag'.north east) ++(-0.5,-1.5) coordinate (tr);
\node[rectangle, fit=(bl)(tr)] (subdiag) {};
\fill[catc] (subdiag'.south west) rectangle (subdiag'.north east);
\fill[catterm] (subdiag'.south west) rectangle (subdiag'.north);
\begin{scope}
\path[clip] (subdiag.south west) rectangle (subdiag.north);
\node[rectangle, fit=(subdiag.south west)(subdiag.north)] (left) {};
\singlemor{left}{$f$}{$X$}{$Z$}{catterm}{catc};
\end{scope}
\begin{scope}
\path[clip] (subdiag.south) rectangle (subdiag.north east);
\node[rectangle, fit=(subdiag.south)(subdiag.north east)] (right) {};
\singlemor{right}{$g$}{$Y$}{$Z$}{catterm}{catc};
\end{scope}
\draw[very thick] (subdiag.south west) rectangle (subdiag.north east);
\draw[very thick] (subdiag.south) -- (subdiag.north);
\draw (subdiag'.south) to node[swap]{$X + Y$} (subdiag.south);
\draw (subdiag.north) to node{$Z$} (subdiag'.north);
\draw[very thick] (subdiag'.south west) rectangle (subdiag'.north east);
\draw (cbb.south) coordinate[label=below:$X$] -- (subdiag'.south)
      (subdiag'.north) -- (cbb.north) coordinate[label=above:$Z$];
\path (subdiag'.north west) ++(0.5,-0.5) node{$\lhd$};
\begin{pgfonlayer}{background}
\fill[catterm] (cbb.north west) rectangle (cbb.south);
\fill[catc] (cbb.north east) rectangle (cbb.south);
\end{pgfonlayer}
\end{tikzpicture}
\end{gathered}
=
\begin{gathered}
\twocelldiag{f}{X}{Z}{catterm}{catc}{2}{2}{3.5}
\end{gathered}
\end{equation*}
\begin{equation*}
\begin{gathered}
\begin{tikzpicture}[auto, x=0.5cm, y=0.5cm]
\diagbb{4.5cm}{3.5cm};
\path (cbb.south west) ++(0.5,1) coordinate (bl')
      (cbb.north east) ++(-0.5,-1) coordinate (tr');
\node[rectangle, fit=(bl')(tr')] (subdiag') {};
\path (subdiag'.south west) ++(0.5,1.5) coordinate (bl)
      (subdiag'.north east) ++(-0.5,-1.5) coordinate (tr);
\node[rectangle, fit=(bl)(tr)] (subdiag) {};
\fill[catc] (subdiag'.south west) rectangle (subdiag'.north east);
\fill[catterm] (subdiag'.south west) rectangle (subdiag'.north);
\begin{scope}
\path[clip] (subdiag.south west) rectangle (subdiag.north);
\node[rectangle, fit=(subdiag.south west)(subdiag.north)] (left) {};
\singlemor{left}{$f$}{$X$}{$Z$}{catterm}{catc};
\end{scope}
\begin{scope}
\path[clip] (subdiag.south) rectangle (subdiag.north east);
\node[rectangle, fit=(subdiag.south)(subdiag.north east)] (right) {};
\singlemor{right}{$g$}{$Y$}{$Z$}{catterm}{catc};
\end{scope}
\draw[very thick] (subdiag.south west) rectangle (subdiag.north east);
\draw[very thick] (subdiag.south) -- (subdiag.north);
\draw (subdiag'.south) to node[swap]{$X + Y$} (subdiag.south);
\draw (subdiag.north) to node{$Z$} (subdiag'.north);
\draw[very thick] (subdiag'.south west) rectangle (subdiag'.north east);
\draw (cbb.south) coordinate[label=below:$Y$] -- (subdiag'.south)
      (subdiag'.north) -- (cbb.north) coordinate[label=above:$Z$];
\path (subdiag'.north west) ++(0.5,-0.5) node{$\rhd$};
\begin{pgfonlayer}{background}
\fill[catterm] (cbb.north west) rectangle (cbb.south);
\fill[catc] (cbb.north east) rectangle (cbb.south);
\end{pgfonlayer}
\end{tikzpicture}
\end{gathered}
=
\begin{gathered}
\twocelldiag{g}{Y}{Z}{catterm}{catc}{2}{2}{3.5}
\end{gathered}
\end{equation*}
\begin{remark}
Throughout this section there has been no attempt to present a minimal set
of equations. Many of the equalities above are interderivable, but we
have aimed to give explicit statements of many useful properties, and to
exploit results about of representable functors to prove standard identities
from general principles.
\end{remark}

\section{Bifunctors}
In earlier sections, for example the accounts of binary products and coproducts 
in sections \ref{sec:binprod} and \ref{sec:bincoprod},
we have implicitly been dealing with bifunctors, we now describe a general strategy
for handling bifunctors using string diagrams. Our approach can be seen as using ``sections''
in the terminology of functional programming.
Bifunctors are slightly awkward in the string diagrammatic framework as
ideally we would use horizontal juxtaposition to describe the parameters to the bifunctor,
in the style used in the calculi for monoidal categories.
Unfortunately we
have already used this graphical freedom to describe functor composition, if we wish
to continue working with 2-dimensional diagrams we must make some compromises.

Consider a bifunctor:
\begin{equation*}
T : \mathcal{C} \times \mathcal{D} \rightarrow \mathcal{E}
\end{equation*}
By fixing either the first or second parameter, each $\mathcal{C}$ object $C$ and $\mathcal{D}$ object $D$ induce (unary) functors:
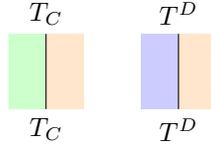
\begin{equation*}
\begin{gathered}
\onecelldiag{T_C}{catd}{cate}{2}{1}
\end{gathered}
\qquad
\begin{gathered}
\onecelldiag{T^D}{catc}{cate}{2}{1}
\end{gathered}
\end{equation*}
Also, $\mathcal{C}$ morphism $f : C \rightarrow C'$ and $\mathcal{D}$ morphism $g : D \rightarrow D'$ induce natural transformations:
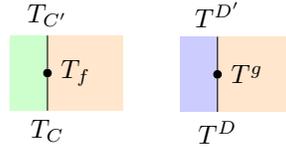
\begin{equation*}
\begin{gathered}
\twocelldiag{T_f}{T_C}{T_{C'}}{catd}{cate}{1}{2}{1}
\end{gathered}
\qquad
\begin{gathered}
\twocelldiag{T^g}{T^D}{T^{D'}}{catc}{cate}{1}{2}{1}
\end{gathered}
\end{equation*}
These satisfy the following obvious functoriality equations:
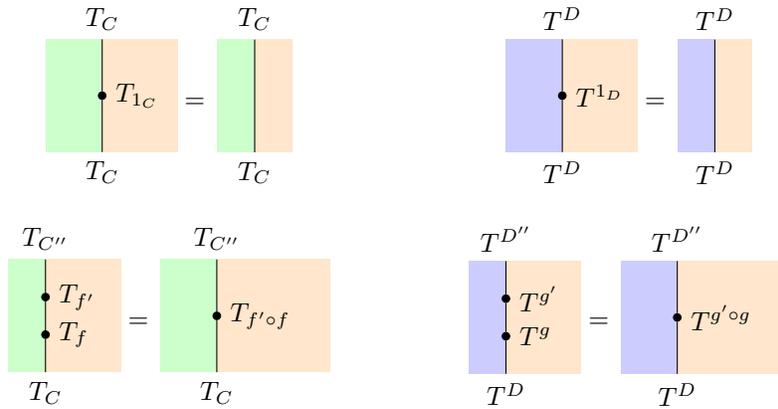
\begin{subequations}
\begin{trivlist}
\item
\begin{minipage}{0.495\textwidth}
\begin{equation*}
\begin{gathered}
\twocelldiag{T_{1_C}}{T_C}{T_C}{catd}{cate}{1.5}{2}{1.5}
\end{gathered}
=
\begin{gathered}
\onecelldiag{T_C}{catd}{cate}{3}{1}
\end{gathered}
\end{equation*}
\end{minipage}
\begin{minipage}{0.495\textwidth}
\begin{equation*}
\begin{gathered}
\twocelldiag{T^{1_D}}{T^D}{T^D}{catc}{cate}{1.5}{2}{1.5}
\end{gathered}
=
\begin{gathered}
\onecelldiag{T^D}{catc}{cate}{3}{1}
\end{gathered}
\end{equation*}
\end{minipage}
\item
\begin{minipage}{0.495\textwidth}
\begin{equation*}
\begin{gathered}
\twocellpairdiag{T_f}{T_{f'}}{T_C}{T_{C''}}{catd}{cate}{1}{2}{1}
\end{gathered}
=
\begin{gathered}
\twocelldiag{T_{f' \circ f}}{T_C}{T_{C''}}{catd}{cate}{1.5}{3}{1.5}
\end{gathered}
\end{equation*}
\end{minipage}
\begin{minipage}{0.495\textwidth}
\begin{equation*}
\begin{gathered}
\twocellpairdiag{T^g}{T^{g'}}{T^D}{T^{D''}}{catc}{cate}{1}{2}{1}
\end{gathered}
=
\begin{gathered}
\twocelldiag{T^{g' \circ g}}{T^D}{T^{D''}}{catc}{cate}{1.5}{3}{1.5}
\end{gathered}
\end{equation*}
\end{minipage}
\end{trivlist}
\end{subequations}
For bifunctors $S, T : \mathcal{C} \times \mathcal{D} \rightarrow \mathcal{E}$ we can
consider natural transformations $\alpha : S \Rightarrow T$,
typically $\alpha_{C,D}$ is then said to be ``natural in both $C$ and $D$''. Again we fix either the first
or second parameter, giving for each $\mathcal{C}$ object $C$ and each $\mathcal{D}$ object $D$ families of natural transformations:
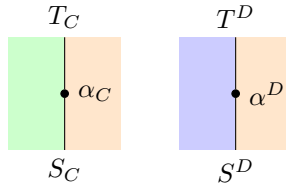
\begin{equation*}
\begin{gathered}
\twocelldiag{\alpha_C}{S_C}{T_C}{catd}{cate}{1.5}{1.5}{1.5}
\end{gathered}
\qquad
\begin{gathered}
\twocelldiag{\alpha^D}{S^D}{T^D}{catc}{cate}{1.5}{1.5}{1.5}
\end{gathered}
\end{equation*}
Naturality in the fixed parameters must then be handled explicitly as the satisfaction of the following equations:
\begin{subequations}
\begin{trivlist}
\item
\begin{minipage}{0.495\textwidth}
\begin{equation}
\label{eq:binat}
\begin{gathered}
\twocellpairdiag{\alpha_C}{T_f}{S_C}{T_{C'}}{catd}{cate}{1}{2}{1}
\end{gathered}
=
\begin{gathered}
\twocellpairdiag{S_f}{\alpha_{C'}}{S_C}{T_{C'}}{catd}{cate}{1}{2}{1}
\end{gathered}
\end{equation}
\end{minipage}
\begin{minipage}{0.495\textwidth}
\begin{equation}
\label{eq:binattwo}
\begin{gathered}
\twocellpairdiag{\alpha^D}{T^g}{S^D}{T^{D'}}{catc}{cate}{1}{2}{1}
\end{gathered}
=
\begin{gathered}
\twocellpairdiag{S^g}{\alpha^{D'}}{S^D}{T^{D'}}{catc}{cate}{1}{2}{1}
\end{gathered}
\end{equation}
\end{minipage}
\end{trivlist}
\end{subequations}
We can then work with transformations natural in two parameters by choosing one parameter
to fix, naturality in the other parameter then behaves in the usual way for string
diagrams, and naturality in the fixed parameter is captured equationally as
the commutativity conditions of equations \eqref{eq:binat} and \eqref{eq:binattwo}.
Some of the calculations are now effectively being performed in the subscripts and subscripts
in these diagrams, but we do retain a topological feel to our reasoning.

We can also relate the two choices of notation via the following equations:
\begin{trivlist}
\item
\begin{minipage}{0.495\linewidth}
\begin{equation}
\label{eq:bifunrela}
\begin{gathered}
\begin{tikzpicture}[scale=0.5]
\path coordinate[dot, label=left:$g$] (g) +(0,1) coordinate[label=above:$D'$] (tl) +(0,-1) coordinate[label=below:$D$] (bl)
 ++(1.5,0) coordinate[dot, label=right:$T_f$] (Tf) +(0,1) coordinate[label=above:$T_{C'}$] (tr) +(0,-1) coordinate[label=below:$T_C$] (br);
\draw (bl) -- (g) -- (tl)
 (br) -- (Tf) -- (tr);
\begin{pgfonlayer}{background}
\fill[catterm] ($(tl) + (-1.0, 0)$) rectangle (bl);
\fill[catd] (tl) rectangle (br);
\fill[cate] (tr) rectangle ($(br) + (1.5,0)$);
\end{pgfonlayer}
\end{tikzpicture}
\end{gathered}
=
\begin{gathered}
\begin{tikzpicture}[scale=0.5]
\path coordinate[dot, label=left:$f$] (g) +(0,1) coordinate[label=above:$C'$] (tl) +(0,-1) coordinate[label=below:$C$] (bl)
 ++(1.5,0) coordinate[dot, label=right:$T^g$] (Tf) +(0,1) coordinate[label=above:$T^{D'}$] (tr) +(0,-1) coordinate[label=below:$T^D$] (br);
\draw (bl) -- (g) -- (tl)
 (br) -- (Tf) -- (tr);
\begin{pgfonlayer}{background}
\fill[catterm] ($(tl) + (-1.0, 0)$) rectangle (bl);
\fill[catc] (tl) rectangle (br);
\fill[cate] (tr) rectangle ($(br) + (1.5,0)$);
\end{pgfonlayer}
\end{tikzpicture}
\end{gathered}
\end{equation}
\end{minipage}
\begin{minipage}{0.495\linewidth}
\begin{equation}
\label{eq:bifunrelb}
\begin{gathered}
\begin{tikzpicture}[scale=0.5]
\path +(0,1) coordinate[label=above:$D$] (tl) +(0,-1) coordinate[label=below:$D$] (bl)
 ++(1.5,0) coordinate[dot, label=right:$\alpha_C$] (Tf) +(0,1) coordinate[label=above:$T_{C}$] (tr) +(0,-1) coordinate[label=below:$S_C$] (br);
\draw (bl) -- (tl)
 (br) -- (Tf) -- (tr);
\begin{pgfonlayer}{background}
\fill[catterm] ($(tl) + (-1, 0)$) rectangle (bl);
\fill[catd] (tl) rectangle (br);
\fill[cate] (tr) rectangle ($(br) + (1.5,0)$);
\end{pgfonlayer}
\end{tikzpicture}
\end{gathered}
=
\begin{gathered}
\begin{tikzpicture}[scale=0.5]
\path +(0,1) coordinate[label=above:$C$] (tl) +(0,-1) coordinate[label=below:$C$] (bl)
 ++(1.5,0) coordinate[dot, label=right:$\alpha^D$] (Tf) +(0,1) coordinate[label=above:$T^D$] (tr) +(0,-1) coordinate[label=below:$S^D$] (br);
\draw (bl) -- (tl)
 (br) -- (Tf) -- (tr);
\begin{pgfonlayer}{background}
\fill[catterm] ($(tl) + (-1, 0)$) rectangle (bl);
\fill[catc] (tl) rectangle (br);
\fill[cate] (tr) rectangle ($(br) + (1.5,0)$);
\end{pgfonlayer}
\end{tikzpicture}
\end{gathered}
\end{equation}
\end{minipage}
\end{trivlist}
The relationships in equations \ref{eq:bifunrela} and \ref{eq:bifunrelb} illustrate a new phenomenon whereby we have equalities
in which the functors at the top and bottom of the diagrams on each side are apparently ``different'', 
but the composites are actually equal, for example in this case by definition we have equalities $T_X Y = T(X,Y) = T^Y X$. 
These types of equations
between composite functors are an occasion on which it can be useful to explicitly insert identity natural transformations, in
order to witness the equalities, for example, we can rewrite equation \ref{eq:bifunrela} in a topologically more instructive
form as:
\begin{equation*}
\begin{gathered}
\begin{tikzpicture}[scale=0.5]
\path coordinate[dot, label=above:$1$] (eq) +(-1,1) coordinate[label=above:$C'$] (tl) +(1,1) coordinate[label=above:$T^{D'}$] (tr)
 (eq) ++(-1,-1) coordinate[dot, label=left:$g$] (g) ++(0,-1) coordinate[label=below:$D$] (bl)
 (eq) ++(1,-1) coordinate[dot,label=right:$T_f$] (Tf) ++(0,-1) coordinate[label=below:$T_C$] (br);
\draw (bl) -- (g) to[out=90, in=180] (eq.west) -- (eq.west) to[out=180, in=-90] (tl)
 (br) -- (Tf) to[out=90, in=0] (eq.east) -- (eq.east) to[out=0, in=-90] (tr);
\begin{pgfonlayer}{background}
\fill[catterm] ($(tl) + (-1.5,0)$) rectangle ($(br) + (1.5,0)$);
\fill[catd] (bl) -- (g.south) -- (g.north) to[out=90, in=-180] (eq.west) -- (eq.east) to[out=0, in=90] (Tf) -- (br) -- cycle;
\fill[catc] (tl) to[out=-90, in=180] (eq.west) -- (eq.east) to[out=0, in=-90] (tr) -- cycle;
\fill[cate] (br) -- (Tf.south) -- (Tf.north) to[out=90, in=0] (eq.east) to[out=0, in=-90] (tr) -- ($(tr) + (1.5,0)$) -- ($(br) + (1.5,0)$) -- cycle;
\end{pgfonlayer}
\end{tikzpicture}
\end{gathered}
=
\begin{gathered}
\begin{tikzpicture}[scale=0.5]
\path coordinate[dot, label=below:$1$] (eq) +(-1,-1) coordinate[label=below:$D$] (bl) +(1,-1) coordinate[label=below:$T_C$] (br)
 (eq) ++(-1,1) coordinate[dot, label=left:$f$] (f) ++(0,1) coordinate[label=above:$C'$] (tl)
 (eq) ++(1,1) coordinate[dot, label=right:$T^g$] (Tg) ++(0,1) coordinate[label=above:$T^{D'}$] (tr);
\draw (bl) to[out=90, in=180] (eq.west) -- (eq.west) to[out=180, in=-90] (f) -- (tl)
 (br) to[out=90, in=0] (eq.east) -- (eq.east) to[out=0, in=-90] (Tg) -- (tr);
\begin{pgfonlayer}{background}
\fill[catterm] ($(tl) + (-1.5,0)$) rectangle ($(br) + (1.5,0)$);
\fill[catd] (bl) to[out=90, in=180] (eq.west) -- (eq.east) to[out=0, in=90] (br) -- cycle;
\fill[catc] (tl) -- (f.north) -- (f.south) to[out=-90, in=180] (eq.west) -- (eq.east) to[out=0, in=-90] (Tg) -- (tr) -- cycle;
\fill[cate] (br) to[out=90, in=0] (eq.east) to[out=0, in=-90] (Tg) -- (tr) -- ($(tr) + (1.5,0)$) -- ($(br) + (1.5,0)$) -- cycle;
\end{pgfonlayer}
\end{tikzpicture}
\end{gathered}
\end{equation*}
These \define{witnessing identity morphisms} then appear as a trivial form of distributive law, 
smoothing diagrammatic calculations by allowing us to switch between two equal composite functors.

\section{Conclusion}
We have shown in the previous sections how string diagrams can be used
to perform calculational proofs whilst retaining the vital type information.
Although we could not hope to provide comprehensive coverage, for example
we have not discussed dual adjunctions, equivalences, exponentials or 
many common types of limits and colimits,
hopefully we have provided sufficient background to make extending the graphical
proof style to further topics straightforward.

Tool support for developing and documenting proofs in this diagrammatic style
would be tremendously useful, and is a direction for further investigation.
The proofs we develop, although conceptually compact, often physically take up
a large amount of page space and the proof style also makes essential use of colour.
These attributes are not particularly ``publication friendly'', tool support could
also aim to provide (preferably bidirectional) translations between string based 
proofs and more traditional formats.

Although we do not claim that string diagrammatic proofs are always the most suitable
approach to any category theoretic question, many problems where there is a lot
of ``bookkeeping'' involving functoriality and naturality conditions can be greatly
simplified. String diagrams also provide a different, more visual and topological
intuition for category theoretic concepts and proofs. 
It is therefore hard to explain the absence of string diagrammatic tools
from introductory courses and texts on category theory and the relegation of this material
to folklore in the community. Teaching experience 
introducing these tools earlier in category theory courses would be also therefore be of interest.

\subsection*{Acknowledgements}
The author would like to thank Samson Abramsky, Andreas D\"oring, Ralf Hinze, Peter Selinger, 
Pawel Sobocinski, Eduardo Dubuc and Brendan Fong for useful feedback and encouragement.

\bibliography{papers}

\begin{thebibliography}{53}
\providecommand{\natexlab}[1]{#1}
\providecommand{\url}[1]{\texttt{#1}}
\expandafter\ifx\csname urlstyle\endcsname\relax
  \providecommand{\doi}[1]{doi: #1}\else
  \providecommand{\doi}{doi: \begingroup \urlstyle{rm}\Url}\fi

\bibitem[Abramsky and Coecke(2008)]{AbramskyCoecke2008}
S.~Abramsky and B.~Coecke.
\newblock Categorical quantum mechanics.
\newblock In K.~Engesser, D.~M. Gabbay, and D.~Lehmann, editors, \emph{Handbook
  of Quantum Logic and Quantum Structures}. Elsevier, 2008.

\bibitem[Abramsky and Tzevelekos(2011)]{AbrTze2011}
S.~Abramsky and N.~Tzevelekos.
\newblock Introduction to categories and categorical logic.
\newblock In B.~Coecke, editor, \emph{New Structures for Physics}, volume 813
  of \emph{Lecture Notes in Physics}, pages 3--94. Springer, 2011.

\bibitem[Ad\'{a}mek et~al.(2009)Ad\'{a}mek, Herrlich, and Strecker]{ACC2009}
J.~Ad\'{a}mek, H.~Herrlich, and G.~E. Strecker.
\newblock \emph{Abstract and Concrete Categories}.
\newblock Dover, 2009.

\bibitem[Backhouse(1989)]{Backhouse1989}
R.~C. Backhouse.
\newblock Make formality work for us.
\newblock \emph{EATCS Bulletin}, 38:\penalty0 219--249, 1989.

\bibitem[Baez and Stay(2011)]{BaezStay2011}
J.~C. Baez and M.~Stay.
\newblock Physics, topology, logic and computation: A {R}osetta stone.
\newblock In B.~Coecke, editor, \emph{New Structures for Physics}, volume 813
  of \emph{Lecture Notes in Physics}, pages 95--174. Springer, 2011.

\bibitem[Barr and Wells(2005)]{BarrWells2005}
M.~Barr and C.~Wells.
\newblock Toposes, triples and theories.
\newblock \emph{Reprints in Theory and Applications of Categories}, 1:\penalty0
  1--289, 2005.

\bibitem[Beck(1969)]{Beck1969}
J.~Beck.
\newblock Distributive laws.
\newblock In H.~Appelgate, M.~Barr, J.~Beck, F.~W. Lawvere, F.~Linton,
  E.~Manes, M.~Tierney, and F.~Ulmer, editors, \emph{Seminar on Triples and
  Categorical Homotopy Theory}, volume~80 of \emph{Lecture Notes in
  Mathematics}, pages 119--140. Springer, 1969.

\bibitem[B\'enabou(1967)]{Benabou1967}
J.~B\'enabou.
\newblock Introduction to bicategories.
\newblock In A.~Dold and B.~Eckmann, editors, \emph{Reports of the Midwest
  Category Seminar}, volume~47 of \emph{Lecture Notes in Mathematics}, pages
  1--77. Springer, 1967.

\bibitem[Blute et~al.(2002)Blute, Cockett, and Seely]{BluteCockettSeely2002}
R.~Blute, J.~R.~B. Cockett, and R.~A.~G. Seely.
\newblock The logic of linear functors.
\newblock \emph{Mathematical Structures in Computer Science}, 12\penalty0
  (4):\penalty0 513--539, 2002.

\bibitem[Borceux(1994)]{Borceux1994a}
F.~Borceux.
\newblock \emph{Handbook of Categorical Algebra}, volume~1.
\newblock Cambridge University Press, 1994.

\bibitem[Cockett and Seely(1999)]{CockettSeely1999}
J.~R.~B. Cockett and R.~A.~G. Seely.
\newblock Linearly distributive functors.
\newblock \emph{The Barrfestschrift, Journal of Pure and Applied Algebra},
  143:\penalty0 133--173, 1999.

\bibitem[Coecke and Duncan(2011)]{CoeckeDuncan2011}
B.~Coecke and R.~Duncan.
\newblock Interacting quantum observables: categorical algebra and
  diagrammatics.
\newblock \emph{New Journal of Physics}, 13\penalty0 (4):\penalty0 043016,
  2011.

\bibitem[Dijkstra and Scholten(1990)]{DijkstraScholten1990}
E.~W. Dijkstra and C.~S. Scholten.
\newblock \emph{Predicate calculus and program semantics}.
\newblock Texts and monographs in computer science. Springer, 1990.

\bibitem[Dubuc and Szyld(2013)]{Dubuc2013}
E.~J. Dubuc and M.~Szyld.
\newblock A {T}annakian context for {G}alois theory.
\newblock \emph{Advances in Mathematics}, 234:\penalty0 528--549, 2013.

\bibitem[Ehresmann(1963)]{Ehresmann1963}
C.~Ehresmann.
\newblock Cat\'egories structur\'ees.
\newblock \emph{Ann. Sci. Ecole Norm. Sup.}, 80:\penalty0 349--425, 1963.

\bibitem[Eilenberg and Moore(1965)]{EilenbergMoore1965}
S.~Eilenberg and J.~C. Moore.
\newblock Adjoint functors and triples.
\newblock \emph{Illinois J. Math.}, 9:\penalty0 381--398, 1965.

\bibitem[Fokkinga(1992{\natexlab{a}})]{Fokkinga1992}
M.~M. Fokkinga.
\newblock \emph{{A Gentle Introduction to Category Theory --- the calculational
  approach}}, volume Part I, pages 1--72.
\newblock University of Utrecht, Utrecht, Netherlands, Sep 1992{\natexlab{a}}.

\bibitem[Fokkinga(1992{\natexlab{b}})]{Fokkinga1992b}
M.~M. Fokkinga.
\newblock Calculate categorically!
\newblock \emph{Formal Aspects of Computing}, 4\penalty0 (4):\penalty0
  673--692, 1992{\natexlab{b}}.

\bibitem[Fokkinga and Meertens(1994)]{FokkingaMeertens1994}
M.~M. Fokkinga and L.~Meertens.
\newblock {Adjunctions}.
\newblock Technical Report Memoranda Inf 94-31, University of Twente, Enschede,
  Netherlands, Jun 1994.

\bibitem[Goguen et~al.(1975)Goguen, Thatcher, Wagner, and
  Wright]{GoguenThatcherWagnerWright1975}
J.~A. Goguen, J.~W. Thatcher, E.~G. Wagner, and J.~B. Wright.
\newblock An introduction to categories, algebraic theories and algebras.
\newblock Technical report, IBM Thomas J. Watson Research Centre, Yorktown
  Heights, 1975.

\bibitem[Goguen et~al.(1977)Goguen, Thatcher, Wagner, and
  Wright]{GoguenThatcherWagnerWright1977}
J.~A. Goguen, J.~W. Thatcher, E.~G. Wagner, and J.~B. Wright.
\newblock Initial algebra semantics and continuous algebras.
\newblock \emph{Journal of the ACM}, 24\penalty0 (1):\penalty0 68--95, 1977.

\bibitem[Gray(1974)]{Day1974}
J.~W. Gray.
\newblock \emph{Formal Category Theory: Adjointness for 2-Categories}, volume
  391 of \emph{Lecture Notes in Mathematics}.
\newblock Springer, 1974.

\bibitem[Gries and Schneider(1993)]{GriesSchneider1993}
D.~Gries and F.~B. Schneider.
\newblock \emph{A Logical Approach to Discrete Math}.
\newblock Springer, 1993.

\bibitem[Hagino(1987{\natexlab{a}})]{Hagino1987}
T.~Hagino.
\newblock \emph{A Categorical Programming Language}.
\newblock PhD thesis, University of Edinburgh, 1987{\natexlab{a}}.

\bibitem[Hagino(1987{\natexlab{b}})]{Hagino1987b}
T.~Hagino.
\newblock A typed lambda calculus with categorical type constructors.
\newblock In D.~H. Pitt, A.~Poign{\'e}, and D.~E. Rydeheard, editors,
  \emph{Category Theory and Computer Science}, volume 283 of \emph{Lecture
  Notes in Computer Science}, pages 140--157. Springer, 1987{\natexlab{b}}.

\bibitem[Hinze(2012)]{Hinze2012}
R.~Hinze.
\newblock Kan extensions for program optimisation or: Art and {D}an explain an
  old trick.
\newblock In J.~Gibbons and P.~Nogueira, editors, \emph{MPC}, volume 7342 of
  \emph{Lecture Notes in Computer Science}, pages 324--362. Springer, 2012.

\bibitem[Jacobs(2012)]{Jacobs2012b}
B.~Jacobs.
\newblock Introduction to coalgebra, towards mathematics of states and
  observation.
\newblock Book in preparation, 2012.

\bibitem[Joyal and Street(1988)]{JoyalStreet1988}
A.~Joyal and R.~Street.
\newblock Planar diagrams and tensor algebra, 1988.

\bibitem[Joyal and Street(1991)]{JoyalStreet1991}
A.~Joyal and R.~Street.
\newblock Geometry of tensor calculus {I}.
\newblock \emph{Advances in Mathematics}, 88:\penalty0 55--113, 1991.

\bibitem[Joyal and Street(1995)]{JoyalStreet1995}
A.~Joyal and R.~Street.
\newblock Geometry of tensor calculus {II}.
\newblock unpublished draft, 1995.

\bibitem[Kelly and Street(1974)]{KellyStreet1974}
G.~M. Kelly and R.~Street.
\newblock Review of the elements of 2-categories.
\newblock In \emph{Sydney Category Seminar}, volume 420 of \emph{Lecture Notes
  in Mathematics}, pages 75--103. Springer, 1974.

\bibitem[Kelly and Laplaza(1980)]{KellyLaplaza1980}
M.~Kelly and M.~L. Laplaza.
\newblock Coherence for compact closed categories.
\newblock \emph{Journal of Pure and Applied Algebra}, 19:\penalty0 193--213,
  1980.

\bibitem[Kissinger et~al.(2014)Kissinger, Merry, Frot, Coecke, Quick, Dixon,
  Soloviev, Duncan, and Zamdzhiev]{Quantomatic}
A.~Kissinger, A.~Merry, B.~Frot, B.~Coecke, D.~Quick, L.~Dixon, M.~Soloviev,
  R.~Duncan, and V.~Zamdzhiev.
\newblock The {Q}uantomatic tool website.
\newblock https://sites.google.com/site/quantomatic/, 2014.

\bibitem[Kleisli(1965)]{Kleisli1965}
H.~Kleisli.
\newblock Every standard construction is induced by a pair of functors.
\newblock \emph{Proc. Am. Math. Soc.}, 16:\penalty0 544--546, 1965.

\bibitem[MacLane(1998)]{MacLane1998}
S.~MacLane.
\newblock \emph{Categories for the Working Mathematician}, volume~5 of
  \emph{Graduate Texts in Mathematics}.
\newblock Springer, 1998.

\bibitem[Malcolm(1990{\natexlab{a}})]{Malcolm1990}
G.~Malcolm.
\newblock \emph{Algebraic Data Types and Program Transformation}.
\newblock PhD thesis, Rijksuniversiteit Groningen, 1990{\natexlab{a}}.

\bibitem[Malcolm(1990{\natexlab{b}})]{Malcolm1990b}
G.~Malcolm.
\newblock Data structures and program transformation.
\newblock \emph{Sci. Comput. Program.}, 14\penalty0 (2-3):\penalty0 255--279,
  1990{\natexlab{b}}.

\bibitem[Manes(1976)]{Manes1976}
E.~G. Manes.
\newblock \emph{Algebraic Theories}, volume~26 of \emph{Graduate Texts in
  Mathematics}.
\newblock Springer, 1976.

\bibitem[Melli{\`e}s(2006)]{Mellies2006}
P.-A. Melli{\`e}s.
\newblock Functorial boxes in string diagrams.
\newblock In Zolt{\'a}n {\'E}sik, editor, \emph{CSL}, volume 4207 of
  \emph{Lecture Notes in Computer Science}, pages 1--30. Springer, 2006.

\bibitem[Melli{\`e}s(2012)]{Mellies2012}
P.-A. Melli{\`e}s.
\newblock Game semantics in string diagrams.
\newblock In \emph{{Proceedings of the 27th Annual IEEE Symposium on Logic in
  Computer Science, LICS 2012, Dubrovnik, Croatia, June 25-28, 2012}}, pages
  481--490. IEEE, 2012.

\bibitem[Moggi(1991)]{Moggi1991}
E.~Moggi.
\newblock Notions of computation and monads.
\newblock \emph{Inf. Comput.}, 93\penalty0 (1):\penalty0 55--92, 1991.

\bibitem[Palmquist(1971)]{Palmquist1971}
P.~H. Palmquist.
\newblock The double category of adjoint squares.
\newblock In J.~W. Gray and S.~MacLane, editors, \emph{Reports of the Midwest
  Category Seminar {V}}, volume 195 of \emph{Lecture Notes in Mathematics}.
  Springer, 1971.

\bibitem[Pierce(1991)]{Pierce1991}
B.~C. Pierce.
\newblock \emph{Basic Category Theory for Computer Scientists}.
\newblock MIT Press, 1991.

\bibitem[Power and Watanabe(2002)]{PowerWatanabe2002}
J.~Power and H.~Watanabe.
\newblock Combining a monad and a comonad.
\newblock \emph{Theor. Comput. Sci.}, 280\penalty0 (1-2):\penalty0 137--162,
  2002.

\bibitem[Rutten(2000)]{Rutten2000}
J.~J. M.~M. Rutten.
\newblock Universal coalgebra: a theory of systems.
\newblock \emph{Theor. Comput. Sci.}, 249\penalty0 (1):\penalty0 3--80, 2000.

\bibitem[Selinger(2011)]{Selinger2011}
P.~Selinger.
\newblock A survey of graphical languages for monoidal categories.
\newblock In B.~Coecke, editor, \emph{New Structures for Physics}, volume 813
  of \emph{Lecture Notes in Physics}, pages 289--355. Springer, 2011.

\bibitem[Stay and Vicary(2013)]{StayVicary2013}
M.~Stay and J.~Vicary.
\newblock Bicategorical semantics for nondeterministic computation.
\newblock \emph{CoRR}, abs/1301.3393, 2013.

\bibitem[Street(1972)]{Street1972}
R.~Street.
\newblock The formal theory of monads.
\newblock \emph{J. Pure Appl. Algebra}, 2:\penalty0 149--168, 1972.

\bibitem[Street(1995)]{Street1995}
R.~Street.
\newblock Low-dimensional topology and higher-order categories.
\newblock In \emph{Proceedings of CT95}, 1995.

\bibitem[Street and Lack(2002)]{StreetLack2002}
R.~Street and S.~Lack.
\newblock The formal theory of monads {II}.
\newblock \emph{J. Pure Appl. Algebra}, 175:\penalty0 243--265, 2002.

\bibitem[van Gasteren(1990)]{vanGasteren1990}
A.~J.~M. van Gasteren.
\newblock \emph{On the Shape of Mathematical Arguments}, volume 445 of
  \emph{Lecture Notes in Computer Science}.
\newblock Springer, 1990.

\bibitem[Vicary(2012)]{Vicary2012}
J.~Vicary.
\newblock Higher semantics of quantum protocols.
\newblock In \emph{Proceedings of the 27th Annual IEEE Symposium on Logic in
  Computer Science, LICS 2012, Dubrovnik, Croatia, June 25-28, 2012}, pages
  606--615. IEEE, 2012.

\bibitem[Wadler(1995)]{Wadler1995}
P.~Wadler.
\newblock Monads for functional programming.
\newblock In J.~Jeuring and E.~Meijer, editors, \emph{Advanced Functional
  Programming}, volume 925 of \emph{Lecture Notes in Computer Science}, pages
  24--52. Springer, 1995.

\end{thebibliography}
\bibliographystyle{plainnat}

\end{document}